\documentclass[a4paper,fleqn]{elsarticle}

\usepackage[numbers]{natbib}

\usepackage{amsthm}
\usepackage{amsmath}
\usepackage{amssymb}
\usepackage{verbatim}
\usepackage{xcolor}
\usepackage{cleveref}
\usepackage{subfig}
\usepackage{float}
\usepackage{color}
\usepackage[printonlyused]{acronym}

\crefalias{subequation}{equation} 
\crefalias{eqnarray}{equation} 
\crefformat{pluraleq}{(#2#1#3)}

\usepackage{bm}
\usepackage[ruled]{algorithm2e}

\newcommand{\be}{\boldsymbol{e}}

\newcommand{\bu}{\boldsymbol{u}}

\newcommand{\bx}{\boldsymbol{x}}

\newcommand{\bO}{\boldsymbol{O}}

\newtheorem{theorem}{Theorem}
\newtheorem{assumption}{Assumption}
\newtheorem{remark}{Remark}
\newtheorem{definition}{Definition}
\newtheorem{corollary}{Corollary}
\newtheorem{lemma}{Lemma}

\def\tsc#1{\csdef{#1}{\textsc{\lowercase{#1}}\xspace}}
\tsc{WGM}
\tsc{QE}
\tsc{EP}
\tsc{PMS}
\tsc{BEC}
\tsc{DE}

\begin{document}
\let\WriteBookmarks\relax
\def\floatpagepagefraction{1}
\def\textpagefraction{.001}
\shorttitle{A number-theoretic sampling method in NN}
\shortauthors{Yu Yang et~al.}

\title [mode = title]{\textcolor{black}{A novel number-theoretic sampling method for neural network solutions of partial differential equations}}                      
\tnotemark[1]

\tnotetext[1]{This research is partially sponsored by the National key R \& D Program of China (No.2022YFE03040002) and the National Natural Science Foundation of China (No.11971020, No.12371434). }


        \author[1]{Yu Yang}
        [type=editor,
        auid=000,bioid=1,
        orcid=0009-0005-1428-6745
        ]
        \ead{yangyu1@stu.scu.edu.cn}

        \address[1]{School of Mathematics, Sichuan University, 610065, Chengdu, China.}

                                
        
        
        


        \author[2,3]{Pingan He}[style=chinese]
                                
        
        \ead{t330202702@mail.uic.edu.cn}
        \address[2]{Faculty of Science and Technology, BNU-HKBU United International College, 519087, Zhuhai, China.}
        
        \author[3]{Xiaoling Peng}[style=chinese]
                                
         \cormark[1]
        \ead{xlpeng@uic.edu.cn}
        \address[3]{Guangdong Provincial Key Laboratory of Interdisciplinary Research and Application for Data Science, BNU-HKBU United International College, 519087, Zhuhai, China}
        
        \author[1]{Qiaolin He}[style=chinese]
                                
        \cormark[1] 
        
        \ead{qlhejenny@scu.edu.cn}
        
        

        \cortext[cor1]{Corresponding author}

\begin{abstract}
\\
\textcolor{black}{Traditional Monte Carlo integration using uniform random sampling exhibits degraded efficiency in low-regularity or high-dimensional problems. We propose a novel deep learning framework based on deterministic number-theoretic sampling points, which is a robust approach specifically designed to handle partial differential equations with rough solutions or in high dimensions.} \textcolor{black}{The sample points are generated by the generating vector to  achieve the smallest discrepancy.}
\textcolor{black}{
The architecture integrates Physics-Informed Neural Networks (PINNs) with rigorous mathematical guarantees demonstrating lower error bounds compared to conventional  uniform random sampling. Numerical validation includes low-regularity Poisson equations, two-dimensional inverse Helmholtz problems, and high-dimensional linear/nonlinear PDEs, systematically demonstrating the algorithm's superior performance and generalization capabilities.}
\end{abstract}







\maketitle


 	\section{Introduction}
	\label{sec:intro}

   Numerically solving partial differential equations (PDEs) \cite{renardy2006introduction} has always been one of the concerns of frontier research. In recent years, as the computational efficiency of hardware continues to improve, the approach of utilizing deep learning to solve PDEs has increasingly gained attention. 
   \textcolor{black}{Starting from the strong form,
   the neural networks with enhanced physical knowledge and the deep Galerkin method were introduced in references \cite{PINN} and \cite{DGM}, respectively. Starting from the weak form, the Deep Ritz method, which is grounded in the variational formulation of PDEs, was developed in \cite{DeepRitz}. And Weak Adversarial Networks was proposed in \cite{WAN} to transform the weak solution problem of PDEs into a minimax problem, which is solved by alternatively updating the primal and adversarial networks.}  The deep backward stochastic differential equations (BSDEs) framework, adept at addressing high-dimensional PDEs, was established in \cite{DeepBSDE} based on stochastic differential equations(SDEs). From the perspective of operator learning, deep operator networks (DeepONet) were proposed in \cite{DeepONet}, and the Fourier Neural Operator (FNO) framework was introduced in \cite{FNO}.
   In this work, we are concerned with Physics-Informed Neural Networks (PINNs)\cite{PINN}.

   For the general PINN, its network structure is a common fully connected neural network, the input is the sampling points in the computational domain, and the output is the predicted solution of PDEs. Its loss function is usually made up of the norm $L_2$ of the residual function of PDEs, 
   the boundary loss 
   and more relevant physical constraints. 
   The neural network is then trained by an optimization method to approximate the solution of PDEs. 
    At present, it is a widely accepted fact that the error of PINNs is composed of three components \cite{shin2020convergence} which are approximation error, estimation error and optimization error, 
    where the estimation error is due to the Monte Carlo (MC) method to discretize the loss function.

  

    \textcolor{black}{The main difficulty is that when using MC method to discretize the loss function in integral form, sampling points obeying a uniform distribution are usually sampled to serve as empirical loss, which may cause large error for the problem with low-regularity or high-dimensional problem.} Then it is worth considering how to get better sampling.  Currently, there have been many works on non-uniform adaptive sampling. In the work \cite{lu2021deepxde}, the residual-based adaptive refinement (RAR) method aimed at enhancing the training efficiency of PINNs is introduced. Specifically, this method adaptively adds training points in regions where the residual loss is large. In \cite{wu2023comprehensive}, a novel residual-based adaptive distribution (RAD) method is proposed. The main concept of this method involves constructing a probability density function that relies on residual information and subsequently employing adaptive sampling techniques in accordance with this density function. 
    \textcolor{black}{In \cite{FIPINN}, the failure probability was introduced and utilized as a posteriori error indicators to generate new training points. In \cite{FIPINN2}, further advanced this work by incorporating re-sampling techniques and proposing the Subset Simulation algorithm for error estimation. Based on the failure probability, sampling points can be efficiently generated in the failure region.  An adaptive sampling framework for solving PDEs is proposed in \cite{das-pinns},
    which formulating a variance of the residual loss and leveraging KRNet, as introduced in \cite{tang2020deep}.} Contrary to the methods mentioned above that rely on residual loss, an MMPDE-Net grounded in the moving mesh method in \cite{MSPINN} is proposed, which cleverly utilizes the information from the solution to create a monitor function, facilitating adaptive sampling of training points. 
    \textcolor{black}{However, non-uniform sampling methods still have drawbacks. First, adaptive methods generally require an initial solution from a neural network, which means pre-training is necessary and consumes resources. Second, the degree of non-uniformity in sampling points is often determined by hyperparameters in the adaptive method. It is challenging to find a common optimal set of hyperparameters that works across different problems.}

      Compared to non-uniform adaptive sampling, uniform non-adaptive sampling does not require a pre-training step to obtain information about the loss function or the solution, which directly improves efficiency. There are several common methods of uniform non-adaptive sampling, which include 1) equispaced uniform points sampling, a method that uniformly selects points from an array of evenly spaced points, frequently employed in addressing low-dimensional problems; 2) Latin hypercube sampling (LHS, \cite{stein1987large}), which is a stratified sampling method based on multivariate parameter distributions; 3) uniform random sampling, which draws points at random in accordance with a continuous uniform distribution within the computational domain, representing the most widely adopted method (\cite{jin2021nsfnets,krishnapriyan2021characterizing,PINN}). However, equispaced uniform points sampling is difficult to achieve in high dimensions, and the randomness inherent in LHS and uniform random sampling may result in an insufficient number of sampling points in regions of low regularity, thereby leading to large errors.
    
    Our motivation is to improve the prediction performance of PINNs by reducing the estimation error via an innovative sampling method. In this work, we develop a uniform non-adaptive sampling framework, according to number-theoretic methods (NTM), which is introduced by Korobov \cite{korobov1959approximate}. It is a method based on number theory to produce uniformly distributed points, which is used to reduce the need for sampling points when calculating numerical integrals. Later on, researchers found that the NTM and the MC method share some commonalities in terms of uniform sampling, so it is also known as the quasi-Monte Carlo (QMC) method. In NTM, there are several types of low-discrepancy point sets, such as Halton sequence \cite{halton1960efficiency}, Hammersley sequence \cite{hammersley1965monte}, Sobol sequence \cite{sobol1967distribution} and  good lattice point set \cite{korobov1959approximate}. Among them, we are mainly concerned with the good lattice point set (GLP set). The GLP set is introduced by Korobov for the numerical simulations of multiple integrals \cite{korobov1959approximate}, and it has been widely applied in NTM. Subsequently,  Fang et al. introduced the GLP set into uniform design (\cite{fang1993number,fang2018theory}), leading to a significant development. Compared to uniform random sampling, the GLP set offers the following advantages: (1) The GLP set is established on the NTM, and once the generating vectors are determined, its points are deterministic. (2) The GLP set has higher uniformity, which is attributed to its low-discrepancy property. This advantage becomes more pronounced in high-dimensional cases, as the GLP set can be uniformly cattered throughout the space, which is difficult to do with uniform random sampling. (3) The error bound of the GLP set is better than that of uniform random sampling in numerical integrations. It is well known (\cite{caflisch1998monte,dick2013high,niederreiter1992random}) that the MC method approximates the integral with an error of $\bO\left(N^{-\frac{1}{2}}\right)$ when discretizing the integral using $N $ i.i.d random points. With $N$ deterministic points, the GLP set approximates the integral with an error of $\bO\left(\frac { (\log N)^{d} } {N}\right)$, where $d$ is the dimension of the integral (see Lemma \ref{lem:MC_Discrepancy}). Thus,  for the same number of sampling points, the GLP set performs better than uniform random sampling. When using PINNs to solve the PDEs, a large number of sampling points may be required, such as low regularity problems or high dimensional problems. \textcolor{black}{Research on applying good lattice points to neural networks remains limited. In \cite{matsubara2023good}, they proposed a physics-informed loss function based on Fourier series for good lattice training. However, it did not delve into cases requiring many training points, such as low-regularity and high-dimensional problems, nor did it conduct a comprehensive theoretical analysis of the solution's error. These gaps are thoroughly addressed in our study.} 
    In such case, the uniform random sampling is very inefficient. In this work, 
    GLP sampling is utilized instead of uniform random sampling, which not only reduces the consumption of computational resources, but also improves the accuracy. Meanwhile, we also give an upper bound about the error of the predicted solution of PINNs when using GLP sampling through rigorous mathematical proofs. According to the error upper bound, we prove that GLP sampling is more effective than uniform random sampling.

    In summary, our contributions in this work are as follows:
		\begin{itemize}
			\item We introduce a method to obtain good lattice point sets and, with the goal of reducing estimation errors, elaborate on how to incorporate these good lattice point sets into the loss function.
			\item Rigorous mathematical proofs are provided to explain the impact of reducing estimation errors on the accuracy of the solutions.
			\item We apply the number-theoretic method sampling neural network proposed in this paper to tackle challenging low-regularity and high-dimensional problems, achieving better performance.
		\end{itemize}

    The rest of the paper is organized as follows. In Section 2, we will give the problem formulation of PDEs and a brief review of PINNs.
    In Section 3, we will describe GLP sampling and also present the neural network framework incorporating GLP sampling and give the error estimation. Numerical experiments are shown in Section 4 to verify the effectiveness of our method. Finally, conclusions are given in Section 5.

	\section{Preliminary work}
	\label{sec:Preliminary Work}

\subsection{Problem formulation}
\label{sec:Problem formulation}
The general formulation of the PDE problem is as follows
\begin{equation}\label{eq:gen_PDE}
    \left\{
  \begin{aligned}
 \mathcal{A}[\bu(\bx)] &= f(\bx), \quad \bx \in \Omega, \\
 \mathcal{B}[\bu(\bx)] &= g(\bx), \quad \bx \in \partial\Omega.
  \end{aligned}
    \right.
\end{equation}
where $\Omega \subset \mathbb{R}^d$ represents a bounded, open and connected domain which surrounded by a polygonal boundary $\partial\Omega$.  \textcolor{black}{And $\bu(\bx) = [u_1(x),u_2(x),...,u_{d_{out}}(x)] \in \mathbb{R}^{d_{out}}$ is the unique solution of this problem. Without loss of generality, assume that $d_{out}=1$ in the subsequent content.} There are two partial differential operators $\mathcal{A}$ and $\mathcal{B}$ defined in $\Omega$ and on $\partial \Omega$, the $f(\bx)$ is the source function within $\Omega$ and $g(\bx)$ represents the boundary conditions on $\partial\Omega$.


 \subsection{ Introduction to PINNs}
\label{sec:Brief introduction to PINNs}



 For a general PINN, its input consists of sampling points $\bx$ located within the domain $\Omega$ and on its boundary $\partial \Omega$,  and its output is the predicted solution $\bu(\bx;\theta)$, where $\theta$ represents the parameters of the neural network, typically comprising weights and biases. 
Suppose that $\mathcal{A}[\bu(\bx;\theta)]-f(\bx)\in \mathbb{L_2}(\Omega)$ and $\mathcal{B}[\bu(\bx;\theta)]-g(\bx)\in \mathbb{L_2}(\partial\Omega)$, the loss function can be expressed as

\begin{equation}\label{eq:L2_loss}
    \begin{aligned}
     \mathcal{L}(\bx;\theta) &= \alpha_1\Vert  \mathcal{A}[\bu(\bx;\theta)]-f(\bx) \Vert_{2,\Omega}^2 +  \alpha_2\Vert  \mathcal{B}[\bu(\bx;\theta)]-g(\bx) \Vert_{2,\partial\Omega}^2 \\
     &= \alpha_1\int_{\Omega} |\mathcal{A}[\bu(\bx;\theta)]-f(\bx)|^2 d\bx  + \alpha_2\int_{\partial\Omega} |\mathcal{B}[\bu(\bx;\theta)]-g(\bx)|^2 d\bx \\
     & \triangleq \alpha_1\int_{\Omega} |r(\bx;\theta)|^2 d\bx + \alpha_2\int_{\partial\Omega} |b(\bx;\theta)|^2 d\bx \\
     & \triangleq \alpha_1 \mathcal{L}_{r}(\bx;\theta) + \alpha_2 \mathcal{L}_{b}(\bx;\theta),
    \end{aligned}
\end{equation}
where $\alpha_1$ is the weight of the residual loss $\mathcal{L}_{r}(\bx;\theta)$ and $\alpha_2$ is the weight of the boundary loss $\mathcal{L}_{b}(\bx;\theta)$. 

We use the residual loss $\mathcal{L}_{r}(\bx;\theta)$ as an example of how to discretize by Monte Carlo method.
\begin{equation}\label{eq:MC}
    \begin{aligned}
     \mathcal{L}_{r}(\bx;\theta) &=\int_{\Omega} |r(\bx;\theta)|^2 d\bx = V(\Omega) \int_{\Omega} |r(\bx;\theta)|^2 U(\Omega) d\bx\\
     &= V(\Omega) \mathbb{E}_{\bx\sim U}(|r(\bx;\theta)|^2) \\
     & \approx V(\Omega)  \frac{1}{N_{r}} \sum_{i = 1}^{N_{r}} |r(x_i;\theta)|^2, \quad \bx \sim U(\Omega),
    \end{aligned}
\end{equation}
where $V$ is the volume, $U$ is the uniform distribution, and $\mathbb{E}$ is the mathematical expectation.  \textcolor{black}{Thus, after sampling uniformly distributed residual training points $\left\{\bx_i\right\}_{i=1}^{N_r} \subset \Omega$ and boundary training points $\left\{\bx_i\right\}_{i=1}^{N_b}\subset \partial\Omega$, the empirical loss with a total of $N =N_{r}+N_{b}$ training points can be written as
\begin{equation}\label{eq:L2_empiricalloss}
     \mathcal{L}_N(\bx;\theta) =   \frac{\alpha_1V(\Omega)}{N_{r}} \sum_{i = 1}^{N_{r}} |r(\bx_i;\theta)|^2 +
     \frac{\alpha_2V(\Omega)}{N_{b}} \sum_{i = 1}^{N_{b}} |b(\bx_i;\theta)|^2.
\end{equation}}

Finally, we usually use gradient descent methods, such as Adam \cite{Adam} and LBFGS \cite{LBFGS}, to optimize for empirical loss $\mathcal{L}_N(\bx;\theta)$. 
\textcolor{black}{Adam is a linearly convergent stochastic gradient descent algorithm. Its momentum and stochasticity effectively overcome local minima and saddle points, making it commonly used for the first part of optimization.  On the other hand, LBFGS is a superlinearly convergent low-memory quasi-Newton method. While it converges faster than Adam, it is more prone to getting stuck in saddle points. Therefore, LBFGS is typically used for the second part of optimization, specifically for fine-tuning the results obtained from Adam optimization.}







 


\section{Main results}
\label{sec:MainlyWork}
\subsection{Method}
\label{sec:Methods}

In NTM, a criterion is necessary to measure how the sampling points are scattered in space, which is called discrepancy. There are many types of discrepancy, such as discrepancy based on uniform distribution, F-discrepancy that can measure non-uniform distributions, and star discrepancy that considers all rectangles, among others \cite{fang2018theory}. The discrepancy $ D(N,X_{N})$ described in Definition \ref{de:discrepancy} is used to describe the uniformity of a point set $X_{N}$ on the unit cube $C^d = [0,1]^d$.
For a point set $X_{N}$, the smaller its discrepancy, the higher the uniformity of this point set, and the better it represents the uniform  distribution $U(C^d)$. 
Therefore, the definition of good lattice point set generated from number-theoretic method is introduced.

\begin{definition}\cite{fang1993number}\label{de:discrepancy}
    Let $X_{N}=\left\{x_i \right \}_{i=1}^{N}$ be a set of points on $C^d$. For a vector $\gamma \in C^d$, let $K(\gamma,X_{N})$ denotes the number of points in $X_{N}$ that satisfy $x_i<\gamma$. Then there is
    \begin{equation}
    D(N,X_{N}) = \sup_{\gamma \in C^d} \left\vert \frac{K(\gamma,X_{N})}{N} - V\left( [0,\gamma]\right) \right\vert, 
    \end{equation}
   which is called the discrepancy of $X_{N}$, where $V\left( [0,\gamma]\right)$ denotes the volume of the rectangle formed by the vector $\gamma$ from 0.
\end{definition}

\begin{definition} \cite{fang1993number}\label{de:glp}
\textcolor{black}{Let $GV=(N;h_1,h_2,...,h_d) \in C^d$ } to be a vector of integers that satisfy $1 \leq h_i < N$, $h_i \neq h_j (i \neq j)$, $d<N$ and the greatest common divisor $(N,h_i)=1,i=1,...,d$. Let
\begin{equation}\label{eq:glp}
    \left\{
    \begin{aligned}
     q_{ij} &\equiv ih_j \pmod{N},  \\
     x_{ij} &= \frac{2q_{ij}-1}{2N}, \quad i=1,...,N, \ j=1,...,d,
    \end{aligned}
    \right.
\end{equation}
where $q_{ij}$ satisfies $1 \leq q_{ij} \leq N$. Then the point set $X_N=\left\{x_i=(x_{i1},x_{i2},...,x_{id}), i=1,...,N \right \} $ is called the lattice point set of the generating vector $GV$. And a lattice point set $X_N$ is a good lattice point set (GLP set) if $X_N$ has the smallest discrepancy $D(N,X_N)$ among all possible generating vectors.
\end{definition}

Definition \ref{de:glp}  shows that GLP set has a good low discrepancy. Usually, the lattice points of an arbitrary generating vector are not uniformly scattered. Then, an important problem is how to choose the appropriate generating vectors to generate the GLP set via Eq \eqref{eq:glp}. According to (\cite{korobov1959evaluation,hlawka1962angenaherten}), there is a fact that for a given prime number $P$, there is a vector \textcolor{black}{$(h_1,h_2,...,h_d)$} can be chosen such that the lattice point set of \textcolor{black}{$(P;h_1,h_2,...,h_d)$} is GLP set. One feasible method (\cite{korobov1959evaluation,niederreiter1977pseudo}) to find the generating vector is described as follows:

\begin{itemize}
    \item[1)] For a lattice point set $X_N$ whose cardinality is a certain prime number $N=P$.
    \item[2)] Finding a certain special primitive root $a$ among the primitive roots modulo $P$.
    \item[3)] The generating vector satisfy $(h_1,h_2,...,h_d) \equiv(1,a,a^2,...,a^{d-1}) \pmod P$, $1<a<P$.
    \item[4)] The generating vector can be constructed in the form $GV=(P;1,h_1,h_2,...,h_d)$.
\end{itemize}

  When $N=2, 4, P^l, 2P^{l} $, where $P>2$ is a prime and $1\leq l $, 
  it is known that there must exist primitive roots mod $N$.  Moreover, the number of primitive roots is $\phi (\phi(N))$, where $\phi(N)$ denotes the cardinality of the set $\{ i| 1 \leq i <N, (i,N)=1 \}$. The generating vector can be obtained in the same way as above. Since the GLP set is very efficient and convenient, a number of generating vectors corresponding to some good lattice point sets on unit cubes in different dimensions have already been presented in \cite{keng1981applications}. \textcolor{black}{In fact, the method for obtaining a GLP set is not limited to the one mentioned above. For instance, it has been pointed out in \cite{matsubara2023good} that in two-dimensional cases, a GLP set can be constructed using the Fibonacci sequence, while in higher-dimensional cases, a GLP set can be obtained through Bernoulli polynomials \cite{sloan1994lattice}.}


After getting the GLP set, we now discuss the calculation of the loss function of PINNs.
Since the boundary loss usually accounts for a relatively small percentage of the total and can even be completely equal to 0 by forcing the boundary conditions(\cite{lyu2020enforcing}, \cite{lyu2022mim}), we focus on residual loss in the following. 
According to the central limit theorem and the Koksma-Hlawka inequality \cite{hlawka1961funktionen}, the following lemma is shown in \cite{niederreiter1992random}.

\begin{lemma}\cite{niederreiter1992random}
    \label{lem:MC_Discrepancy}
(i) If $f(\bx) \in \mathbb{L_2}(C^d) $ and $\left\{\bx_i\right\}_{i=1}^{N} \subset C^d$ is a set of i.i.d. uniformly distributed random points, then there is

  $$ \int_{C^d} f^2(\bx) d\bx \leq  \frac{1}{N} \sum_{i = 1}^{N} f^2(x_i) + \bO\left(N^{-\frac{1}{2}}\right).$$
(ii)
If $f(\bx)$ over $C^d$ has bounded variation in the sense of Hardy and Krause and $\left\{\bx_i\right\}_{i=1}^{N} \subset [0,1)^d$ is a good lattice point set, then there is

 $$ \int_{C^d} f^2(\bx) d\bx \leq  \frac{1}{N} \sum_{i = 1}^{N} f^2(x_i) + \bO\left(\frac { (\log N)^{d} } {N}\right).$$
\end{lemma}
\begin{remark}
\label{rem:mk_lemma1}
(1) In Lemma \ref{lem:MC_Discrepancy}(i), the error bound is a probabilistic error bound of the form $\bO\left(N^{-\frac{1}{2}}\right)$ in terms of $N$.
(2) According to \cite{hlawka1961funktionen} and \cite{fu2022convergence}, a sufficient condition for a function f to have bounded variance in the Hardy-Krause sense is that f has continuous mixed partial differential derivatives. In fact, some common activation functions in neural networks, such as the hyperbolic tangent function and the Sigmoid function, they are infinite order derivable. Therefore, when the activation functions mentioned above are chosen, the function $r(\bx;\theta)$ and $b(\bx;\theta)$ induced by the predicted solutions $\bu(\bx;\theta)$ generated by the PINNs also satisfies the conditions of Lemma \ref{lem:MC_Discrepancy}(ii) (similar comments can be found in \cite{matsubara2023good}.).

\end{remark}

Recalling Eq \eqref{eq:MC}, the discrete residual loss in $\Omega = (0,1)^{d}$ using uniform sampling can be written as
\begin{equation}\label{eq:DiscreteResLoss}
    \begin{aligned}
     \mathcal{L}_{r}(\bx;\theta) &=\int_{\Omega} |r(\bx;\theta)|^2 d\bx =  \mathbb{E}_{\bx\sim U}(|r(\bx;\theta)|^2) \approx \frac{1}{N_{r}} \sum_{i = 1}^{N_{r}} |r(x_i;\theta)|^2, \quad \bx \sim U(\Omega).   
    \end{aligned}
\end{equation}

Given a good lattice point set $\mathcal{X}_{N_r}=\left\{\bx_i\right\}_{i=1}^{N_r} \subset \Omega$ , there is a discrete residual loss defined in Eq \eqref{eq:GLPResLoss}. In notation, we use superscripts to distinguish between sampling methods, where GLP denotes good lattice point set sampling and UR denotes uniform random sampling.
\begin{equation}\label{eq:GLPResLoss}
     \mathcal{L}_{r,N_r}^{GLP}(\bx;\theta) = \frac{1}{N_{r}} \sum_{i = 1}^{N_{r}} |r(x_i;\theta)|^2, x_i \in \mathcal{X}_{N_r}.  
\end{equation}

If the conditions of Lemma \ref{lem:MC_Discrepancy} are satisfied, the error bound between $\mathcal{L}_{r,N_r}$ and $\mathcal{L}_{r}$ using GLP sampling will be smaller than  the error using uniform random sampling when there are enough training points, which will result in smaller error estimate presented in Section \ref{sec:Error estimation}. According to Eq \eqref{eq:L2_empiricalloss}, we have the following loss function.
\textcolor{black}{
\begin{equation}\label{eq:GLPLoss}
     \mathcal{L}_N(\bx;\theta) =   \alpha_1 \mathcal{L}_{r,N_r}^{GLP}(\bx;\theta)  +
     \alpha_2 \mathcal{L}_{b,N_b}^{UR}(\bx;\theta),
\end{equation}
where $\mathcal{L}_{b,N_b}^{UR}(\bx;\theta) = \frac{1}{N_{b}} \sum_{i = 1}^{N_{b}} |b(x_i;\theta)|^2$, $x_i$ belongs to a uniform random point set $\mathcal{X}_{N_b}=\left\{\bx_i\right\}_{i=1}^{N_b} \subset \partial\Omega$ and $N =N_r + N_b$.}

Finally, the gradient descent method is used to optimize the loss function Eq \eqref{eq:GLPLoss}. The flow of our approach is summarized in Algorithm \ref{alg:NTM-PINN} and Fig \ref{fig:flowchart}. Without loss of generality, in all subsequent derivations, we choose $\Omega =(0,1)^d$, and $\alpha_1=\alpha_2=1$ in Algorithm \ref{alg:NTM-PINN}. 

In our work, no matter how the set of points is formed, the uniform distribution sampling approach is utilized. In order to give the error estimate of our method under the premise that both $\mathcal{A}$ and $\mathcal{B}$ in Eq  \eqref{eq:gen_PDE} are linear operators, several necessary assumptions are introduced.
\begin{assumption}
    \label{asu:Loss_LCcondition}
    (Lipschitz continuity) For any bounded neural network parameters $\theta_1$ and $\theta_2$, there exists a postive constant $\mathbf{L}$ such that the empirical loss function satisfies

    \begin{equation}\label{eq:Loss_LCcondition}
        \Vert \nabla \mathcal{L}_{N}(\bx;\theta_1) -\nabla \mathcal{L}_{N}(\bx;\theta_2) \Vert_{2} \leq \mathbf{L} \Vert  \theta_1-\theta_2 \Vert_{2}.
    \end{equation}

\end{assumption}

Assumption \ref{asu:Loss_LCcondition} is a common assumption in optimization analysis since it is critical to derive Eq \eqref{eq:descentlemma}, which is known as descent lemma.
\begin{equation}\label{eq:descentlemma}
         \mathcal{L}_{N}(\bx;\theta_1)  \leq \mathcal{L}_{N}(\bx;\theta_2) + \nabla \mathcal{L}_{N}(\bx;\theta_2)^T(\theta_1-\theta_2 ) + \frac{1}{2}  \mathbf{L}  \Vert  \theta_1-\theta_2 \Vert_{2}^2.
\end{equation}

\textcolor{black}{Consider a fully connected neural network with  $l_{NN}$ layers, each layer having a width of $h_{NN}$, and with a linear output layer, 
\begin{equation}\label{eq:FCNN}
\left\{
\begin{aligned}
 &Input \ f_0^{NN} = \bx,  \\
 &Hidden \ f_{i}^{NN} = \sigma(\frac{W^{(i)}}{\sqrt{h_{NN}}} f_{i-1}^{NN}), \ 1 \leq i \leq l_{NN}-1, \\
 &Output \ \bu(\bx;\theta) = \frac{W^{(i)}}{\sqrt{h_{NN}}} f_{l_{NN}-1}^{NN}.
\end{aligned}
\right.
\end{equation}
Then, we have the following lemma.
\begin{lemma} \cite{liu2022loss}
    \label{lem:PL*}
Suppose that the initialization parameters of each layer satisfy 
$W^{(i)}_0 \sim \mathcal{N}(0,\mathbf{I}_{h_{NN} \times h_{NN}}), i\leq l_{NN}+1$ and the minimum eigenvalue of the tangent kernel of the loss function  $\lambda_{min} >0$. $\forall c \in (0, \lambda_{min})$, if the width $h_{NN}$
satisfy 
    \begin{equation}\label{eq:lem2_1}
        h_{NN} = \tilde{\Omega} \left(\frac{R^{2+6l_{NN}}}{(\lambda_{min}-c)^2}\right),
    \end{equation}
then $c-PL^*$ condition holds for the loss function in the ball $B(\theta_0, R)$,
\begin{equation}\label{eq:lem2_2}
\Vert \nabla \mathcal{L}_{N}(\bx;\theta)  \Vert_2^2 \geq c  \mathcal{L}_{N}(\bx;\theta) ,\ \forall \theta \in B(\theta_0, R).
\end{equation}
\end{lemma}
As is known to all, for an over-parameterized deep feedforward neural network with nonlinear activation functions, its squared loss function doesn't have the global strong convexity property, nor does it meet the Polyak-Łojasiewicz(PL) condition and invexity \cite{karimi2016linear}. To gain a deeper understanding of the optimization mechanism of neural networks, we need to look into the local properties of the loss function. According to Lemma \ref{lem:PL*}  (Theorem 4 in \cite{liu2022loss}), the loss function meets the $PL^*$ condition within the ball centered at the initial parameters under certain conditions.}

 If the stochastic gradient descent method is utilized, the parameter update equation at the i+1st iteration in the general stochastic gradient descent method is given by
\begin{equation}\label{eq:sGV}
          \theta_{i+1} = \theta_{i} - \eta_i g(x_i,\theta_i),
\end{equation}
where $\eta_i$ is the step size at the i+1st iteration. \textcolor{black}{After M training epochs, we can obtain M+1 parameters $\left\{\theta_i\right\}_{i=0}^M$. According to Section 5 in \cite{liu2022loss}, we assume that the optimization path is covered by a ball $B(\theta_0, R)$.
\begin{assumption}
    \label{asu:PLcondition}
 Assume that our neural network satisfies the conditions of Lemma \ref{lem:PL*}, and 
 $$\exists R >0,\ s.t. ,\ \left\{\theta_i\right\}_{i=0}^M \subset B(\theta_0, R).$$
\end{assumption}
Combining Lemma \ref{lem:PL*} and Assumption \ref{asu:PLcondition}, it can be concluded that the loss function satisfies the $PL^*$ condition during the optimization process. For convenience, we assume that the loss function satisfies the 2c - $PL^*$ condition, that is,
\begin{equation}\label{eq:2cPL}
         \Vert \nabla \mathcal{L}_{N}(\bx;\theta)  \Vert_2^2 \geq 2c  \mathcal{L}_{N}(\bx;\theta).
\end{equation}
} 
Same as the reference \cite{bottou2018optimization,chen2021quasi}, we give the following assumption about the stochastic gradient vector 
$g(x_i,\theta_i) : = \nabla \mathcal{L}_{N}(x_i;\theta_i)$.

\begin{assumption}
    \label{asu:stochastic_gradient}
    \textcolor{black}{
    For  the stochastic gradient vector $g(x_i,\theta_i) := \nabla \mathcal{L}_{N}(x_i;\theta_i) $, $\forall i \in \mathbb{N}^+$,}

    \textcolor{black}{
    (i) there exists $0<\mu \leq \sqrt{\frac{\mathbf{L}}{c}} \mu_G$ such that the  expectation $\mathbb{E}_{x_i}\left(g(x_i,\theta_i)\right)$ satisfies
    \begin{eqnarray}
     & &\nabla \mathcal{L}_{N}(\bx;\theta_i)^T \mathbb{E}_{x_i}\left(g(x_i,\theta_i)\right)  \geq  \mu  \Vert \nabla \mathcal{L}_{N}(\bx;\theta_i) \Vert_2^2, \label{eq:asumoment1}\\
      & & \Vert  \mathbb{E}_{x_i}\left(g(x_i,\theta_i)\right) \Vert_{2} \leq  \mu_G  \Vert\nabla\mathcal{L}_{N}(\bx;\theta_i) \Vert_2, \label{eq:asumoment2}
    \end{eqnarray}}

    (ii) there exist $C_V \geq 0$ and $M_V \geq 0$ such that the variance
    $\mathbb{V}_{x_i}\left(g(x_i,\theta_i)\right)$ satisfies
    \begin{equation}\label{eq:asumoment3}
       \mathbb{V}_{x_i}\left(g(x_i,\theta_i)\right) \leq C_Vs(N_r,N_b)+M_V(1+s(N_r,N_b))\Vert \nabla \mathcal{L}_{N}(\bx;\theta_i) \Vert_2^2.
    \end{equation}
\end{assumption}
In Assumption \ref{asu:stochastic_gradient}, $s(N_r,N_b)$ is the square of the error between the empirical loss function $\mathcal{L}_{N}$ and the loss function in integral form $\mathcal{L}$. For example, if the conditions of Lemma \ref{lem:MC_Discrepancy} are satisfied, the $s(N_r,N_b) = \left( \bO\left({N_r}^{-\frac{1}{2}}\right) + \bO\left({N_d}^{-\frac{1}{2}}\right) \right)^2$ for the uniform random sampling and $s(N_r,N_b) = \left( \bO\left(\frac { (\log N_r)^{d} } {N_r}\right) + \bO\left({N_d}^{-\frac{1}{2}}\right)\right)^2$  for the GLP set. According to \cite{das-pinns} and \cite{MSPINN}, the following stability bound assumption is given.
\begin{assumption}
    \label{asu:Normrelations}
  Let  $\mathbb{L_2}(\Omega)$ be the $L_2$ space defined on  $\Omega$, the following assumption holds
    \begin{equation}\label{eq:Normrelations}
        C_1 \Vert \bu \Vert_{2} \leq \Vert \mathcal{A}[\bu] \Vert_{2} +  \Vert \mathcal{B}[\bu] \Vert_{2} \leq C_2 \Vert \bu \Vert_{2}, \  \forall \bu \in \mathbb{L_2}(\Omega),
    \end{equation}
    where  $\mathcal{A}$ and $\mathcal{B}$ in Eq \eqref{eq:gen_PDE} are linear operators, $C_1$ and $C_2$ are positive constants.
    \end{assumption}


\begin{algorithm}[htbp]
\caption{Algorithm of PINNs on GLP set}\label{alg:NTM-PINN}

\textbf{Symbols:} Maximum epoch number of PINNs training: $M$; the numbers of total points in $\Omega$  and $\partial \Omega$: $N_r$ and $N_b$; generating vector: $GV=(N_r;h_1,h_2,.... ,h_d)$; uniform boundary training points $\mathcal{X}_{N_b} :=\left\{\bx_k \right\}_{k=1}^{N_b} \subset \partial\Omega$;  the number of test set points: $N_T$; test set $\mathcal{X}_T :=\left\{\bx_k \right\}_{k=1}^{N_T} \subset \Omega \cup \partial\Omega$; the parameters of PINNs: $\theta$ .

\textcolor{black}{
\textbf{GLP Sampling:}\\
Based on the number of points $N_r$, find the generating vector $GV=(N_r;h_1,h_2,.... ,h_d)$.}

\textcolor{black}{
Input generating vector $GV=(N_r;h_1,h_2,.... ,h_d)$ for the GLP set:}

\textcolor{black}{
$x_{ij} = \frac{2q_{ij}-1}{2N_r}$, where $ q_{ij} \equiv ih_j \pmod{N_r},\quad i=1,...,N_r, j=1,...,d.$}

\textcolor{black}{
Get the GLP set as the residual training set $ \mathcal{X}_{N_r} := \left\{x_k=(x_{k1},x_{k2},...,x_{kd})\right \}_{k=1}^{N_r} \subset \Omega$ .
}

\textbf{Training:}\\
Input  $\mathcal{X}_{N} := \mathcal{X}_{N_r} \cup \mathcal{X}_{N_b}$ into PINNs. 

Initialize the output $\bu\left( \mathcal{X}_{N};\theta_0 \right)$.

\For{$i =1:M$}{
 
$\mathcal{L}_{N}\left(\mathcal{X}_{N};\theta_i\right) = \alpha_1 \mathcal{L}_{r,N_r}^{GLP}\left(\mathcal{X}_{N_r};\theta_i\right)  +
     \alpha_2 \mathcal{L}_{b,N_b}^{UR}\left(\mathcal{X}_{N_b};\theta_i\right)$;

Update $\theta_{i+1}$ by optimizing $\mathcal{L}_{N}\left(\mathcal{X}_{N};\theta_i\right)$.
}

\textbf{Testing:}\\
Input  test set $\mathcal{X}_T$ into PINNs. 

Output $u(\mathcal{X}_T;\theta_{M+1})$.
\end{algorithm}

\begin{figure}[htbp]
\centering
\includegraphics[width = 0.95\textwidth]{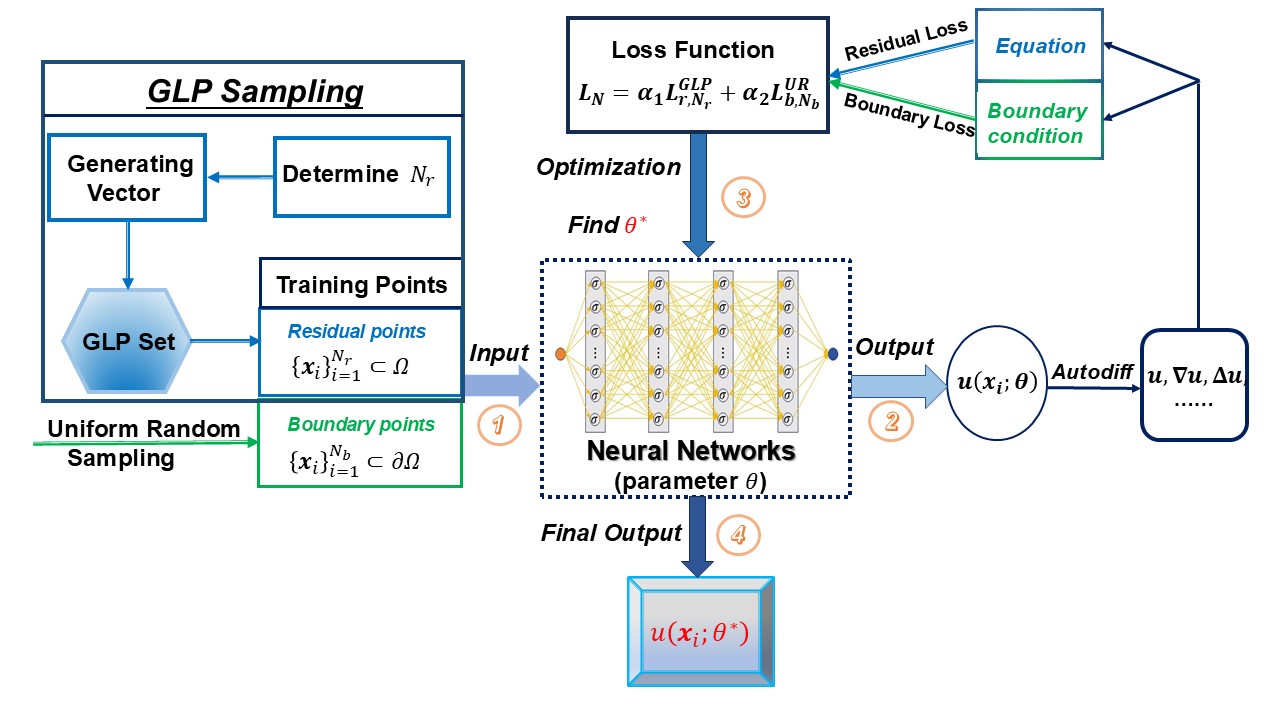}
\caption{Flow chart of our method.}
\label{fig:flowchart}
\end{figure}

\subsection{Error Estimation}
\label{sec:Error estimation}
In this subsection, the error estimate of our method is presented in Theorem \ref{thm:1}, \ref{thm:2}  and \ref{thm:3} on the basis of the previous Assumptions. Corollary \ref{cor:1} shows that the upper bound of the error estimate for our method is better than that of uniform random set.

\begin{theorem}
\label{thm:1}
Suppose that Assumptions \ref{asu:Loss_LCcondition},\ref{asu:PLcondition} and \ref{asu:stochastic_gradient} hold and the stepsize $\eta$ is fixed and satisfies
\begin{equation}\label{eq:assumptioninThm1}
0<\eta\leq \frac{\mu}{ \mathbf{L}\left(M_V(1+s(N_r,N_b))+\mu^2_G\right)},
\end{equation}
then we have
\begin{equation}\label{eq:optim_gap}
    \lim_{i \to \infty} \mathbb{E}\left(\mathcal{L}_{N}(\bx;\theta_i)\right) = \frac{\eta \mathbf{L}C_V}{2c\mu} s(N_r,N_b).
\end{equation}
\end{theorem}

\begin{proof}
See Appendix A.

\end{proof}

\begin{remark}
\label{rem:mk_thm1}
 
It is noted that the step size (learning rate $\eta$) in Theorem \ref{thm:1} is fixed. However, when we implement the algorithm, a diminishing learning rate sequence $\{\eta_i \}$ is often selected. We can still obtain the following Theorem \ref{thm:2}, which is  similar  to Theorem \ref{thm:1}.
\end{remark}

\begin{theorem}
\label{thm:2}
Suppose that Assumptions \ref{asu:Loss_LCcondition},\ref{asu:PLcondition} and \ref{asu:stochastic_gradient} hold and the stepsize $\{\eta_i\}$ is a diminishing sequence
and satisfies
$\{ \eta_i = \frac{\beta}{\xi+i}, \ \text{for all} \ i \in \mathbb{N} \}$ for some $\beta > \frac{1}{c\mu}$ and $\xi >0$ such that $\eta_1 \leq \frac{\mu}{ \mathbf{L}\left(M_V(1+s(N_r,N_b))+\mu^2_G\right)},$
then we have
\begin{equation}\label{eq:optim_gap_diminishingstep}
   \mathbb{E}\left(\mathcal{L}_{N}(\bx;\theta_i)\right) \leq \frac{\kappa }{\xi+i} s(N_r,N_b), \  \forall i \in \mathbb{N},
\end{equation}
where $\kappa := max \{ \frac{\beta^2 \mathbf{L}  C_V}{2(\beta c \mu -1)}, \frac{(\xi+1) \mathcal{L}_{N}(\bx;\theta_1)}{s(N_r,N_b)}\}$.
\end{theorem}

\begin{proof}
See Appendix A.

\end{proof}




\begin{theorem}
\label{thm:3}
Suppose Assumption \ref{asu:Normrelations} holds, let $\bu^* \in \mathbb{L_2}(\Omega) $ be the exact solution of Eq \eqref{eq:gen_PDE}. If $r(\bx;\theta) = \mathcal{A}[\bu(\bx;\theta)] - f(\bx)  $  and $b(\bx;\theta) =\mathcal{B}[\bu(\bx;\theta)] - g(\bx)$ have continuous  partial derivatives, the following error estimates hold.\\
(i) When $\left\{ \bx_i \right\}_{i=1}^{N_r}$ is a good lattice point set,
\begin{equation}\label{eq:errorestimate_GLP}
    \Vert \bu^*(\bx)-\bu(\bx;\theta) \Vert_{2} \leq  \frac{\sqrt{2}}{C_1} \left(\mathcal{L}_{r,N_r}^{GLP} + \mathcal{L}_{b,N_b}^{UR}
      + \bO\left(\frac { (\log N_r)^{d} } {{N_r}}\right) +\bO\left(N_b^{-\frac{1}{2}}\right) \right)^{\frac{1}{2}}.
\end{equation}
(ii) When $\left\{ \bx_i \right\}_{i=1}^{N_r}$ is a uniform random point set,
\begin{equation}\label{eq:errorestimate_UR}
    \Vert \bu^*(\bx)-\bu(\bx;\theta) \Vert_{2} \leq  \frac{\sqrt{2}}{C_1} \left(\mathcal{L}_{r,N_r}^{UR} + \mathcal{L}_{b,N_b}^{UR}
      + \bO\left(N_r^{-\frac{1}{2}}\right) +\bO\left(N_b^{-\frac{1}{2}}\right) \right)^{\frac{1}{2}}.
\end{equation}

\end{theorem}
\begin{proof}
See Appendix A.

\end{proof}

\begin{corollary}\label{cor:1}
    Suppose that Theorem \ref{thm:1} and Theorem \ref{thm:3} hold, then we have\\
    (i) If $\left\{ \bx_k \right\}_{k=1}^{N_r}$ is a good lattice point set,
    \begin{equation}\label{eq:errorestimate_cor_GLP}
    \begin{aligned}
      \lim_{i \to \infty}\mathbb{E}\left(\Vert \bu^*(\bx)-\bu(\bx;\theta_i) \Vert_{2}^2\right)
     \leq \frac{2}{{C_1}^2} \left(\frac{\eta \mathbf{L}C_V}{2c\mu} \left( \bO\left(\frac { (\log N_r)^{d} } {{N_r}}\right) +\bO\left(N_b^{-\frac{1}{2}}\right)\right)^2 + \bO\left(\frac { (\log N_r)^{d} } {{N_r}}\right) +\bO\left(N_b^{-\frac{1}{2}}\right) \right) .
     \end{aligned} 
\end{equation}
    (ii) If $\left\{ \bx_k \right\}_{k=1}^{N_r}$ is a uniform random  point set,
    \begin{equation}\label{eq:errorestimate_cor_UR}
    \begin{aligned}
      \lim_{i \to \infty}\mathbb{E}\left(\Vert \bu^*(\bx)-\bu(\bx;\theta_i) \Vert_{2}^2\right)
     \leq \frac{2}{{C_1}^2} \left(\frac{\eta \mathbf{L}C_V}{2c\mu} \left( \bO\left(N_r^{-\frac{1}{2}}\right) +\bO\left(N_b^{-\frac{1}{2}}\right)\right)^2 + \bO\left(N_r^{-\frac{1}{2}}\right) +\bO\left(N_b^{-\frac{1}{2}}\right) \right) .
     \end{aligned} 
     \end{equation}
\end{corollary}

\begin{proof}
See Appendix A.

\end{proof}

\begin{remark}
\label{rem:mk_cor2}
It is a well known that $\bO\left(\frac { (\log N_r)^{d} } {{N_r}}\right) < \bO\left(N_r^{-\frac{1}{2}}\right)$ when $N_r$ is large. 
That is, in Corollary 1, the upper bound of $\lim_{i \to \infty}\mathbb{E}\left(\Vert \bu^*(\bx)-\bu(\bx;\theta_i) \Vert_{2}^2\right)$ is smaller if $\left\{ \bx_i \right\}_{i=1}^{N_r}$  is a good lattice point set when $N_r$ is large . From this point of view, the effectiveness of the good lattice point set is validated.

\textcolor{black}{
Moreover, it is important to note that in practical numerical computations, increasing the number of points $N_r$  significantly raises computational resource demands. To check if the logarithmic term $(\log N_r)^{d}$ weakens GLP sampling's advantage under such conditions, we designed comparative experiments related to  $N_r$ in the subsequent examples in five and eight dimensions (Tables (\ref{tab:PoissonHd_PointsSettings}, 
 \ref{tab:Poisson8d_PointsSettings},
 \ref{tab:NonlinearHd_PointsSettings} ,
 \ref{tab:Nonlinear8d_PointsSettings})). Numerical results confirm that GLP sampling outperforms uniform random sampling.
However, when $N_r$ is limited and the dimension d is very large, the logarithmic term $(\log N_r)^{d}$ may diminish the advantages of our method.}
\end{remark}



\section{Numerical Experiments}
\label{sec:Numericalsurvey}


\subsection{Symbols and parameter settings}
\label{sec:Symbols and parameter settings}
The $u^*$ is the exact solution,
 $u^{\text{UR}}$ and $u^{\text{GLP}}$ denote the approximate solutions using uniform random sampling and good lattice point set, respectively.
The relative errors in the test set $\left\{\bx_i\right\}_{i=1}^{M_t}$ sampled by the uniform random sampling method are defined as follows
\begin{equation}
    \label{eq:measure}
    \hspace{-0.3cm}
    \begin{array}{r@{}l}
        \begin{aligned}
            e_\infty(u^{\text{UR}}) & = \frac{\max \limits_{1 \leq i \leq M_t}\vert u^*(x_i) -u^{\text{UR}}(x_i;\theta)\vert}{\max \limits_{1 \leq i \leq M_t}\vert u^*(x_i) \vert},\\
            e_2(u^{\text{UR}}) & 
            = \frac{\sqrt{\sum_{i=1}^{M_t}\vert u^*(x_i) -u^{\text{UR}}(x_i;\theta)\vert^2}}{\sqrt{\sum_{i=1}^{M_t}\vert u^*(x_i)\vert^2}} .
        \end{aligned}
    \end{array}
\end{equation}
In order to reduce the effect of randomness on the calculation results, we set all the seeds of the random numbers to 100. \textcolor{black}{Regarding performance, efficiency, and the assumptions about stochastic gradient descent in the proof, this paper adopts the general no-batch training method as in \cite{PINN}.}
The default parameter settings in PINNs are given in Table \ref{tab:parameter-PINN}. 
All numerical simulations are carried on NVIDIA A100 with 80GB of memory and NVIDIA A800 with 80GB of memory.




\begin{table}
	\centering
	\caption{Default settings for main parameters in PINNs.}
	\label{tab:parameter-PINN} 
	\setlength{\tabcolsep}{0.5mm}{
		\begin{tabular}{|c|c|c|c|c|c|c|c|}
			\hline\noalign{\smallskip}
			 torch  & activation &  initialization  & optimization& learning& loss weights& net size& random \\
			 version&  function  &         method   &       method &  rate &($\alpha_1$ / $\alpha_2$)&(two-dimensional/high-dimensional)& seed \\
			\hline
			2.4.0 & tanh & Xavier & Adam/LBFGS & 0.0001/1& 1/1& $40 \times 4$ / $32 \times 8$&100\\
            \hline
		\end{tabular}
	}
\end{table}
 
\subsection{Two-dimensional Poisson equation}
\subsubsection{Two-dimensional Poisson equation with one peak}
\label{sec:OnePeak_2D}
For the following Poisson equation in two-dimension
\begin{equation}
	\label{eq:2d_Poisson}
	\hspace{-0.3cm}
	\begin{array}{r@{}l}
		\left\{
		\begin{aligned}
			 -\Delta u(x,y) & = f(x,y), \quad (x,y) \ \mbox{in} \ \Omega, \\
                  u(x,y) & = g(x,y),  \quad  (x,y) \ \mbox{on} \  \partial \Omega,
		\end{aligned}
		\right.
	\end{array}
\end{equation}
where $\Omega = (-1,1)^2$,  the exact solution which has a peak at $(0,0)$ is chosen as
\begin{equation}
    \label{eq:2d_Peaksolution}
    \hspace{-0.3cm}
    \begin{array}{r@{}l}
        \begin{aligned}
            u = e^{-1000(x^2+y^2)}.
        \end{aligned}
    \end{array}
\end{equation}
The Dirichlet boundary condition $g(x,y)$  and the source function $f(x,y)$ are given by Eq \eqref{eq:2d_Peaksolution}.

In order to compare our method with other five methods, we sample 10946 points in $\Omega$ and 1000 points on $\partial \Omega$ as the training set and $400\times400$ points as the test set for all the methods.
In the following experiments, PINNs are all trained 50000 epochs. 
Six different uniform sampling methods are implemented: (1) uniform random sampling;
(2) LHS method; 
(3) Halton sequence; 
(4) Hammersley sequence; 
(5) Sobol sequence;
(6) GLP sampling. Methods (2)--(5) are implemented by the DeepXDE package \cite{lu2021deepxde}.

\begin{figure}[htbp]
\centering
\subfloat[sampling points]{\includegraphics[width = 0.27\textwidth]{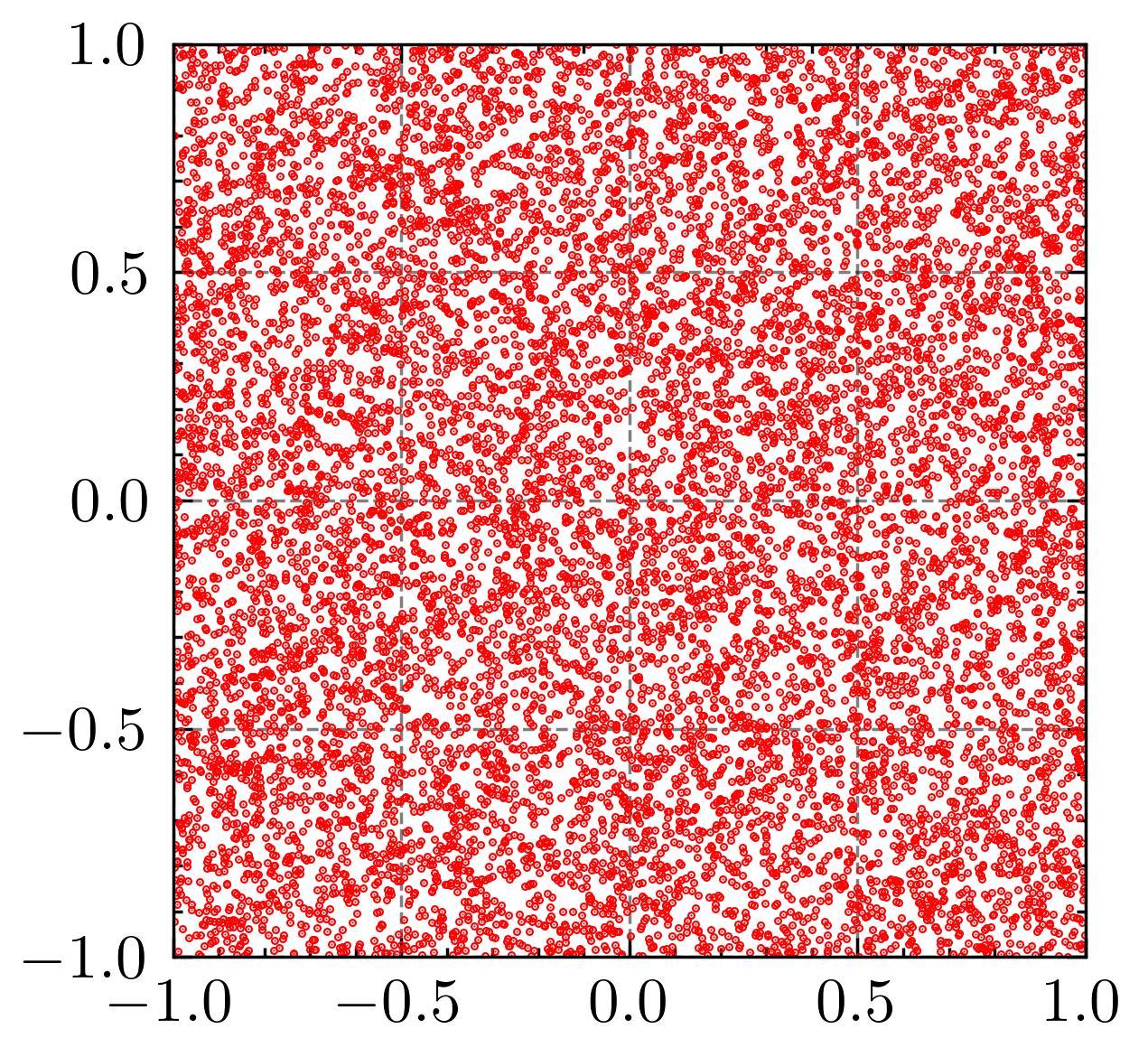}}
\subfloat[approximate solution]{\includegraphics[width = 0.31\textwidth]
{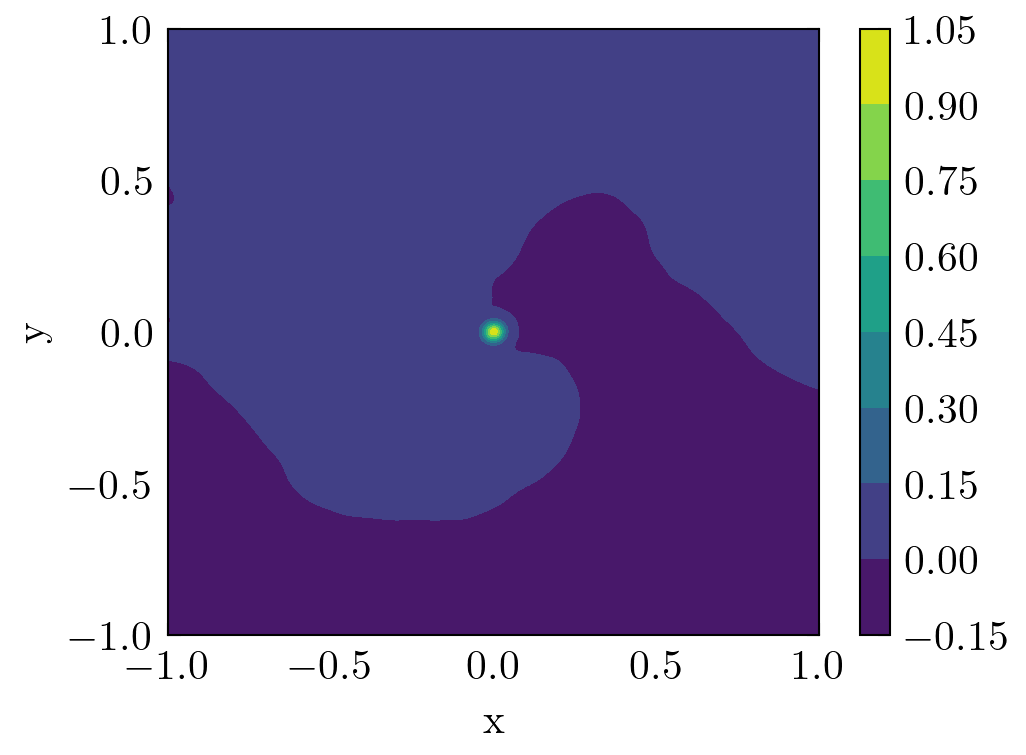}}
\subfloat[absolute error]{\includegraphics[width = 0.31\textwidth]{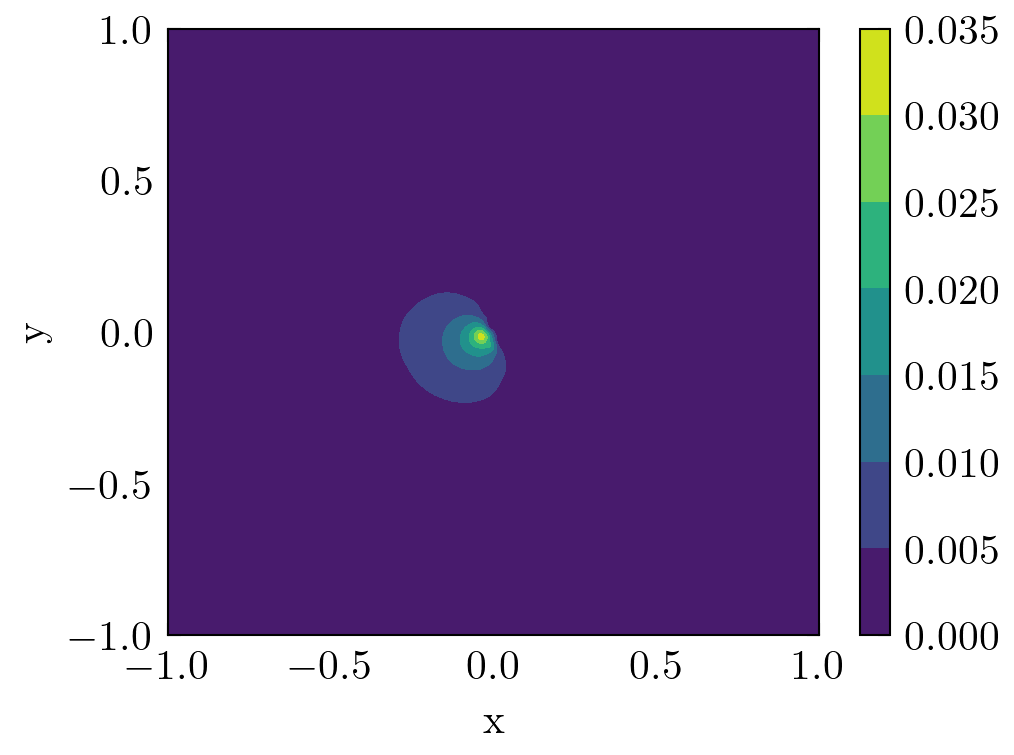}}
\caption{The result of uniform random sampling for the two-dimensional Poisson equation  with the solution Eq \eqref{eq:2d_Peaksolution}. (a) training points sampled by uniform random sampling method; (b) the approximate solution $u^{\text{UR}}$; (c) the absolute error $\vert u^* - u^{\text{UR}} \vert$.}
\label{fig:Peak2D_UniformRandom}
\end{figure}

\begin{figure}[htbp]
\centering
\subfloat[sampling points]{\includegraphics[width = 0.27\textwidth]{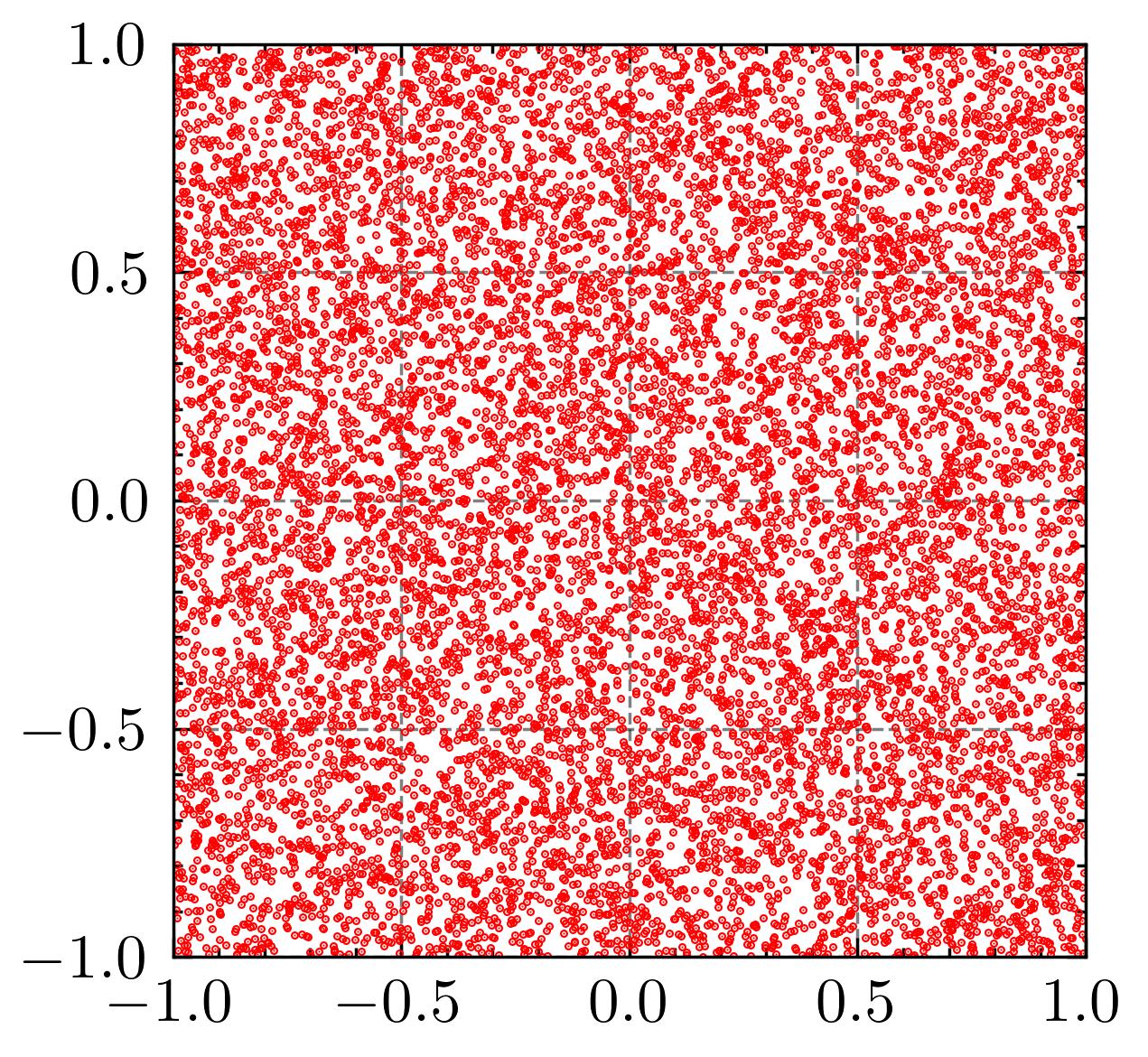}}
\subfloat[approximate solution]{\includegraphics[width = 0.31\textwidth]
{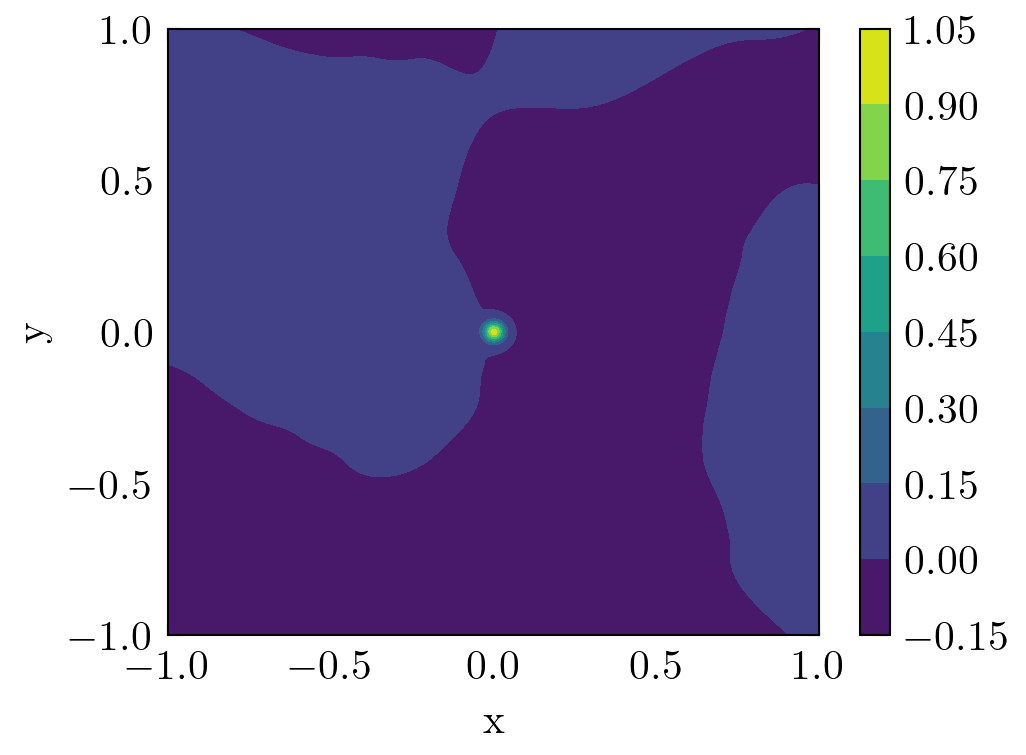}}
\subfloat[absolute error]{\includegraphics[width = 0.31\textwidth]{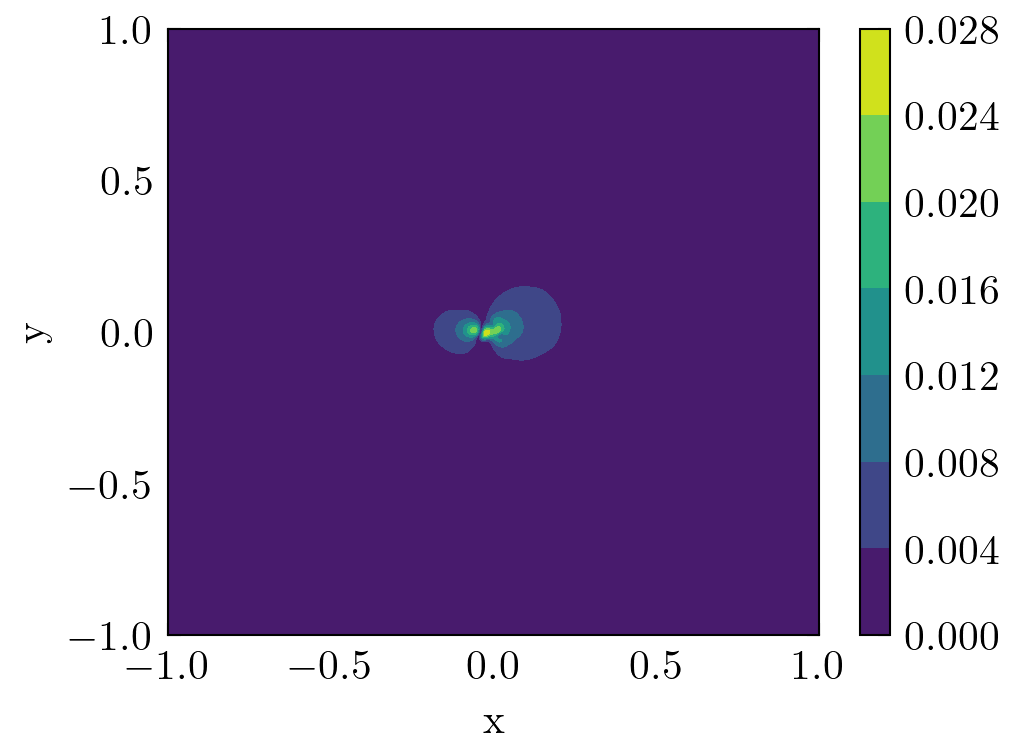}}
\caption{The result of LHS method for the two-dimensional Poisson equation  with the solution Eq \eqref{eq:2d_Peaksolution}. (a) training points sampled by LHS method; (b) the approximate solution $u^{\text{LHS}}$; (c) the absolute error $\vert u^* - u^{\text{LHS}} \vert$.}
\label{fig:Peak2D_LHS}
\end{figure}

\begin{figure}[htbp]
\centering
\subfloat[sampling points]{\includegraphics[width = 0.27\textwidth]{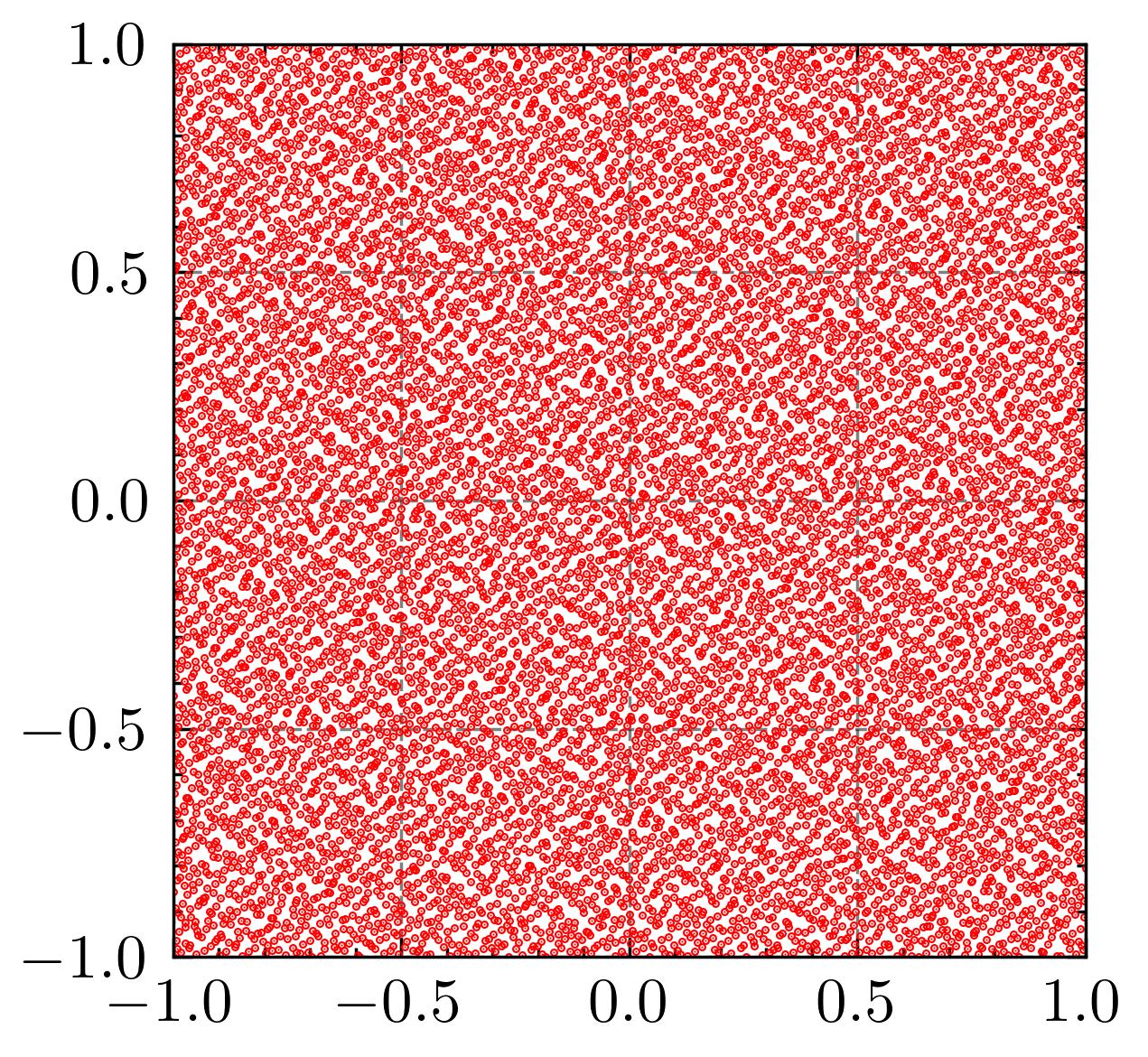}}
\subfloat[approximate solution]{\includegraphics[width = 0.31\textwidth]
{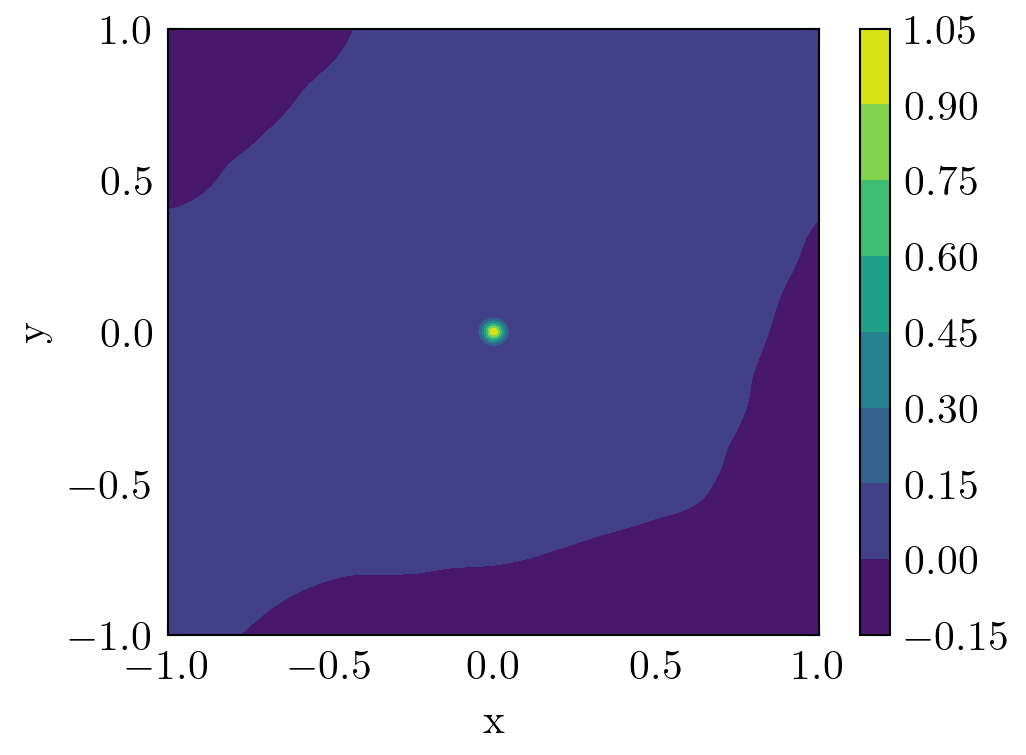}}
\subfloat[absolute error]{\includegraphics[width = 0.31\textwidth]{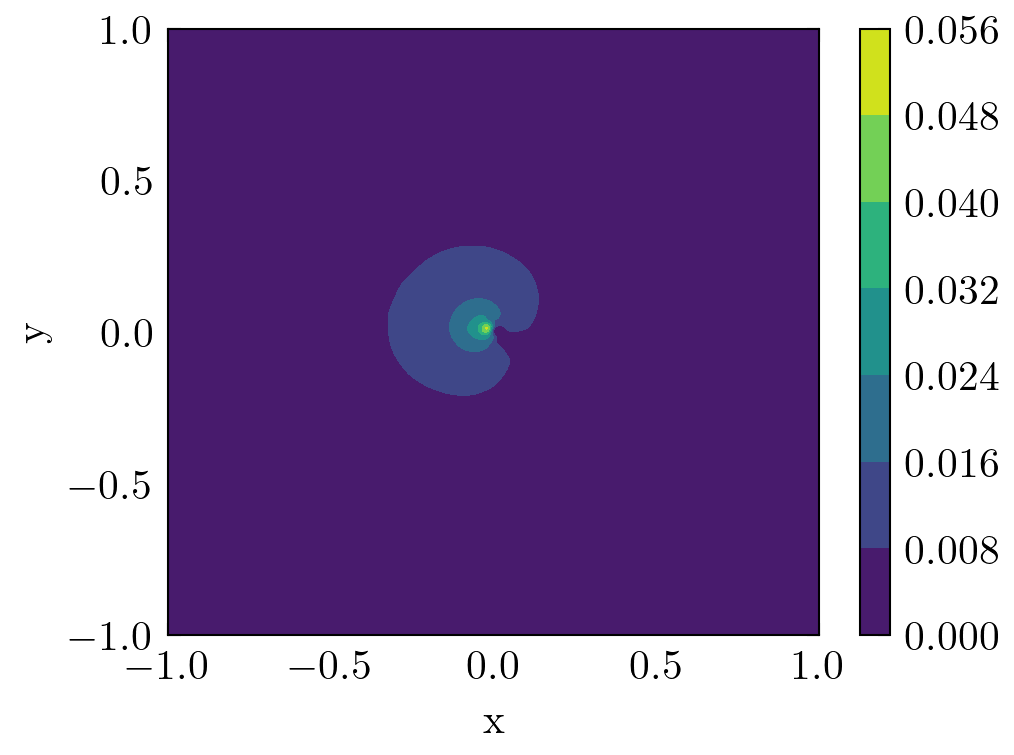}}
\caption{The result of Halton sequence for the two-dimensional Poisson equation  with the solution Eq \eqref{eq:2d_Peaksolution}. (a) training points sampled by Halton sequence; (b) the approximate solution $u^{\text{Halton}}$; (c) the absolute error $\vert u^* - u^{\text{Halton}} \vert$.}
\label{fig:Peak2D_Halton}
\end{figure}

\begin{figure}[htbp]
\centering
\subfloat[sampling points]{\includegraphics[width = 0.27\textwidth]{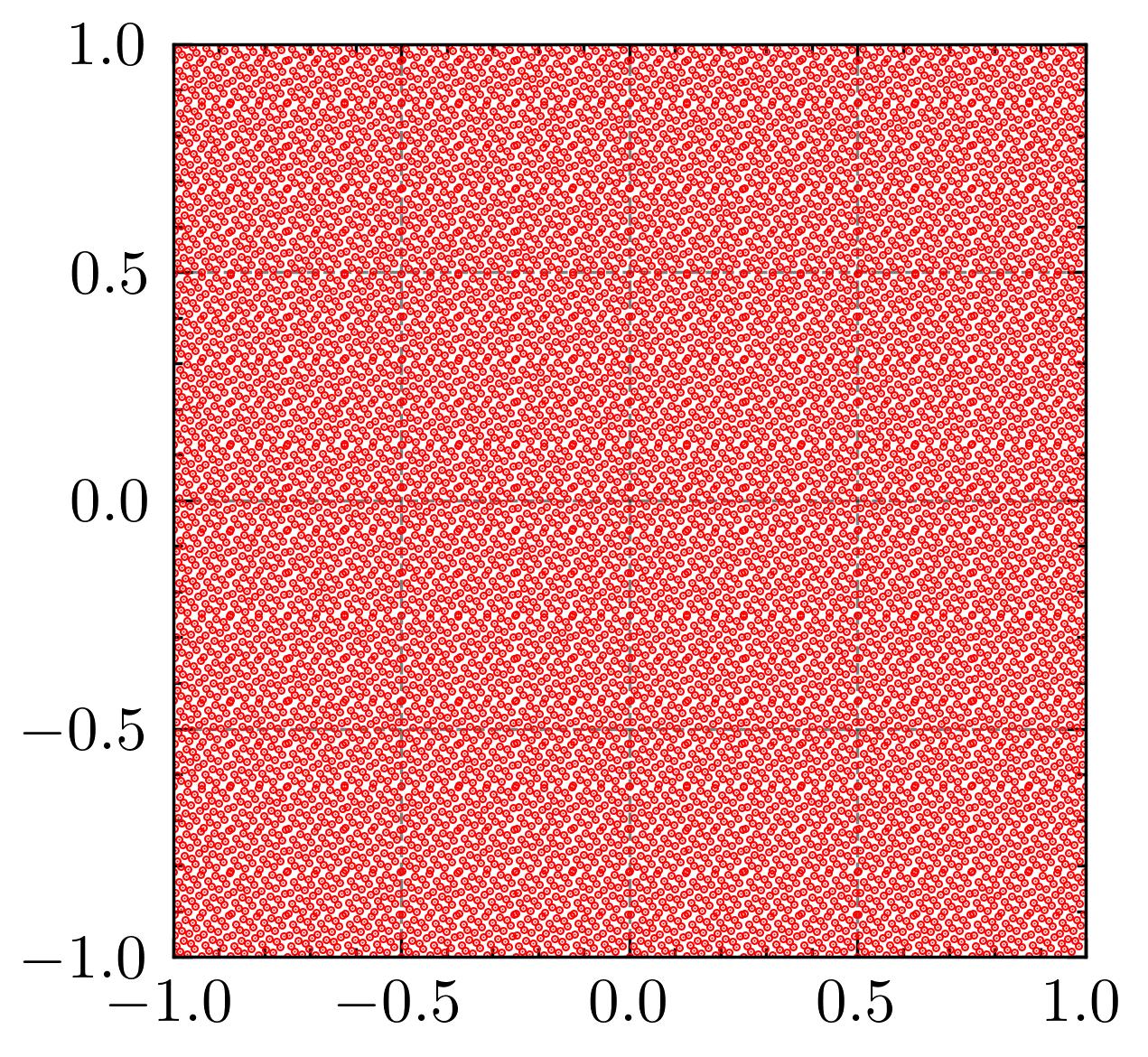}}
\subfloat[approximate solution]{\includegraphics[width = 0.31\textwidth]
{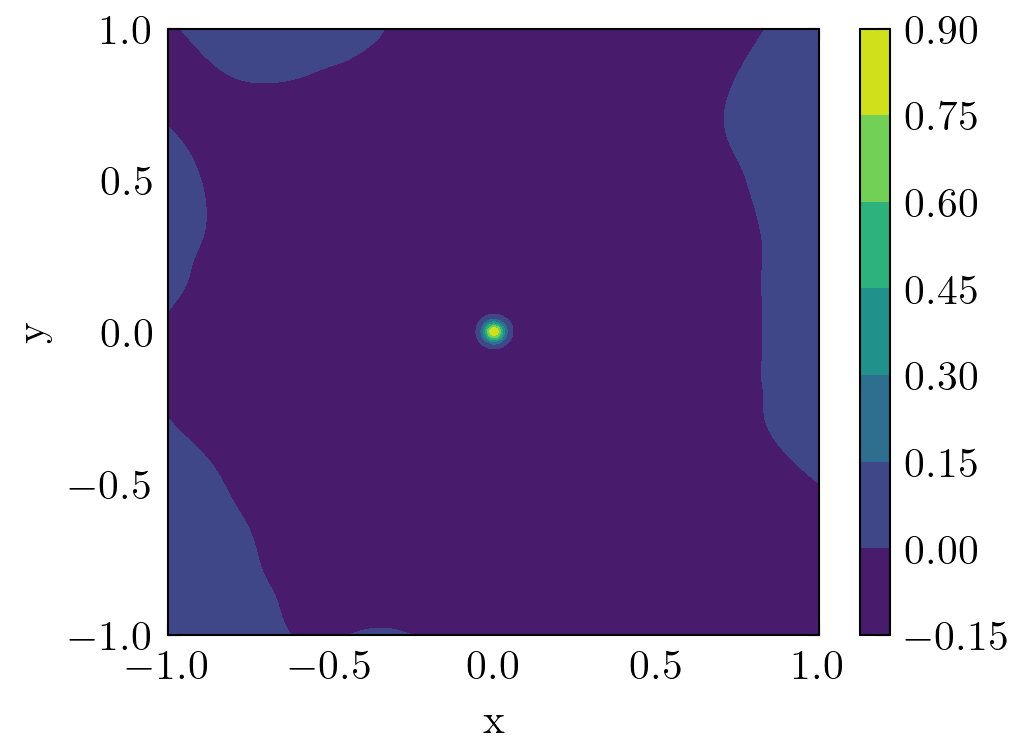}}
\subfloat[absolute error]{\includegraphics[width = 0.31\textwidth]{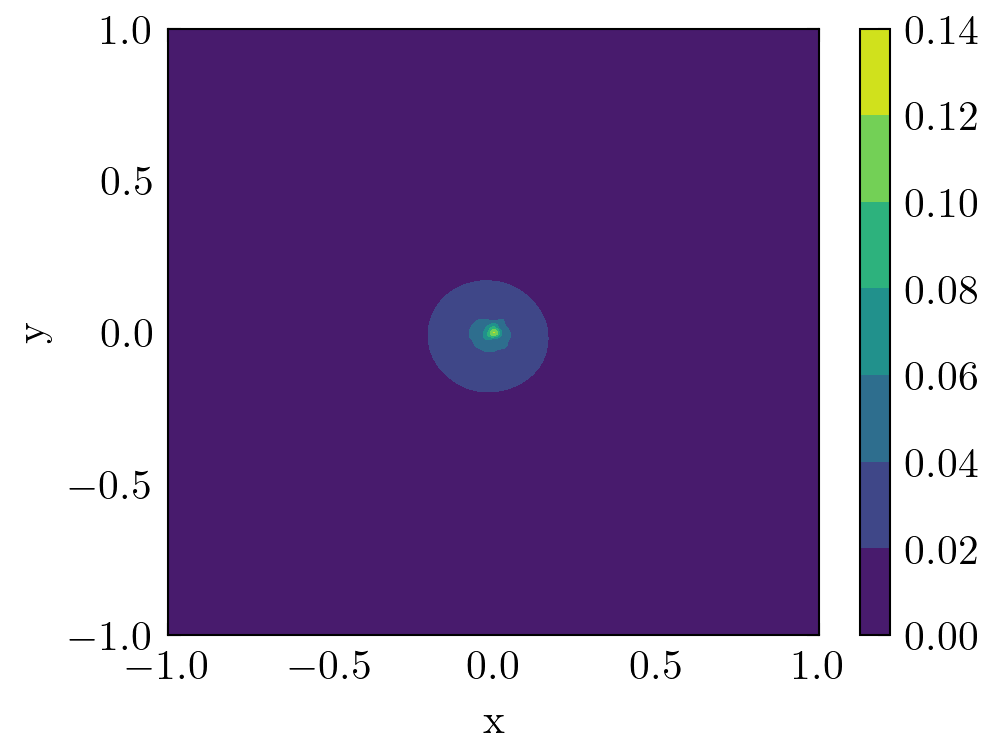}}
\caption{The result of Hammersley sequence for the two-dimensional Poisson equation  with the solution Eq \eqref{eq:2d_Peaksolution}. (a) training points sampled by Hammersley sequence; (b) the approximate solution $u^{\text{Hammersley}}$; (c) the absolute error $\vert u^* - u^{\text{Hammersley}} \vert$.}
\label{fig:Peak2D_Hammersley}
\end{figure}

\begin{figure}[htbp]
\centering
\subfloat[Sampling points]{\includegraphics[width = 0.27\textwidth]{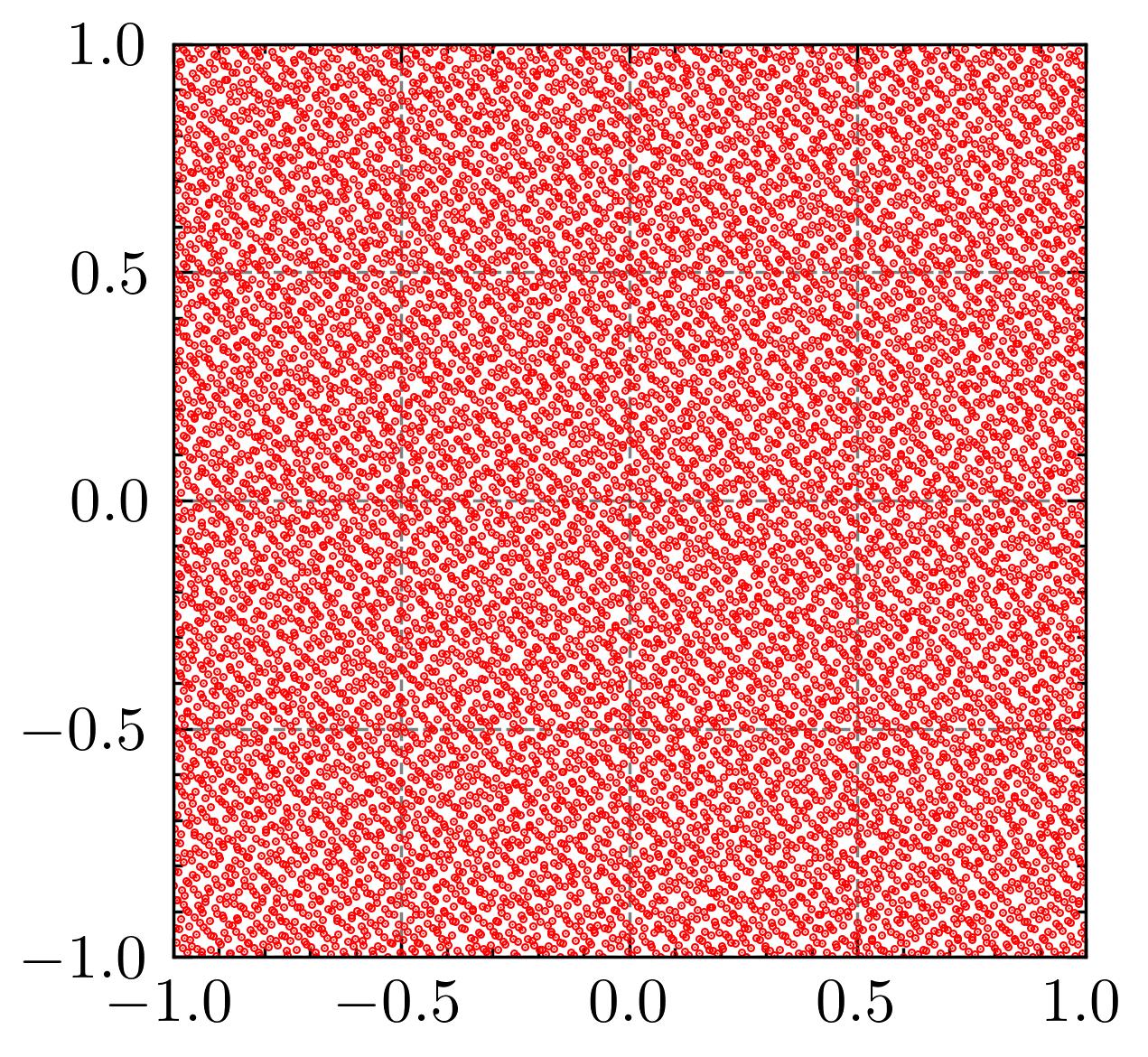}}
\subfloat[approximate solution]{\includegraphics[width = 0.31\textwidth]
{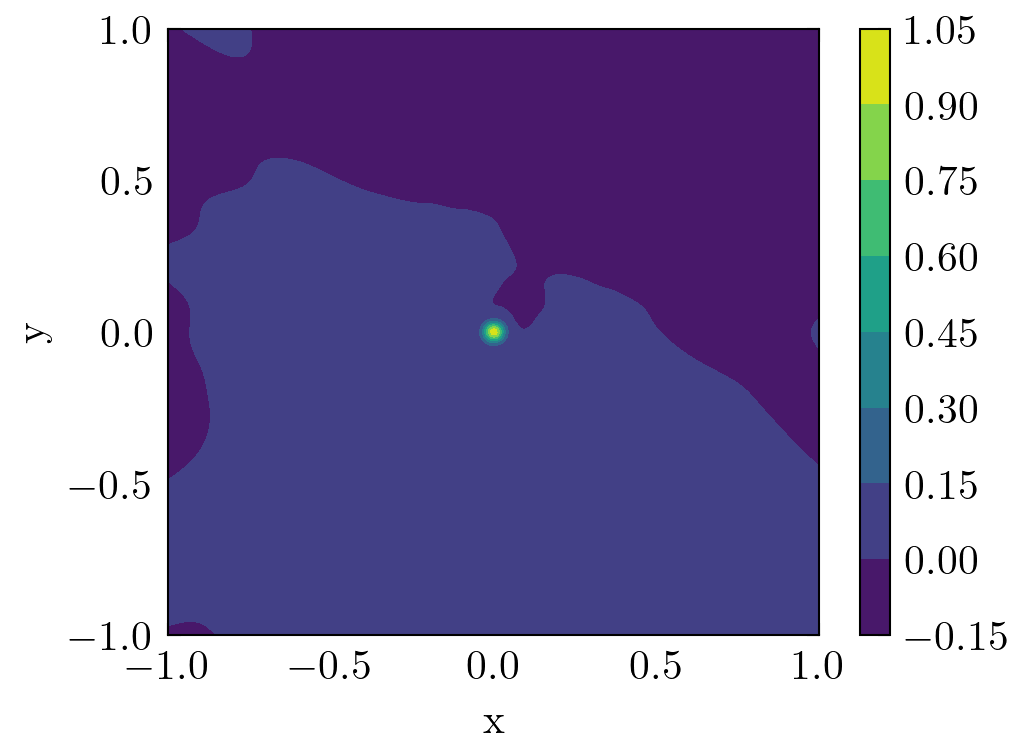}}
\subfloat[absolute error]{\includegraphics[width = 0.31\textwidth]{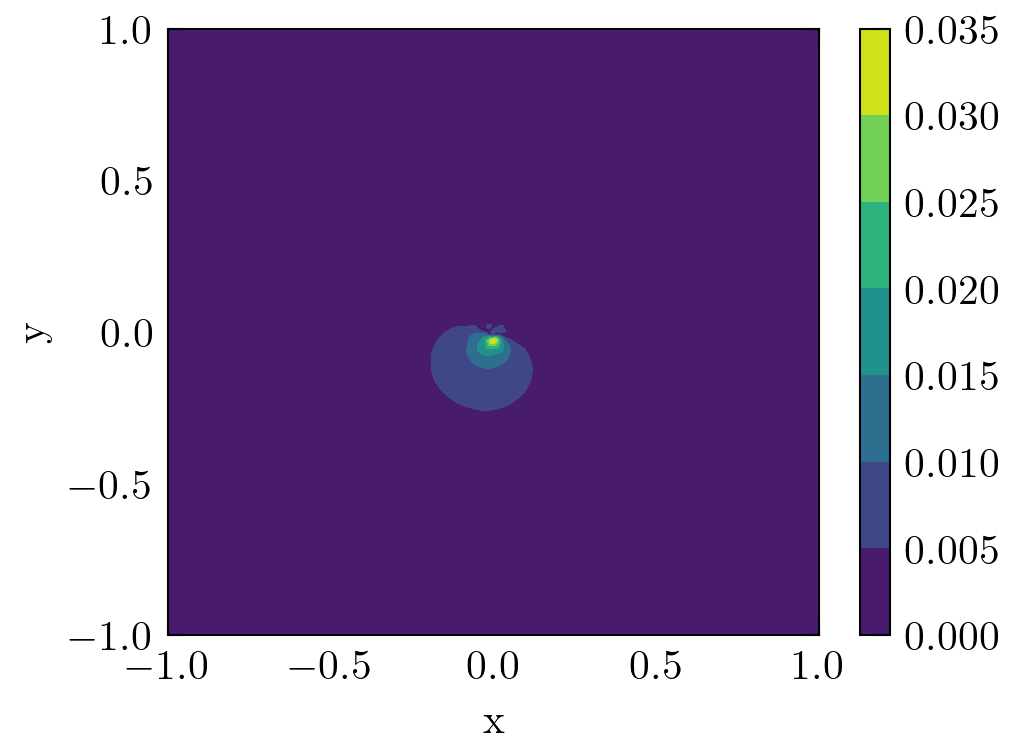}}
\caption{The result of Sobol sequence for the two-dimensional Poisson equation  with the solution Eq \eqref{eq:2d_Peaksolution}. (a) training points sampled by Sobol sequence; (b) the approximate solution $u^{\text{Sobol}}$; (c) the absolute error $\vert u^* - u^{\text{Sobol}} \vert$.}
\label{fig:Peak2D_Sobol}
\end{figure}

\begin{figure}[htbp]
\centering
\subfloat[Sampling points]{\includegraphics[width = 0.27\textwidth]{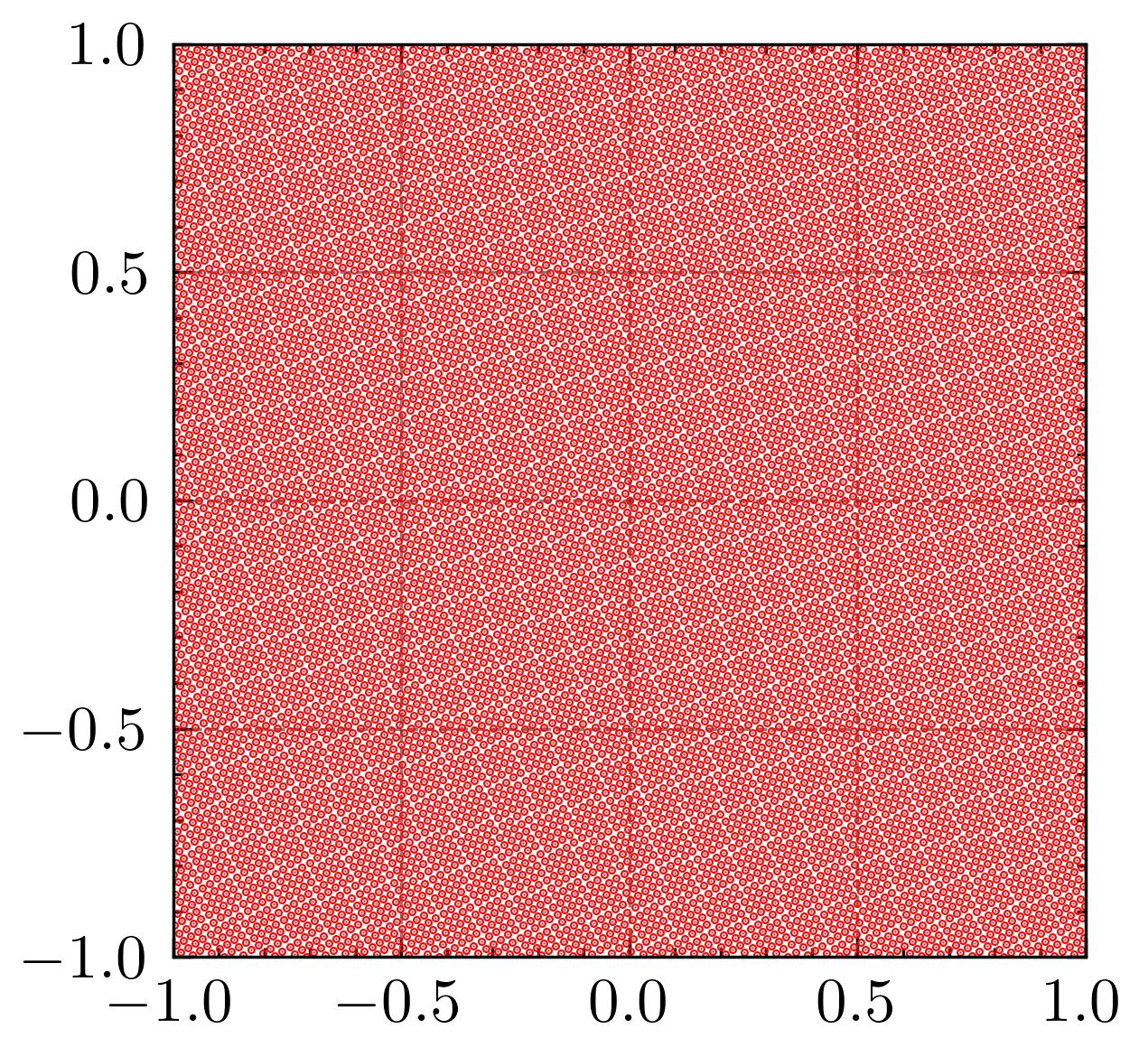}}
\subfloat[approximate solution]{\includegraphics[width = 0.31\textwidth]
{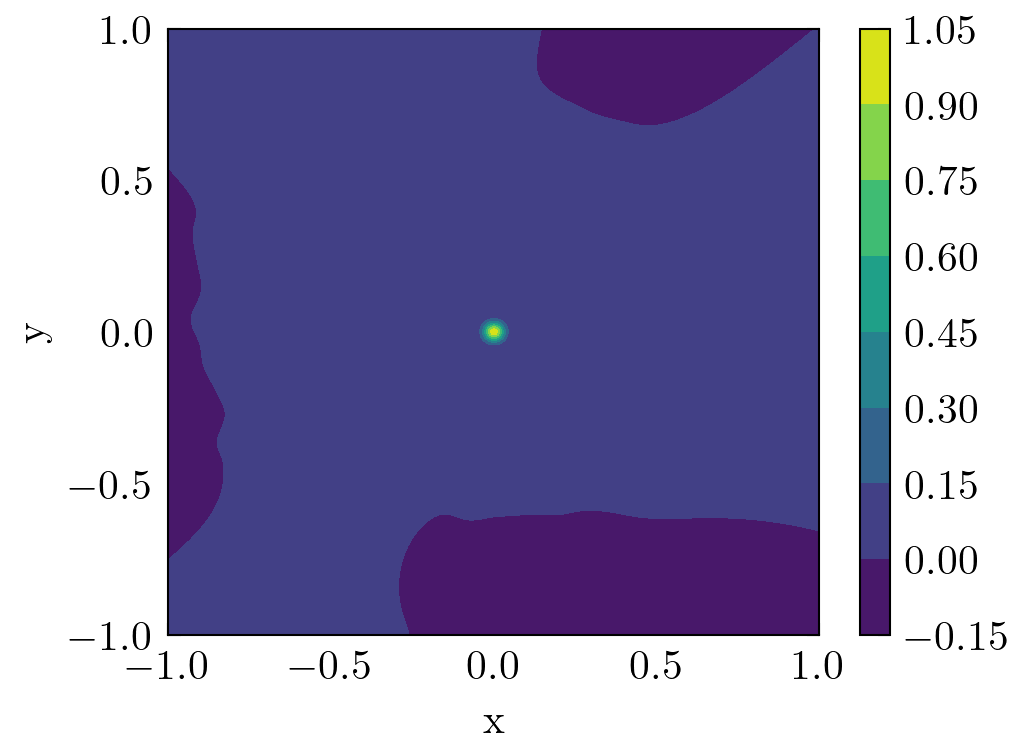}}
\subfloat[absolute error]{\includegraphics[width = 0.31\textwidth]{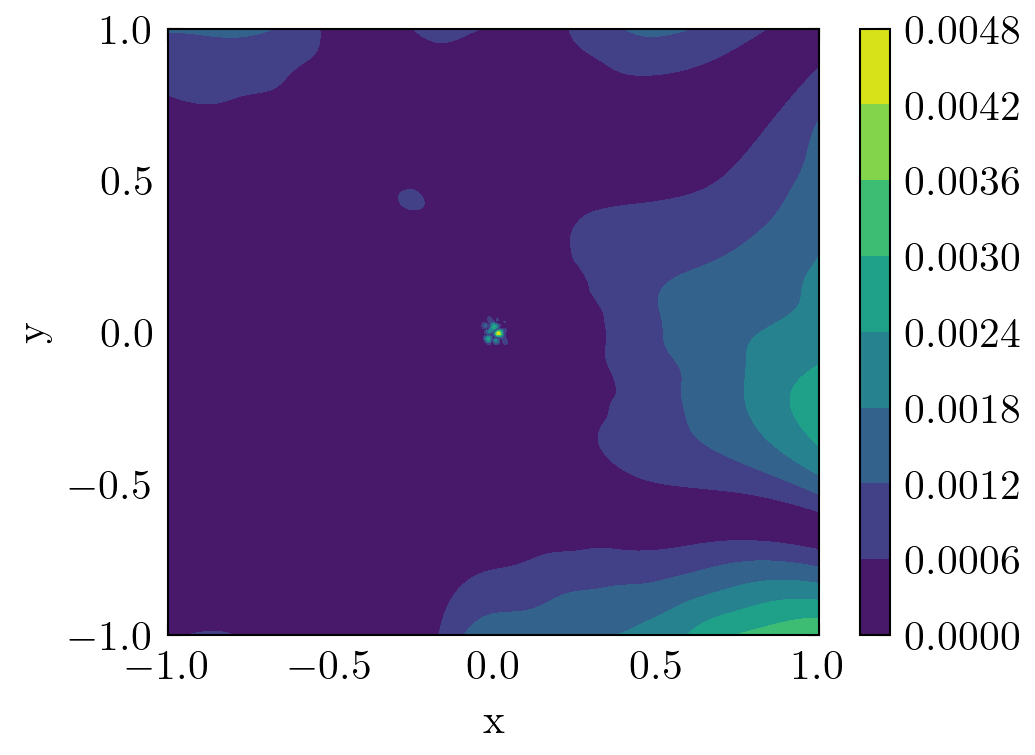}}
\caption{The result of GLP sampling  for the two-dimensional Poisson equation  with the solution Eq \eqref{eq:2d_Peaksolution}. (a) training points sampled by GLP sampling; (b) the approximate solution $u^{\text{GLP}}$; (c) the absolute error $\vert u^* - u^{\text{GLP}} \vert$.}
\label{fig:Peak2D_NTM}
\end{figure}

\begin{figure}[htbp]
\centering
\subfloat[performance of $e_\infty(u)$ ]{\includegraphics[width = 0.40\textwidth]{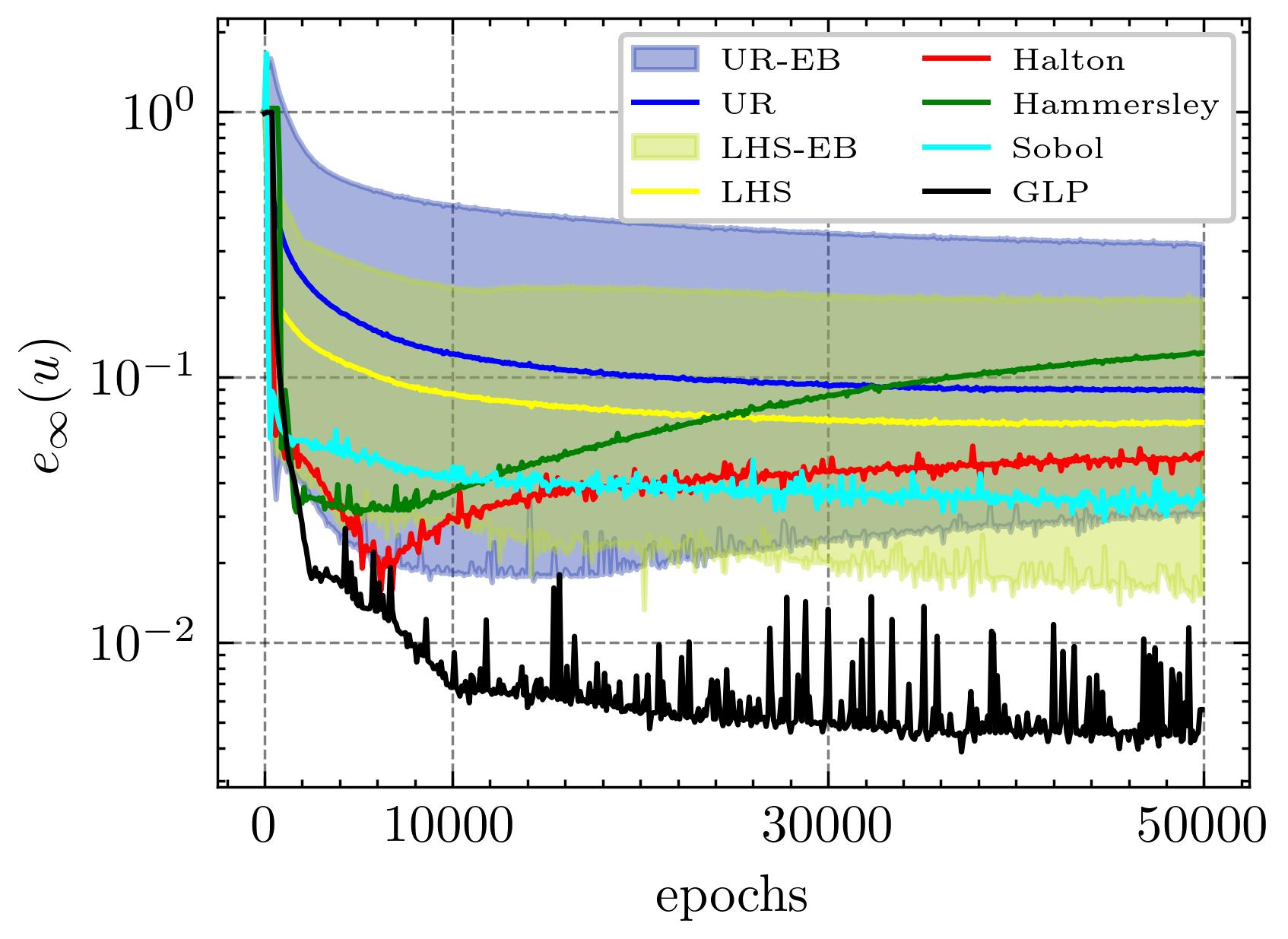}} \quad \quad \quad
\subfloat[performance of $e_2(u)$ ]{\includegraphics[width = 0.40\textwidth]
{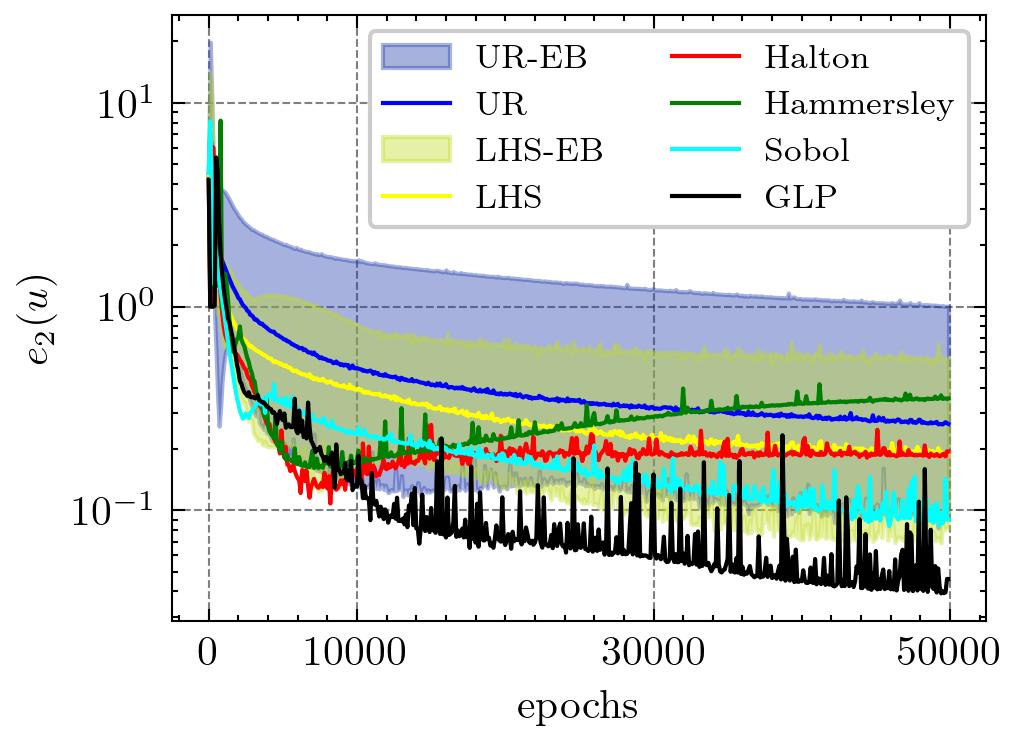}}
\caption{The performance of errors for the two-dimensional Poisson equation  with the solution Eq \eqref{eq:2d_Peaksolution}. (a) the relative error $e_\infty(u)$ with different training epochs; (b) the relative error $e_2(u)$ with different training epochs.}
\label{fig:Peak2D_OnePeak_ErrorEpochs}
\end{figure}

The numerical results of uniform random sampling are given in Fig \ref{fig:Peak2D_UniformRandom}. Training points sampled by uniform random sampling method of the solution Eq \eqref{eq:2d_Peaksolution} are shown in Fig \ref{fig:Peak2D_UniformRandom}(a), the approximate solution of PINNs using uniform random sampling is shown in Fig \ref{fig:Peak2D_UniformRandom}(b), and the absolute error is presented in Fig \ref{fig:Peak2D_UniformRandom}(c). Similarly, the numerical results for LHS mehtod, Halton sequence, Hammersley sequence, Sobel sequence and GLP sampling are given in Fig \ref{fig:Peak2D_LHS}--
\ref{fig:Peak2D_NTM}, respectively.

Fig \ref{fig:Peak2D_OnePeak_ErrorEpochs} illustrates the performance of six distinct methods in terms of their relative errors on the test set throughout the training process. \textcolor{black}{For the random sampling methods, uniform random sampling and LHS method, we randomly select 10 seeds, plot the error bands based on the maximum and minimum errors, and use the average of ten calculations as the results for these two methods.}
\textcolor{black}{As can be seen from Fig  \ref{fig:Peak2D_OnePeak_ErrorEpochs}, the training performance of the Hammersley sequence is poor. However, in a similar comparative experiment (Fig \ref{fig:TwoPeaks2D_ErrorEpochs}), the Hammersley sequence shows good performance. This indicates that the Hammersley sequence is still effective. The author believes that the reason for the Hammersley sequence's error increasing with training steps in Fig \ref{fig:Peak2D_OnePeak_ErrorEpochs} is due to the Adam optimization algorithm. One possible explanation is that during training, the optimization algorithm causes the parameters to enter the neighborhood of a local minima of the loss function and makes it difficult to escape.}
Additionally, Table \ref{tab:Peak_2D_Results} provides a comparative analysis of the errors obtained using these various methods. Upon examining these results, it is evident that the GLP sampling yields superior outcomes compared to the other methods.

\begin{table}[h]
\scriptsize
\centering
\caption{
Comparison of errors using different methods for two-dimensional Poisson equation with one peak of Eq \eqref{eq:2d_Peaksolution}.
}
\setlength{\tabcolsep}{3.mm}{
\begin{tabular}{|c|c|c|c|c|c|c|}
\hline\noalign{\smallskip}
relative      &  uniform& LHS & Halton & Hammersley& Sobol&  GLP\\
error        &  random &method& sequence& sequence& sequence & set\\
\hline
$e_\infty(u)$  & $3.243 \times 10^{-2}$  & $2.780 \times 10^{-2}$  &  $4.963 \times 10^{-2}$ & $1.223 \times 10^{-1}$&$3.412 \times 10^{-2}$&$4.475 \times 10^{-3}$\\
\hline
$e_2(u)$  & $1.051 \times 10^{-1}$  & $7.921 \times 10^{-2}$&
$1.858 \times 10^{-1}$  & $3.523 \times 10^{-1}$&$8.824 \times 10^{-2}$&$4.045 \times 10^{-2}$\\
\hline
\end{tabular}
}
\label{tab:Peak_2D_Results} 
\end{table}

\textcolor{black}{
To verify the robustness of GLP sampling under varying sampling budgets, we design the following comparative experiments: under the same conditions as the aforementioned experiments, we develop different sampling strategies based on the number of sampling points. We observe the performance of GLP sampling and uniform random sampling when $N_r$ takes 987, 4181, 6765, and 10946 respectively.}

\begin{figure}[htbp]
\centering
\subfloat[performance of $e_\infty(u)$ ]{\includegraphics[width = 0.40\textwidth]{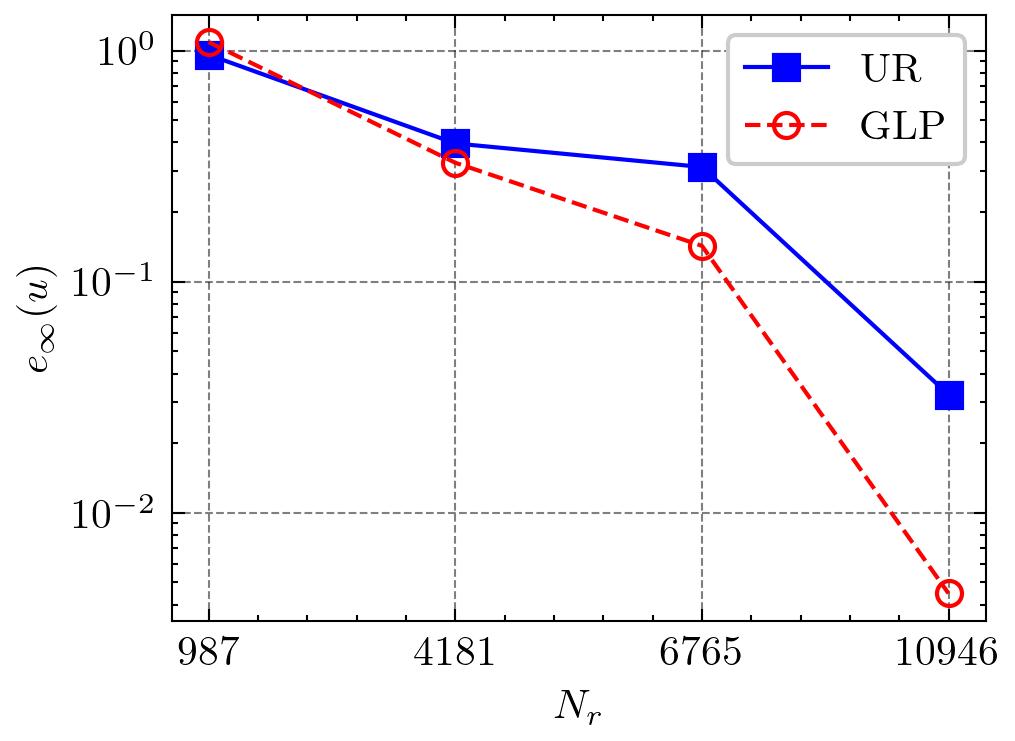}} \quad \quad \quad
\subfloat[performance of $e_2(u)$ ]{\includegraphics[width = 0.40\textwidth]
{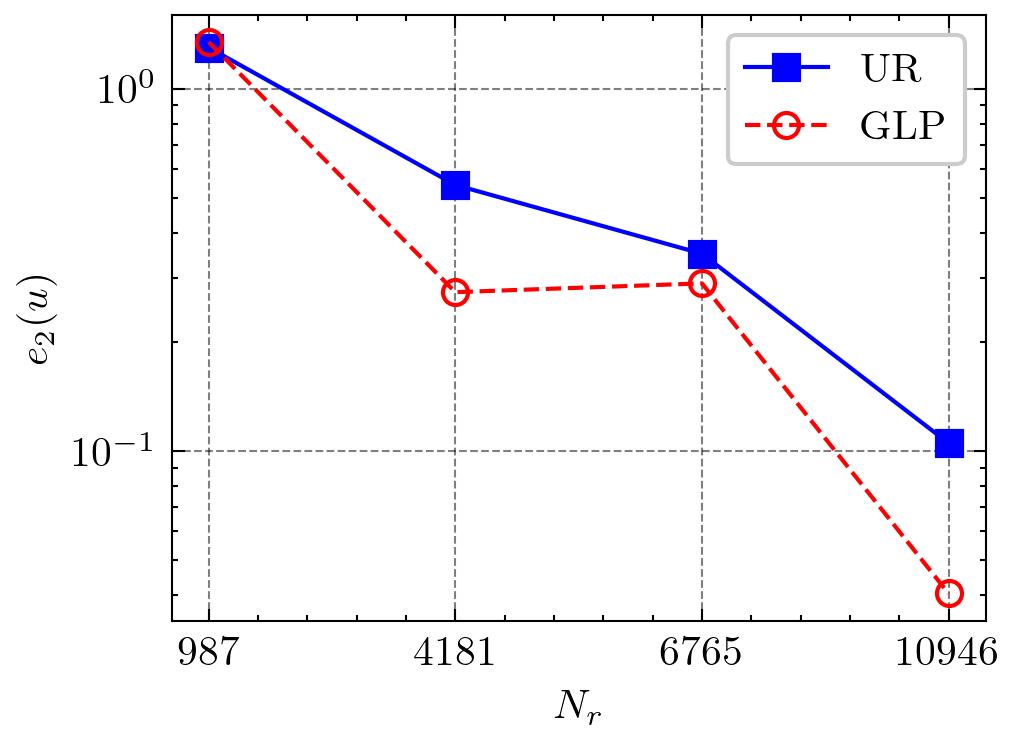}}
\caption{The performance of errors for the two-dimensional Poisson equation  with the solution Eq \eqref{eq:2d_Peaksolution}. (a) the relative error $e_\infty(u)$ with different sampling budgets; (b) the relative error $e_2(u)$ with different sampling budgets.}
\label{fig:Peak2D_OnePeak_ErrorDiffPoints}
\end{figure}

Fig \ref{fig:Peak2D_OnePeak_ErrorDiffPoints} presents the relative error performance. As the number of samples increases, the error of both sampling methods shows a downward trend. However, GLP sampling maintains a distinct advantage over uniform random sampling, which is consistent with our theoretical findings.
\subsubsection{Two-dimensional Poisson equation with two peaks}
\label{sec:TwoPeaks_2D}

Consider the Poisson equation Eq \eqref{eq:2d_Poisson} in previous subsection,
the exact solution which has two peaks is chosen as 
\begin{equation}
    \label{eq:TwoPeaks2d_solution}
    \hspace{-0.3cm}
    \begin{array}{r@{}l}
        \begin{aligned}
            u = e^{-1000 \left( x^2+(y-0.5)^2\right)}+ e^{-1000\left(x^2+(y+0.5)^2\right)}.
        \end{aligned}
    \end{array}
\end{equation}
\textcolor{black}{
In order to more comprehensively explore the advantages of GLP points, in this section, we studied the case of non-rectangular computing regions, that is, $\Omega = \left\{ (x,y)|x^2+y^2 < 1 \right\}$.}

\textcolor{black}{
To obtain good lattice points within a circular region, we used a coordinate transformation (Eq \eqref{eq:TwoPeaks2d_transformation}) that preserves the low discrepancy of good lattice points from a rectangular region, as noted in Reference \cite{fang1993number}. 
\begin{equation}
    \label{eq:TwoPeaks2d_transformation}
    \hspace{-0.3cm}
    \begin{array}{r@{}l}
        \begin{aligned}
            x_i &= \sqrt{r_i}cos(2\pi\theta_i) \\
            y_i &= \sqrt{r_i}sin(2\pi\theta_i),
        \end{aligned}
    \end{array}
\end{equation}
where $\left \{(r_i,\theta_i)\right\}_{i=1}^{N_r}$ is a good lattice point set in $(0,1)^2$.}

We sample 10946 points in $\Omega$  and 400 points on $\partial \Omega$ as the training set and $400 \times 400$  points as the test set.  Fig \ref{fig:TwoPeaks2D_points} details 10946 residual training points from six different sampling methods.
After 50000 epochs of Adam training and 50000 epochs of LBFGS training, the absolute errors using different sampling methods are shown in Fig \ref{fig:TwoPeaks2D_DifferentErrors}. \textcolor{black}{To present the results of these six methods more clearly, we carry out numerical experiments with the same setup as in Fig \ref{fig:Peak2D_OnePeak_ErrorEpochs}. And we have provided the performance of relative errors during the training process in Fig \ref{fig:TwoPeaks2D_ErrorEpochs}, and the relative errors of the six methods in Table \ref{tab:TwoPeaks2D_Results}, which show that GLP sampling behaves best.}

\begin{figure}[htbp]
\centering
\subfloat[UR]{\includegraphics[width = 0.31\textwidth]{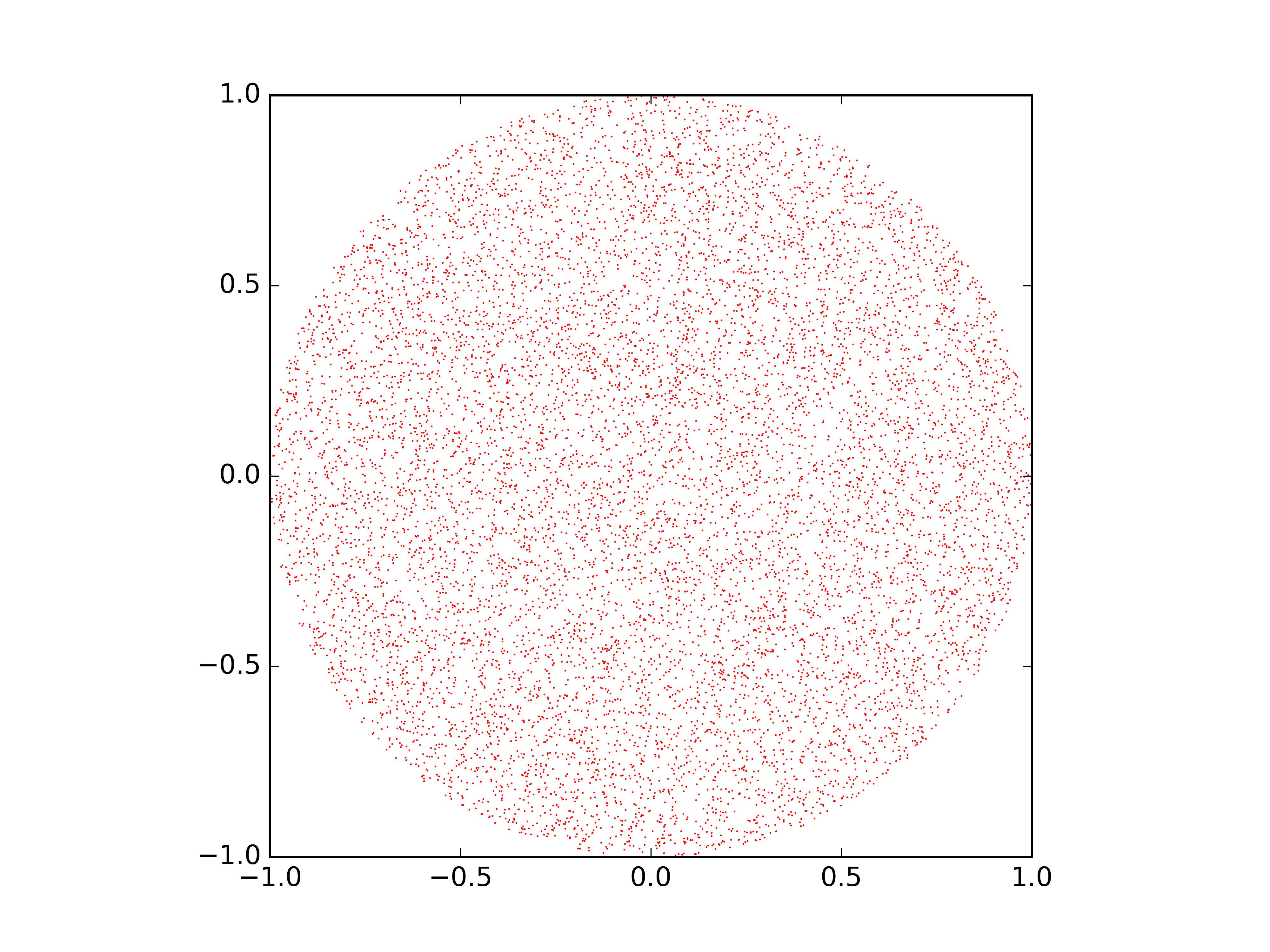}}
\subfloat[LHS]{\includegraphics[width = 0.31\textwidth]{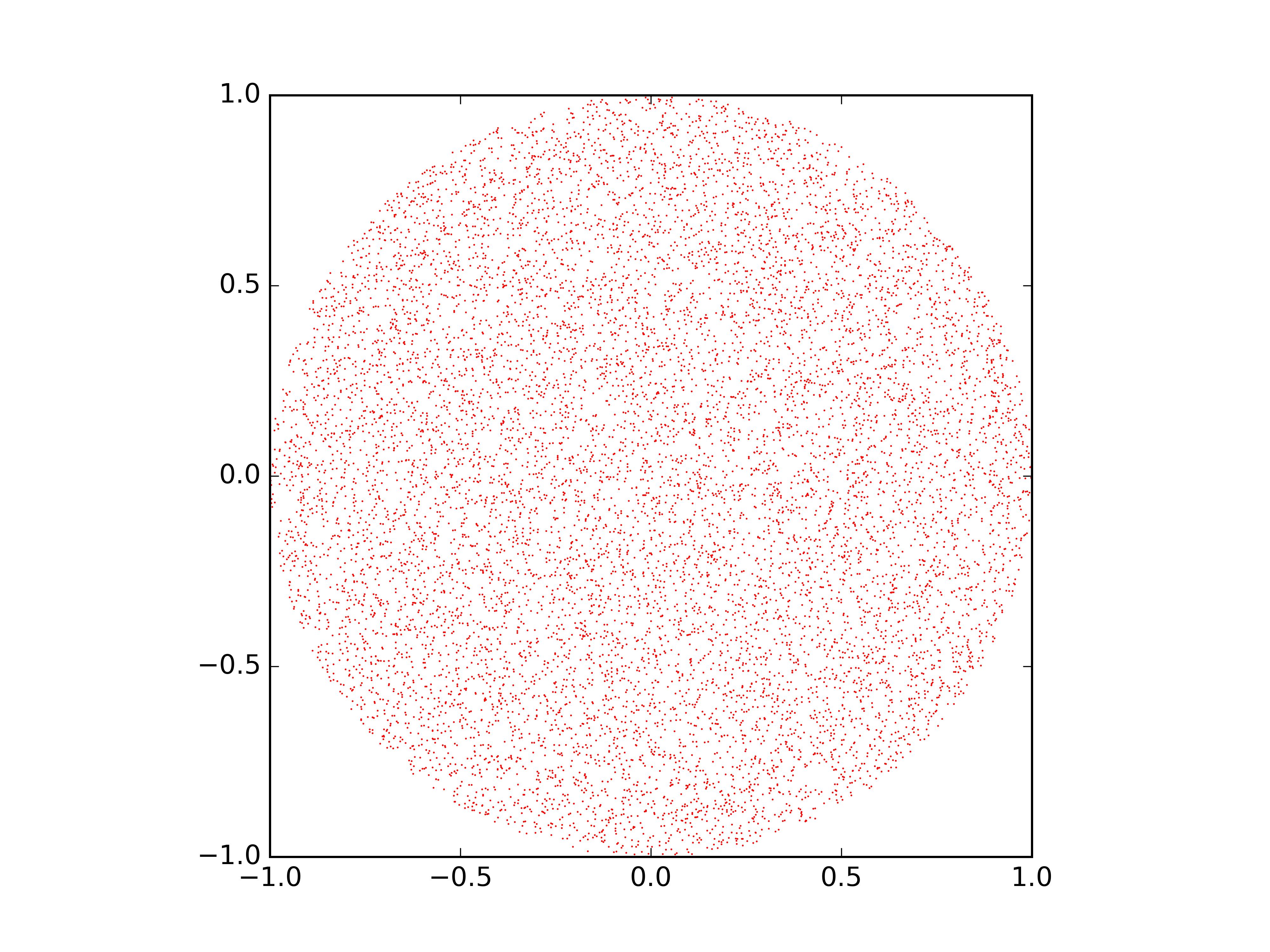}}
\subfloat[Halton]{\includegraphics[width = 0.31\textwidth]{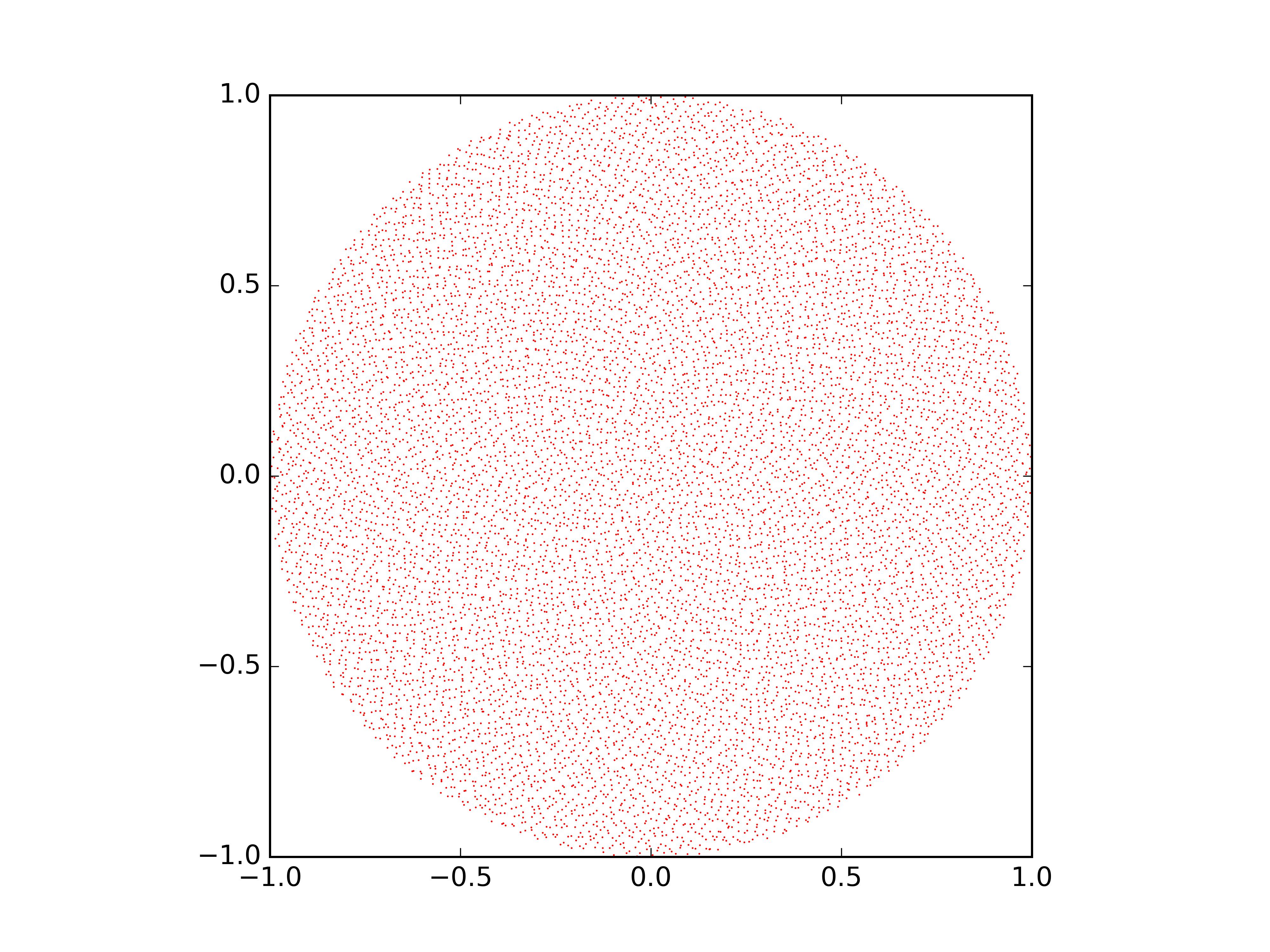}}
\\
\subfloat[Hammersley]{\includegraphics[width = 0.31\textwidth]{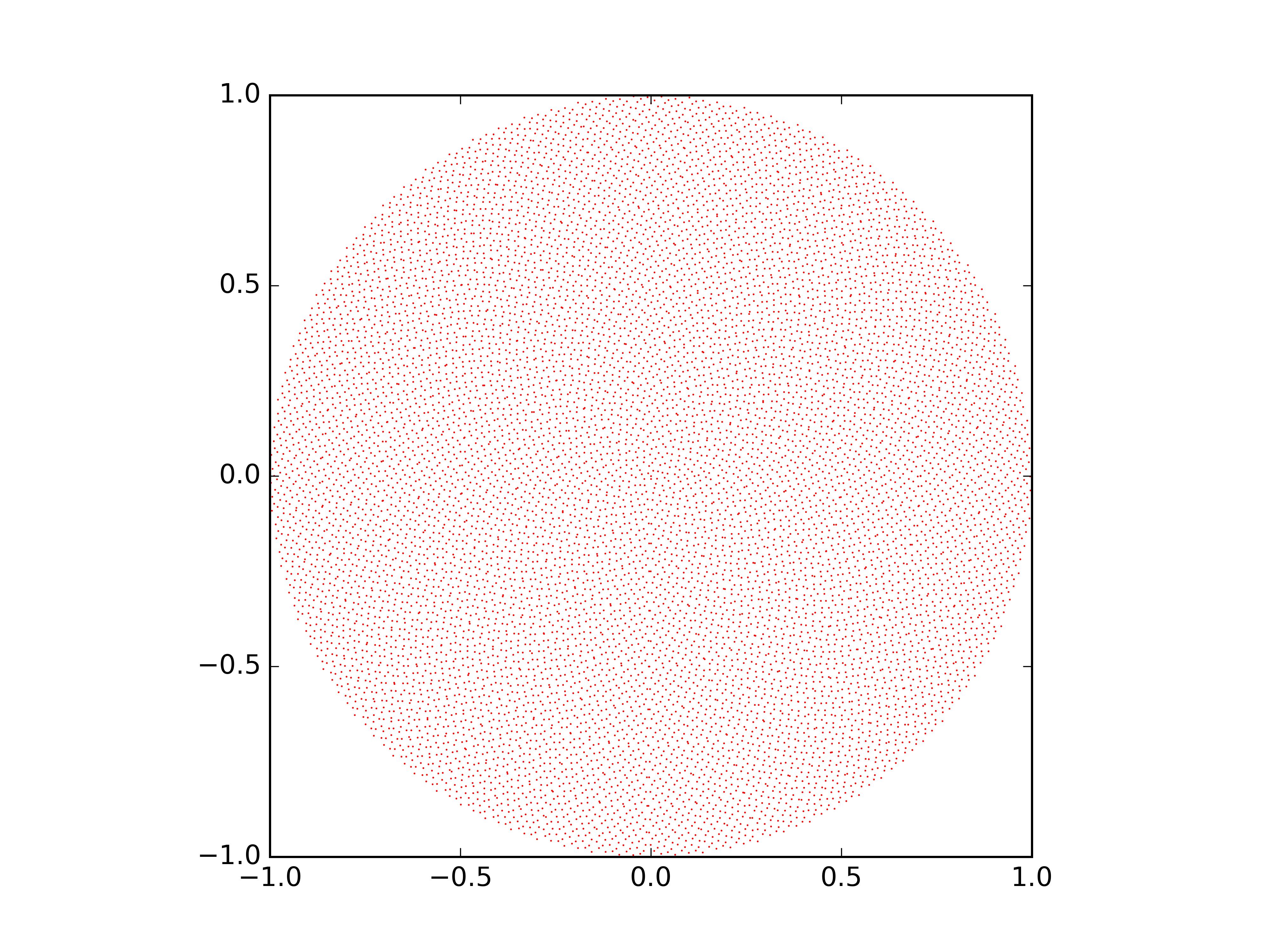}}
\subfloat[Sobol]{\includegraphics[width = 0.31\textwidth]{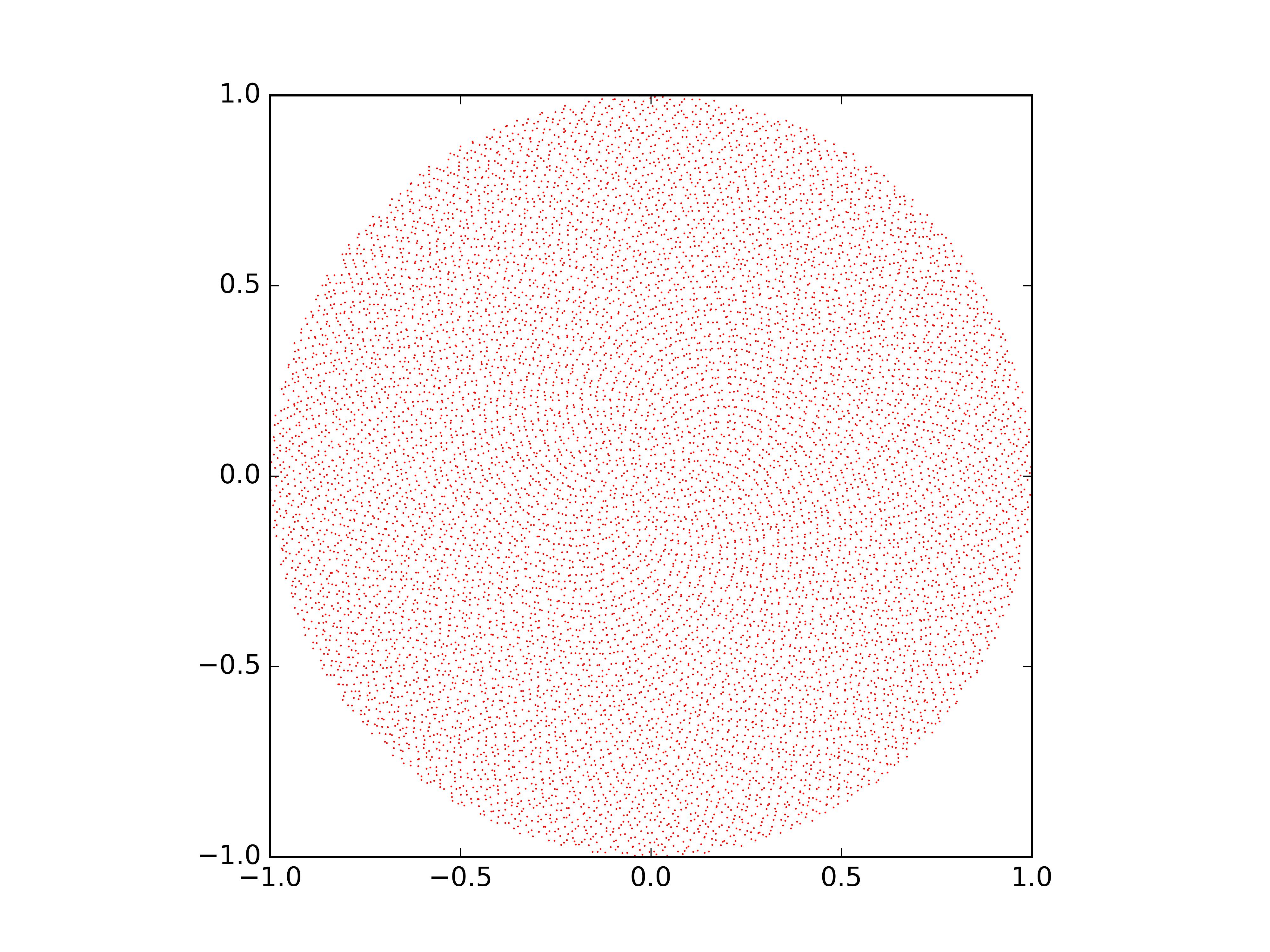}}
\subfloat[GLP]{\includegraphics[width = 0.31\textwidth]{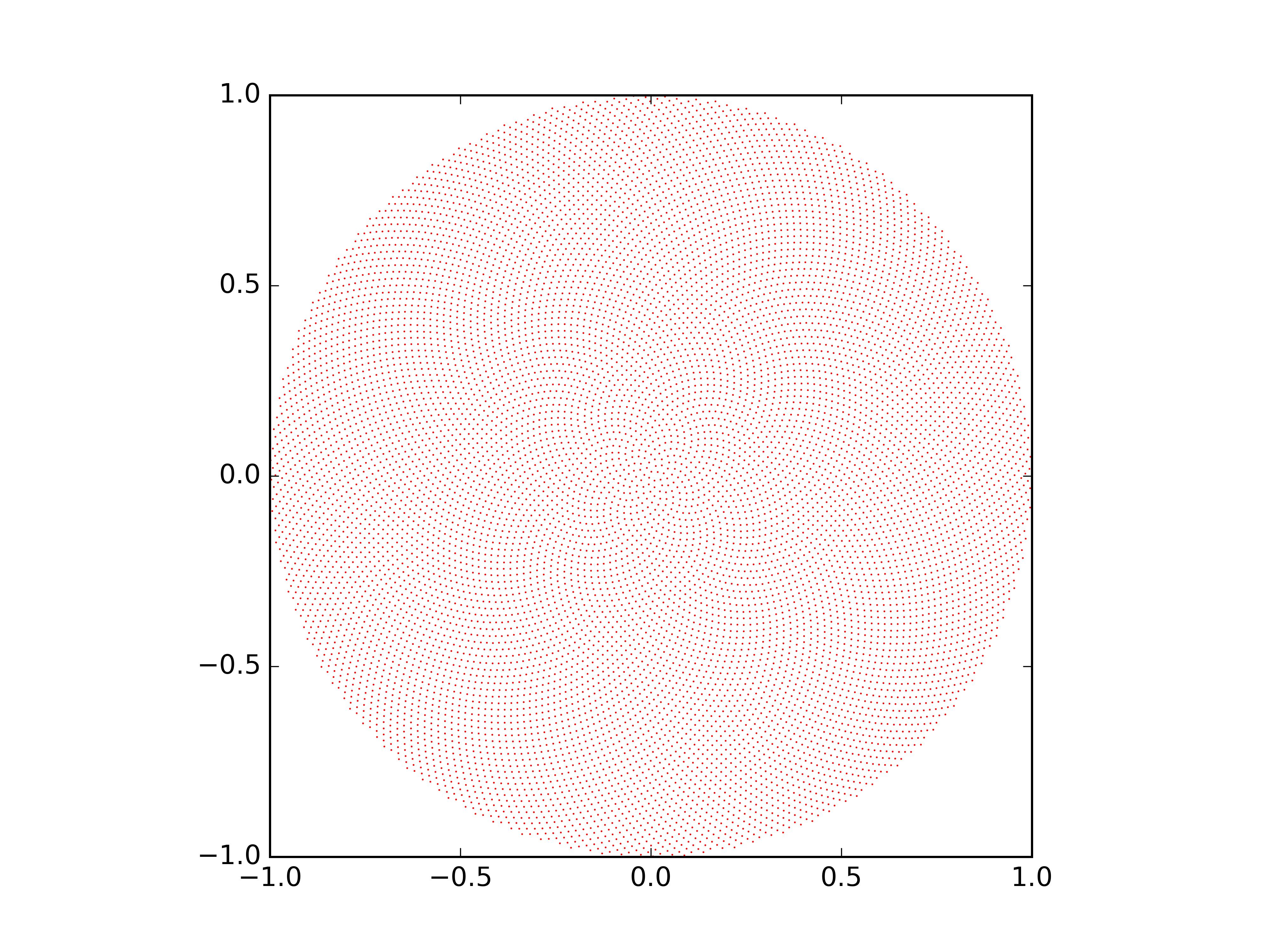}}
\caption{The sampling points of six different sampling methods for the two-dimensional Poisson equation with the exact solution  Eq \eqref{eq:TwoPeaks2d_solution}.}
\label{fig:TwoPeaks2D_points}
\end{figure}

\begin{figure}[htbp]
\centering
\subfloat[absolute error $\vert u^* - u^{\text{UR}} \vert$]{\includegraphics[width = 0.31\textwidth]{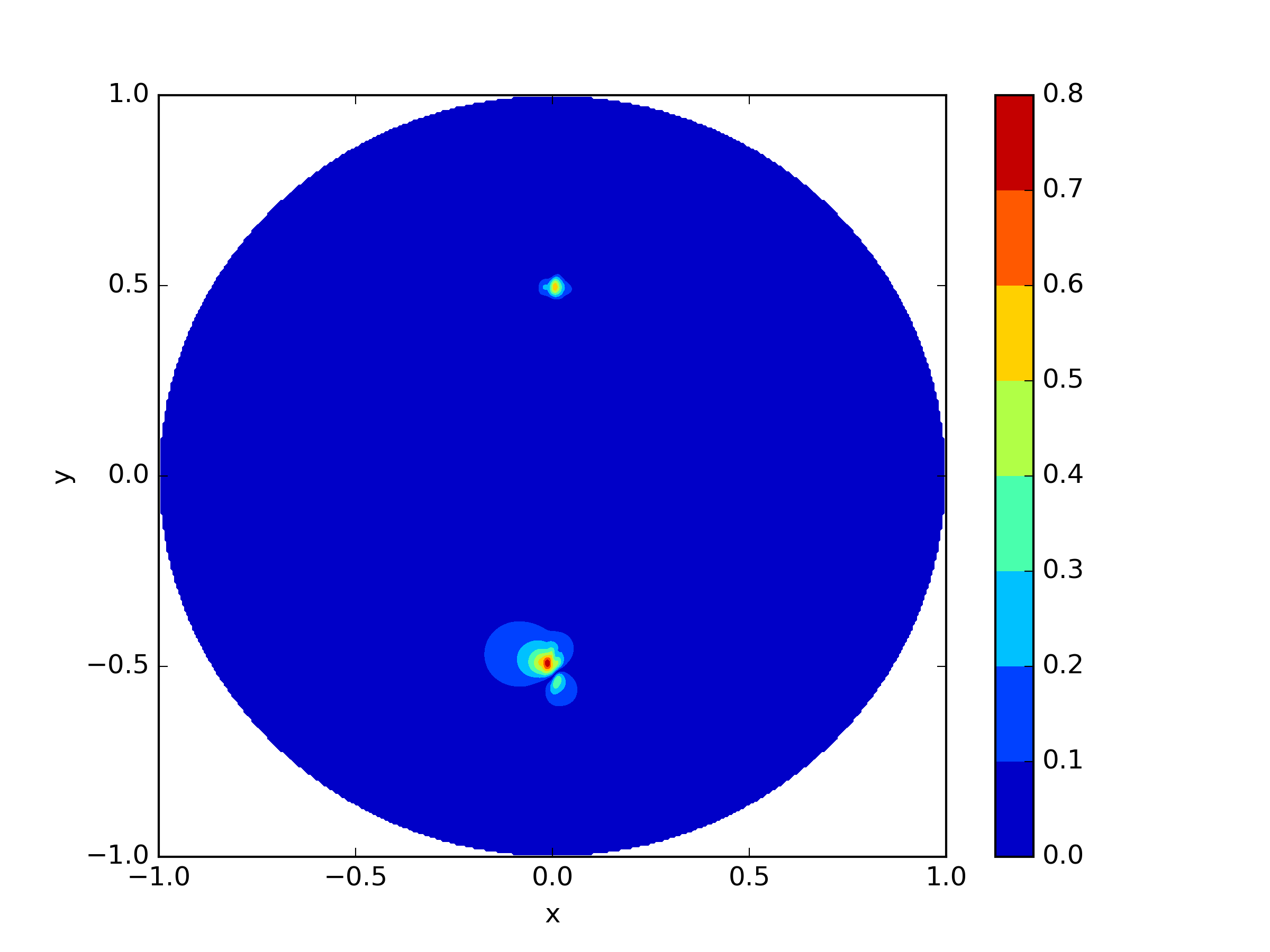}}
\subfloat[absolute error $\vert u^* - u^{\text{LHS}} \vert$]{\includegraphics[width = 0.31\textwidth]{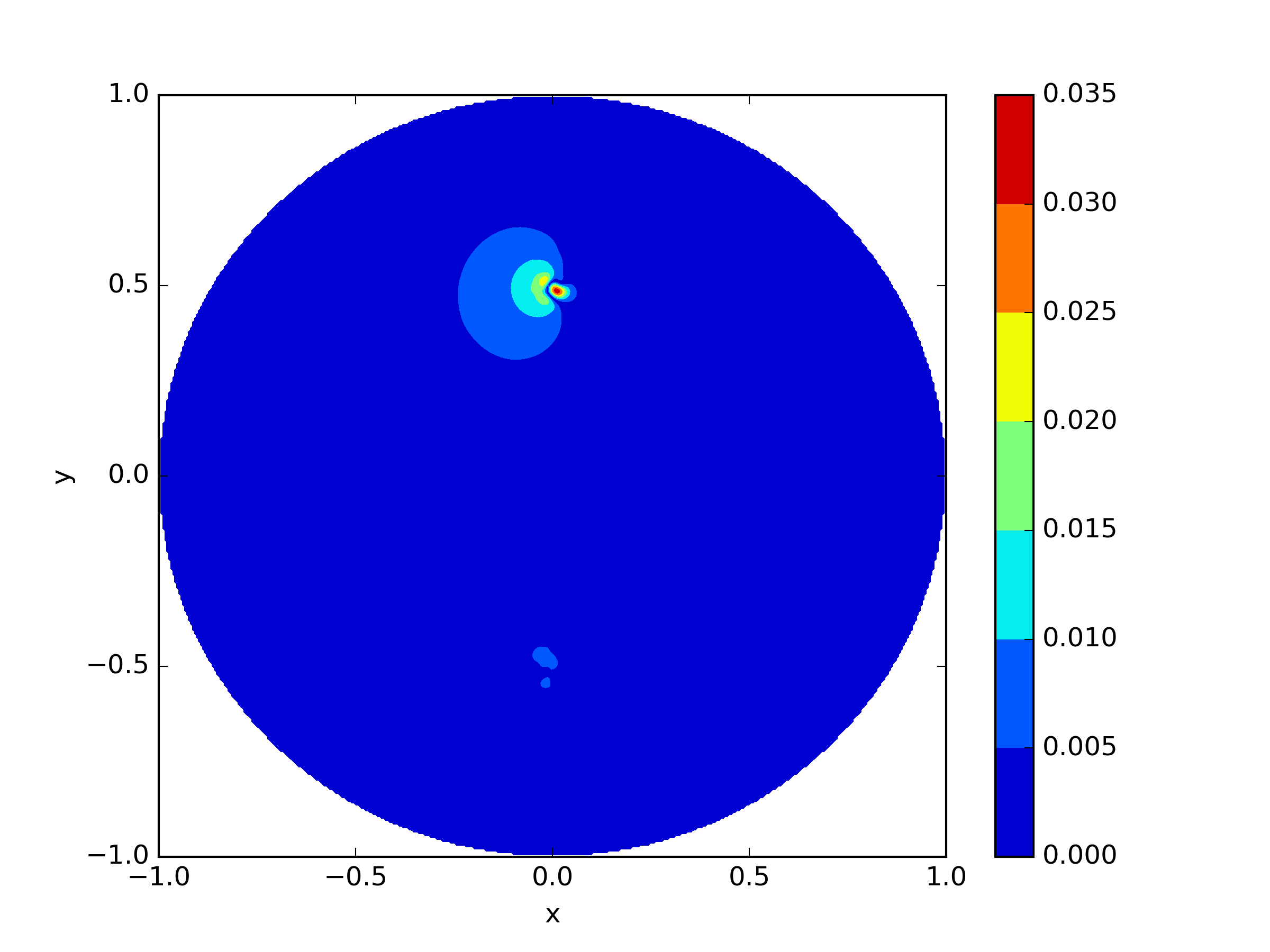}}
\subfloat[absolute error $\vert u^* - u^{\text{Halton}} \vert$]{\includegraphics[width = 0.31\textwidth]{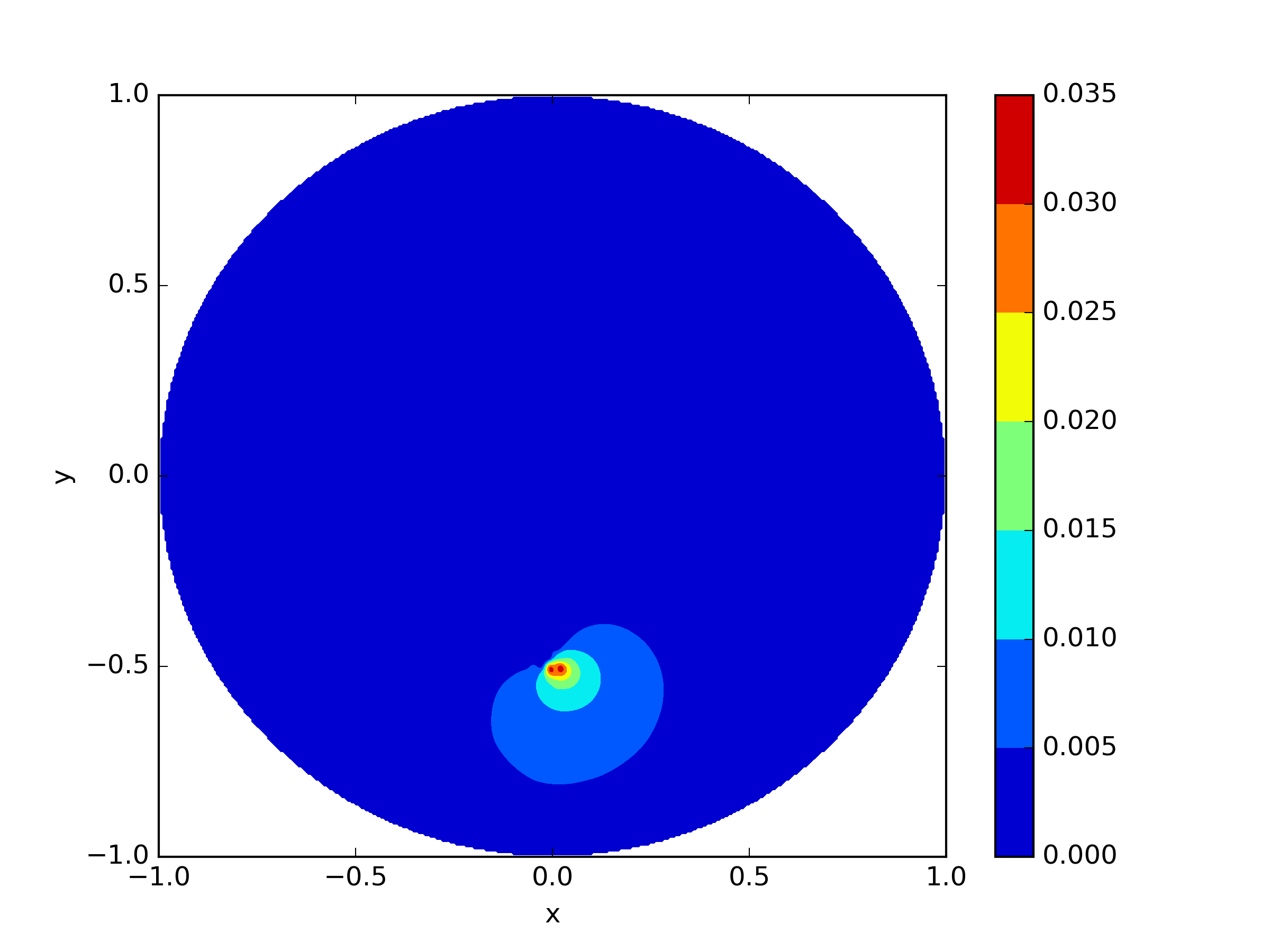}}
\\
\subfloat[absolute error $\vert u^* - u^{\text{Hammersley}} \vert$]{\includegraphics[width = 0.31\textwidth]{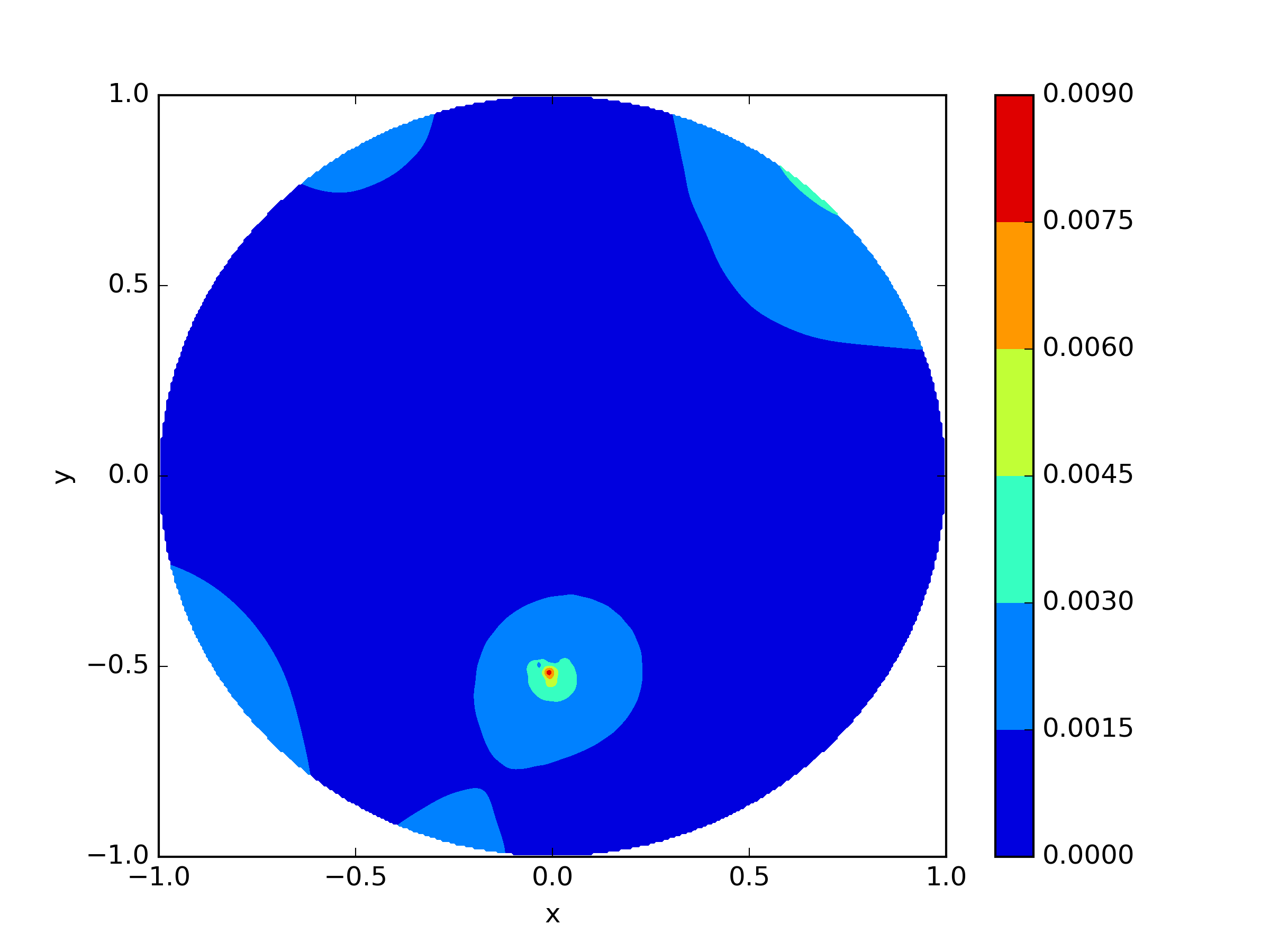}}
\subfloat[absolute error $\vert u^* - u^{\text{Sobol}} \vert$]{\includegraphics[width = 0.31\textwidth]{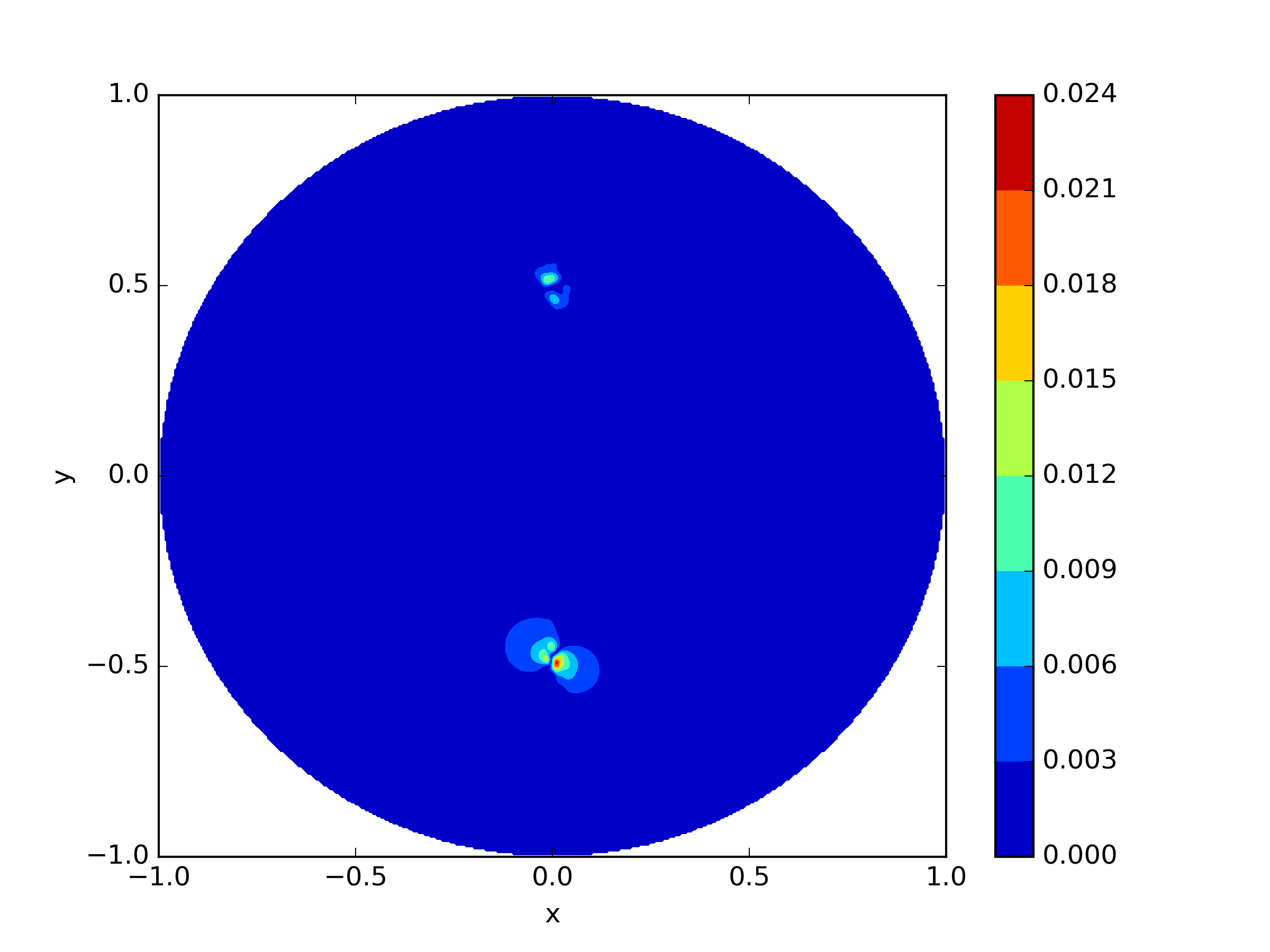}}
\subfloat[absolute error $\vert u^* - u^{\text{GLP}} \vert$]{\includegraphics[width = 0.31\textwidth]{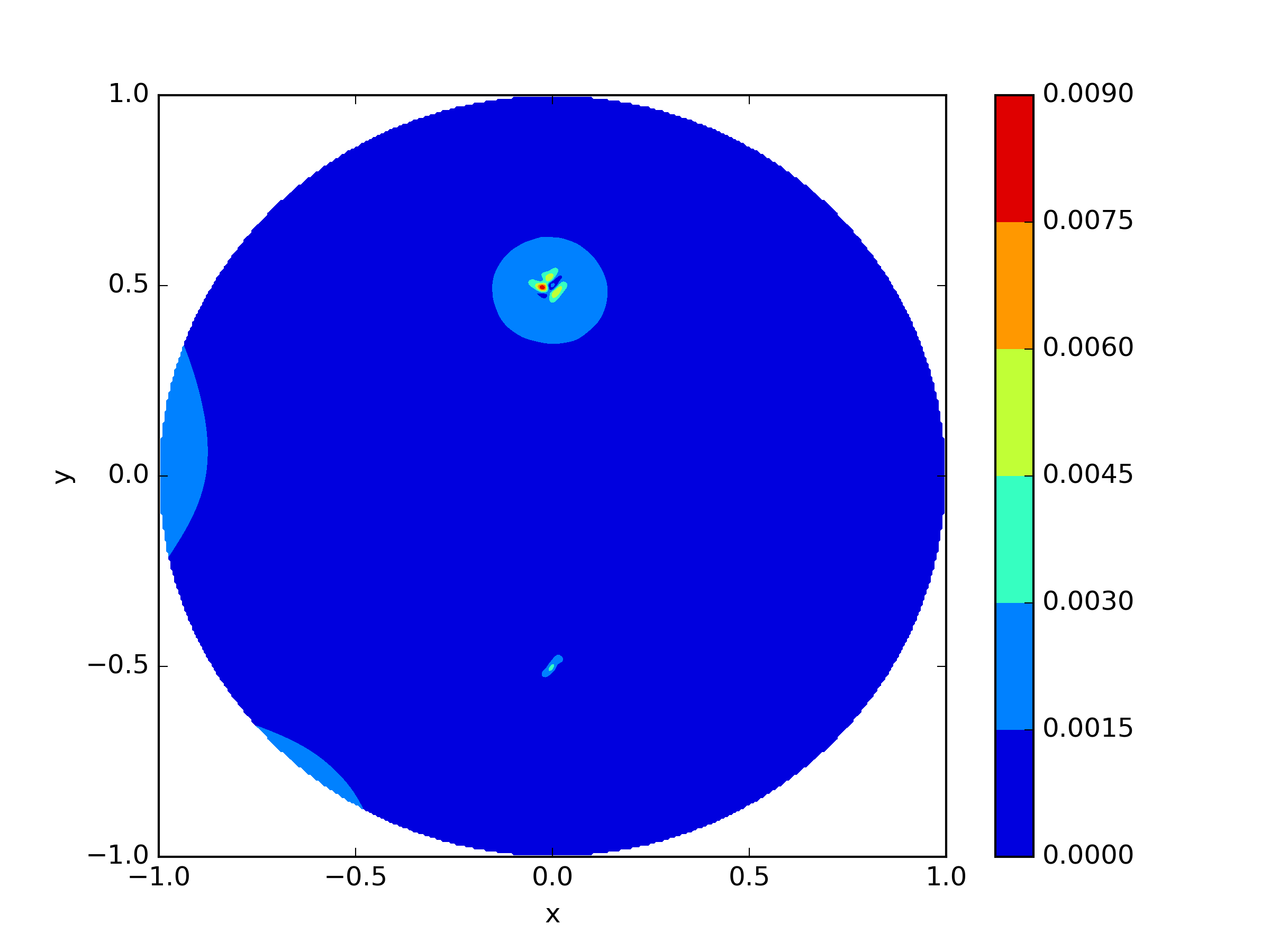}}
\caption{The absolute errors of six different sampling methods for the two-dimensional Poisson equation with the exact solution  Eq \eqref{eq:TwoPeaks2d_solution}.}
\label{fig:TwoPeaks2D_DifferentErrors}
\end{figure}

\begin{figure}[htbp]
\centering 
\subfloat[performance of $e_\infty(u)$]{\includegraphics[width = 0.40\textwidth]{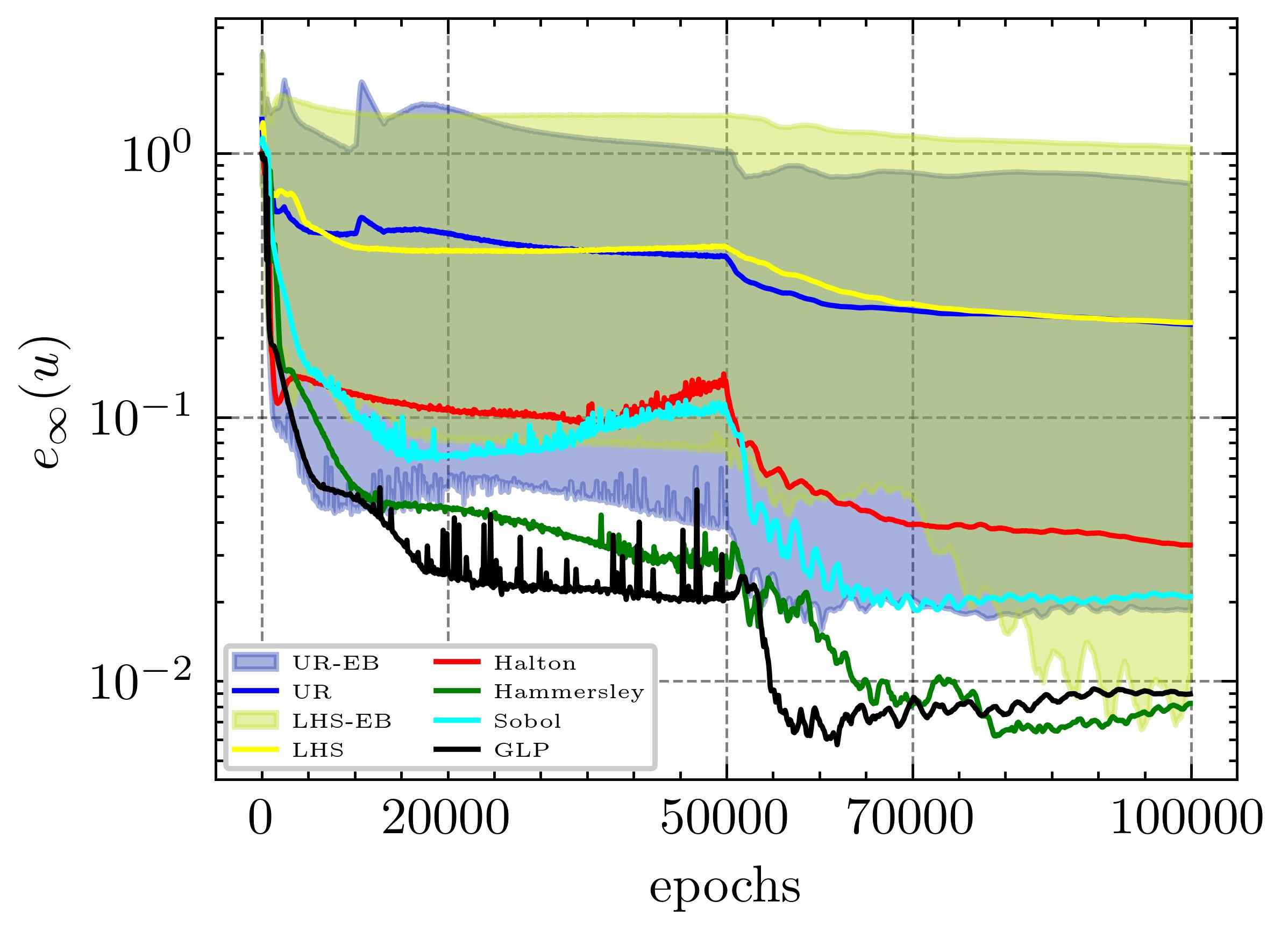}} \quad \quad 
\subfloat[performance of $e_2(u)$]{\includegraphics[width = 0.40\textwidth]
{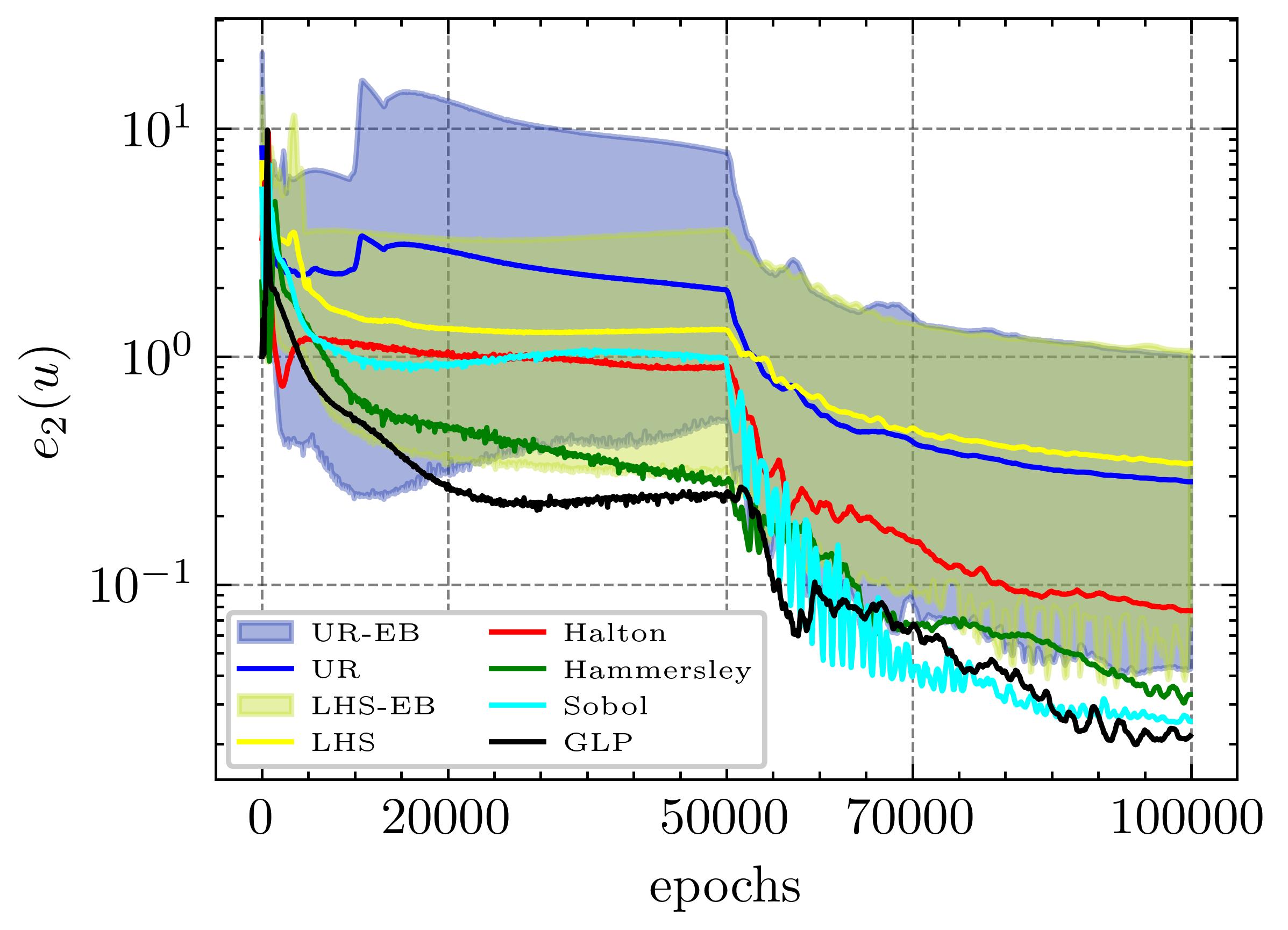}}
\caption{The performance of errors for the two-dimensional Poisson equation  with the solution Eq \eqref{eq:TwoPeaks2d_solution}. (a) the relative error $e_\infty(u)$ with different training epochs; (b) the relative error $e_2(u)$ with different training epochs.}
\label{fig:TwoPeaks2D_ErrorEpochs}
\end{figure}

\begin{table}[h]
\scriptsize
\centering
\caption{
Comparison of errors using different methods  for two-dimensional Poisson equation with two peaks of Eq \eqref{eq:TwoPeaks2d_solution}.
}
\setlength{\tabcolsep}{3.mm}{
\begin{tabular}{|c|c|c|c|c|c|c|}
\hline\noalign{\smallskip}
relative      &  uniform& LHS & Halton & Hammersley& Sobol&  GLP\\
error        &  random & method & sequence& sequence& sequence & set\\
\hline
$e_\infty(u)$  & $7.709 \times 10^{-1}$  & $3.316 \times 10^{-2}$  &  $3.279 \times 10^{-2}$ & $  8.220 \times 10^{-3}$&$2.096 \times 10^{-2}$&$8.943 \times 10^{-3}$\\
\hline
$e_2(u)$  & $1.008 \times 10^{0}$  & $6.640 \times 10^{-2}$&
$7.697 \times 10^{-2}$  & $3.283 \times 10^{-2}$&$2.467 \times 10^{-2}$&$2.171 \times 10^{-2}$\\
\hline
\end{tabular}
}
\label{tab:TwoPeaks2D_Results} 
\end{table}



Similar to the previous subsection, we also design comparative experiments to observe the performance of GLP sampling and uniform random sampling when $N_r$ takes 987, 4181, 6765, and 10946, so as to illustrate the robustness of GLP sampling to sampling budgets.

\begin{figure}[htbp]
\centering
\subfloat[performance of $e_\infty(u)$ ]{\includegraphics[width = 0.40\textwidth]{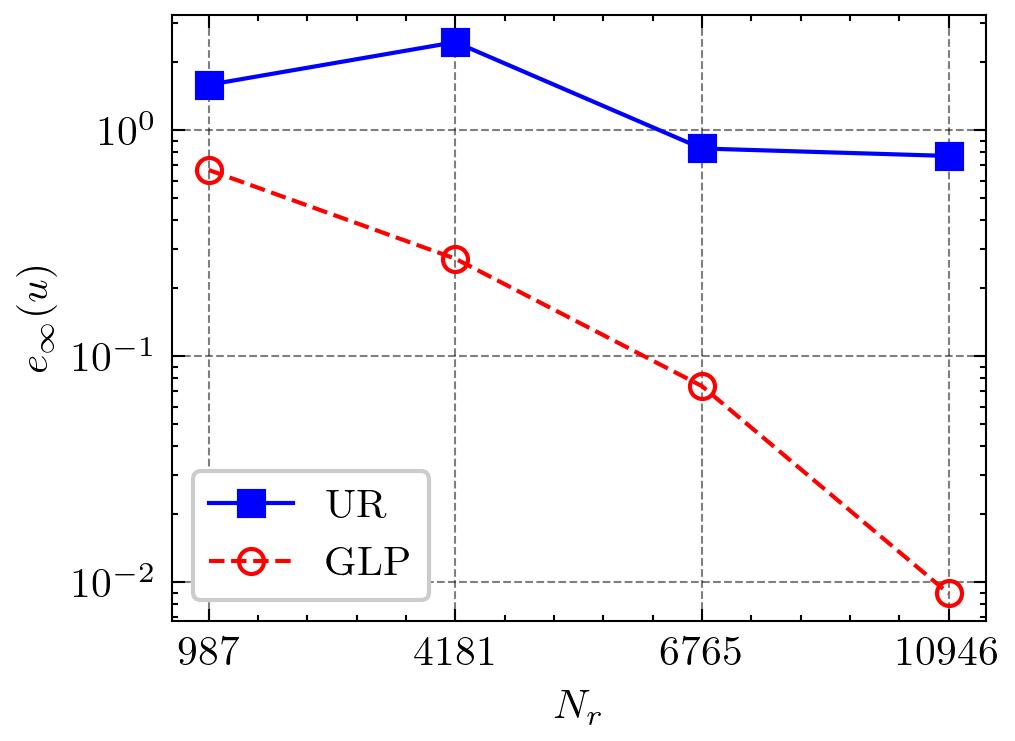}} \quad \quad \quad
\subfloat[performance of $e_2(u)$ ]{\includegraphics[width = 0.40\textwidth]
{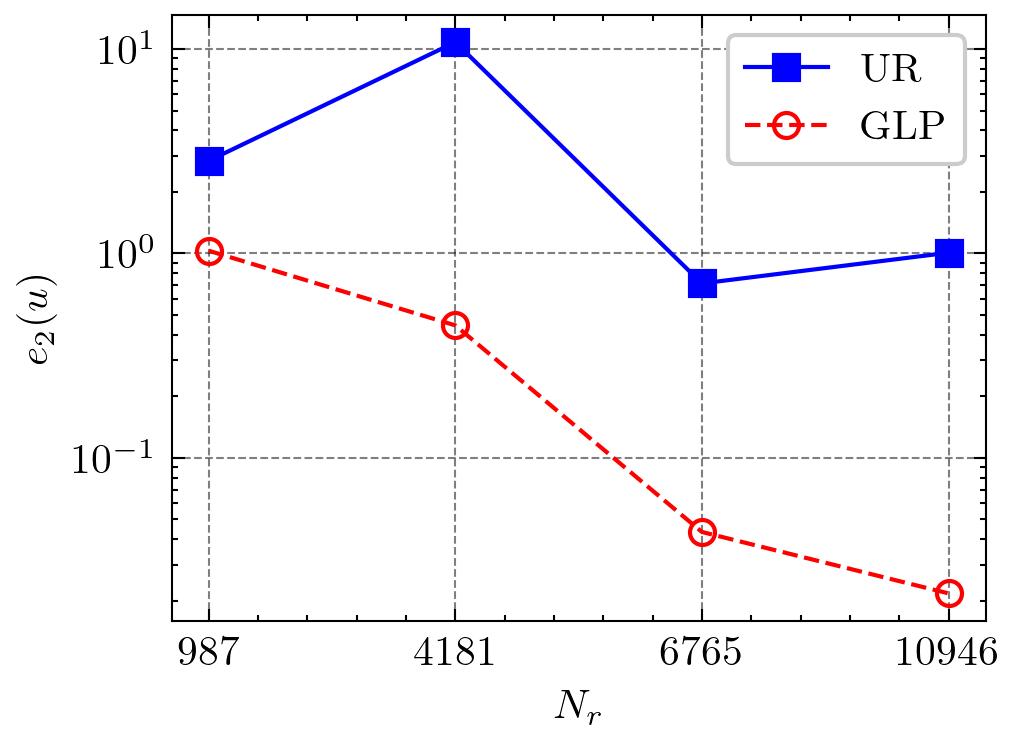}}
\caption{The performance of errors for the two-dimensional Poisson equation  with the solution Eq \eqref{eq:TwoPeaks2d_solution}. (a) the relative error $e_\infty(u)$ with different sampling budgets; (b) the relative error $e_2(u)$ with different sampling budgets.}
\label{fig:TwoPeaks2D_ErrorDiffPoints}
\end{figure}

From Fig \ref{fig:TwoPeaks2D_ErrorDiffPoints}, it can be observed that GLP sampling demonstrates a superior performance across different numbers of sampling points. Moreover, as the number of sampling points increases, the error associated with GLP sampling shows a decreasing trend, which is consistent with our theoretical findings.

\subsection{Inverse problem of two-dimensional Helmholtz equation}
\label{sec:Helmholtz_2D}

For the following two-dimensional Helmholtz equation
\begin{equation}
	\label{eq:2d_Helmholtz}
	\hspace{-0.3cm}
	\begin{array}{r@{}l}
		\left\{
		\begin{aligned}
			 \Delta u(x,y) + k^2 u(x,y)& = f(x,y), \quad (x,y) \ \mbox{in} \ \Omega, \\
                  u(x,y) & = 0,  \quad  (x,y) \ \mbox{on} \  \partial \Omega,
		\end{aligned}
		\right.
	\end{array}
\end{equation}
where $\Omega = (0,1)^2$ and $k^2$ is the unknown parameter, the exact solution is defined as 
\begin{equation}
    \label{eq:Helmholtz2d_solution}
    \hspace{-0.3cm}
    \begin{array}{r@{}l}
        \begin{aligned}
            u = sin(2\pi x) sin(2\pi y).
        \end{aligned}
    \end{array}
\end{equation}
And the source function $f(x,y)$ is given by Eq \eqref{eq:Helmholtz2d_solution}. \textcolor{black}{This is a data-driven problem, which is often used to demonstrate the efficacy of the proposed neural network. Similar as in \cite{dong2025agent,haghighat2021physics,PINN}}, after adding data-driven training points from the exact solution, the output of PINNs is the estimate solution and the parameter $k$. \textcolor{black}{Furthermore, we use uniform random sampling to extract 40 points from the exact solution corresponding to ${\left(k^*\right)}^2=9$ as extra observed data for data-driven.} Additionally, we set the initial value of the parameter $k$ to be 0.1.

We sample 6765 points within $\Omega$ as the training set, while utilizing $400 \times 400$ points for the test set. Since the homogeneous Dirichlet boundary condition is applied, it is easy to use the method of enforcing boundary conditions (\cite{lyu2020enforcing}) to eliminate the influence of boundary loss. In Fig \ref{fig:Helmholtz2D_points}, the GLP sampling points and the uniform random sampling points are shown. The uniformity of the GLP sampling is clearly better  than that of the uniform random sampling.


\begin{figure}[htbp]
\centering
\subfloat[Uniform random points]{\includegraphics[width = 0.42\textwidth]{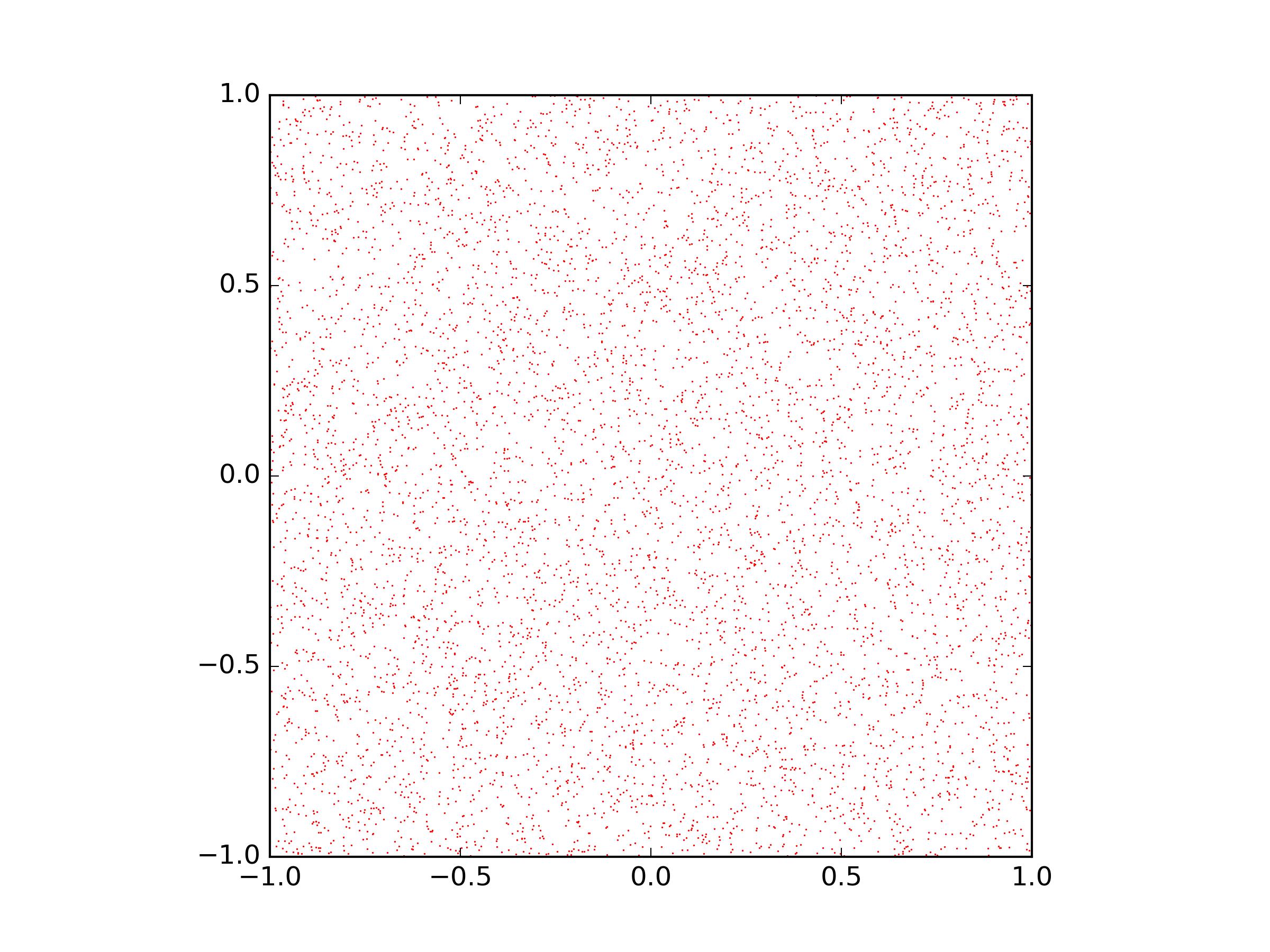}}
\subfloat[GLP set]{\includegraphics[width = 0.42\textwidth]
{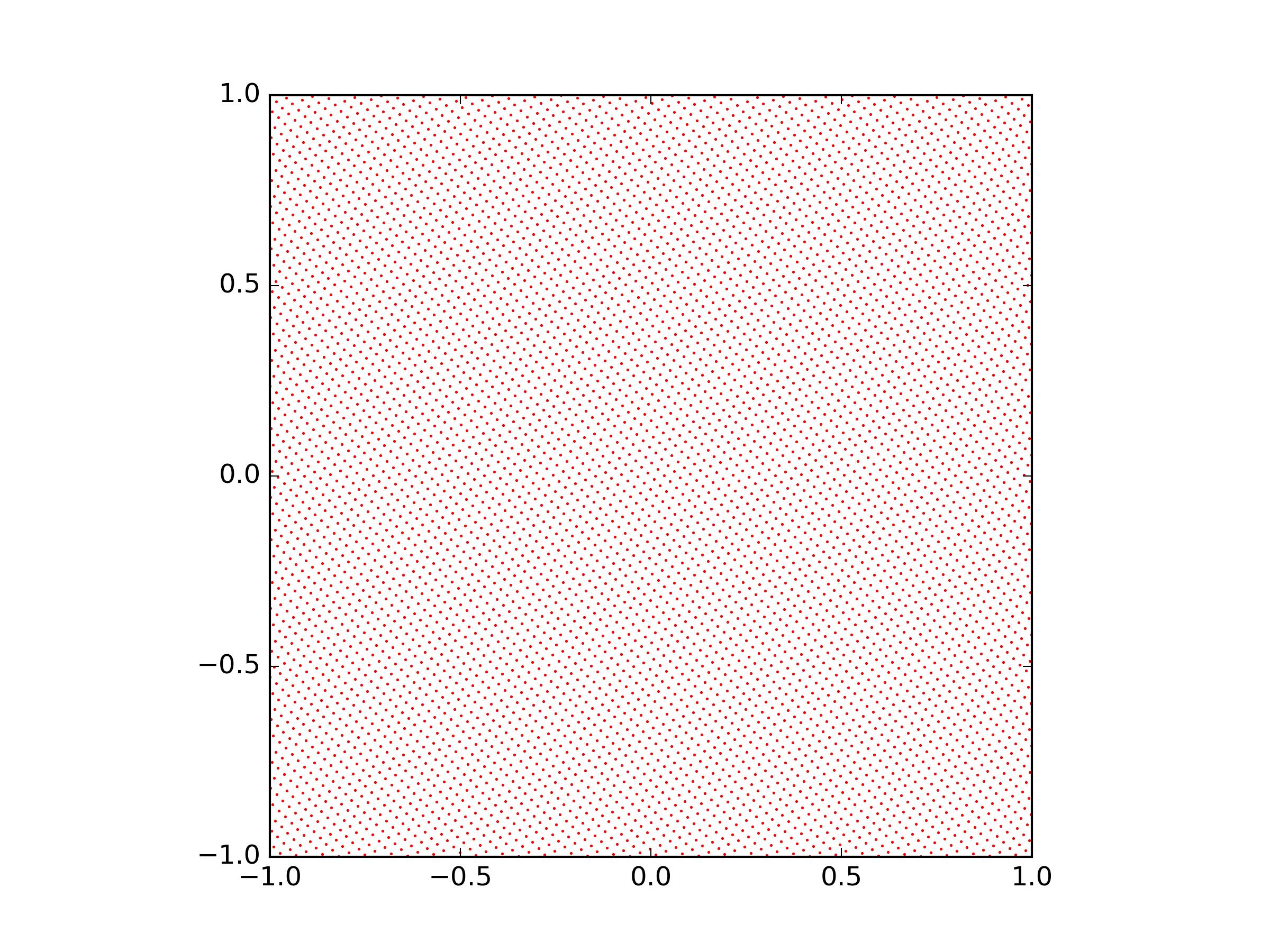}}
\caption{The training points for the inverse problem of two-dimensional Helmholtz equation. (a) the residual training points by uniform random sampling; (b)  the residual training points of GLP sampling.}
\label{fig:Helmholtz2D_points}
\end{figure}

After completing 50,000 epochs of Adam training and an additional 50,000 epochs of LBFGS training, we compare the numerical results obtained through GLP sampling with those derived from uniform random sampling to demonstrate the effectiveness of our proposed method. In Fig \ref{fig:Helmholtz2D_UniformRandom}, the approximate solution and the absolute error using uniform random sampling are shown. And the approximate solution and the absolute error using GLP sampling are given in Fig \ref{fig:Helmholtz2D_NTM}. 
Furthermore, the performance of the relative errors of the exact solution $u$ and the absolute error of parameter $k$ during training is given in Fig \ref{fig:Helmholtz2D_ErrorEpochs}. \textcolor{black}{To verify the robustness of GLP set, we simulate realistic measurement errors by introducing noise at different magnitudes (1\%, 5\%, and 10\%) into the observation data. We also conduct comparison experiments, and the results are summarized in Fig \ref{fig:Helmholtz2D_Noise} and Table \ref{tab:Helmholtz_Results}.}
It is  easy to observe that the GLP sampling has advantages over uniform random sampling when the number of training points is the same.

\begin{figure}[htbp]
\centering
\subfloat[approximate solution]{\includegraphics[width = 0.40\textwidth]
{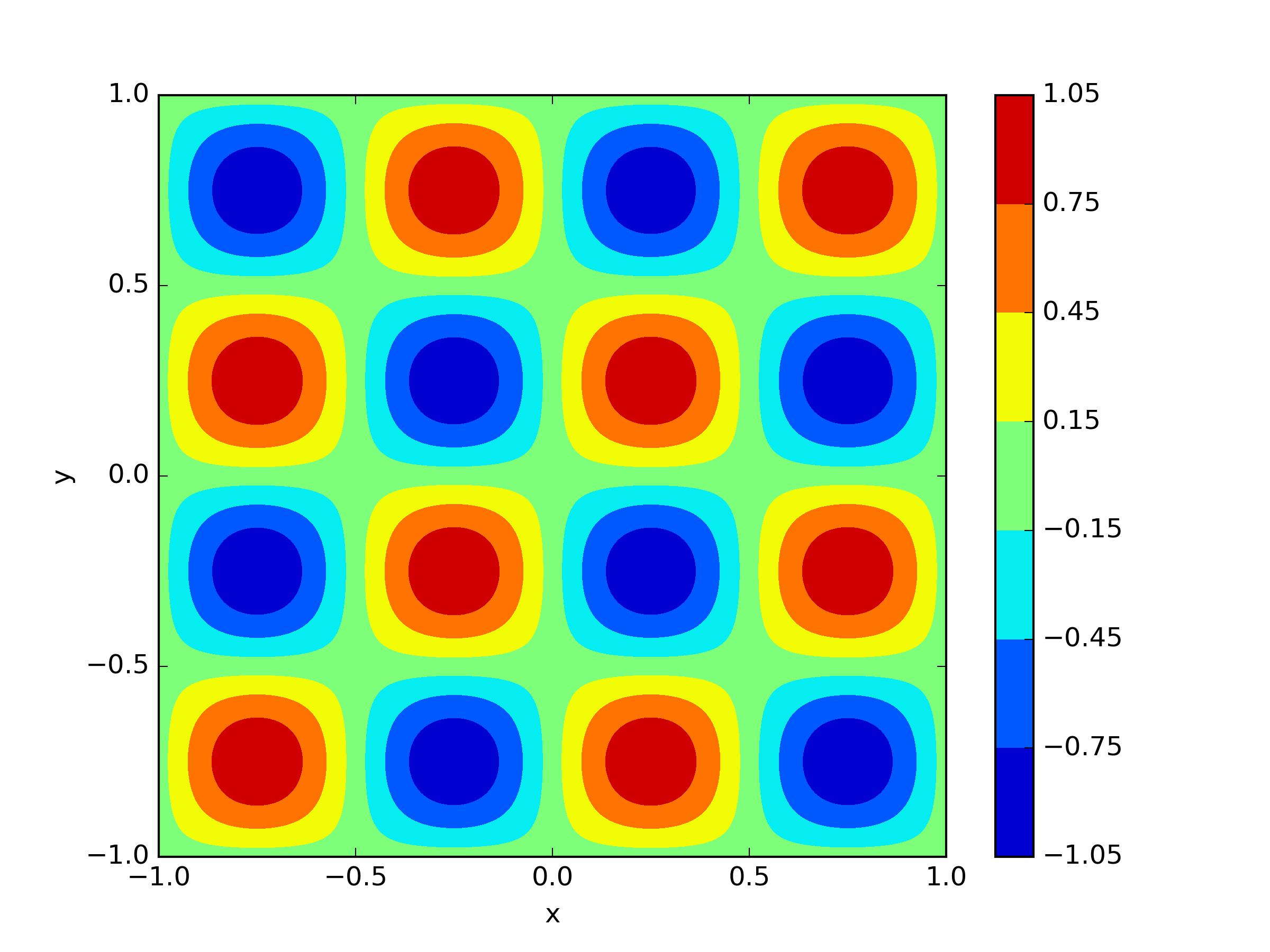}}
\subfloat[absolute error]{\includegraphics[width = 0.40\textwidth]{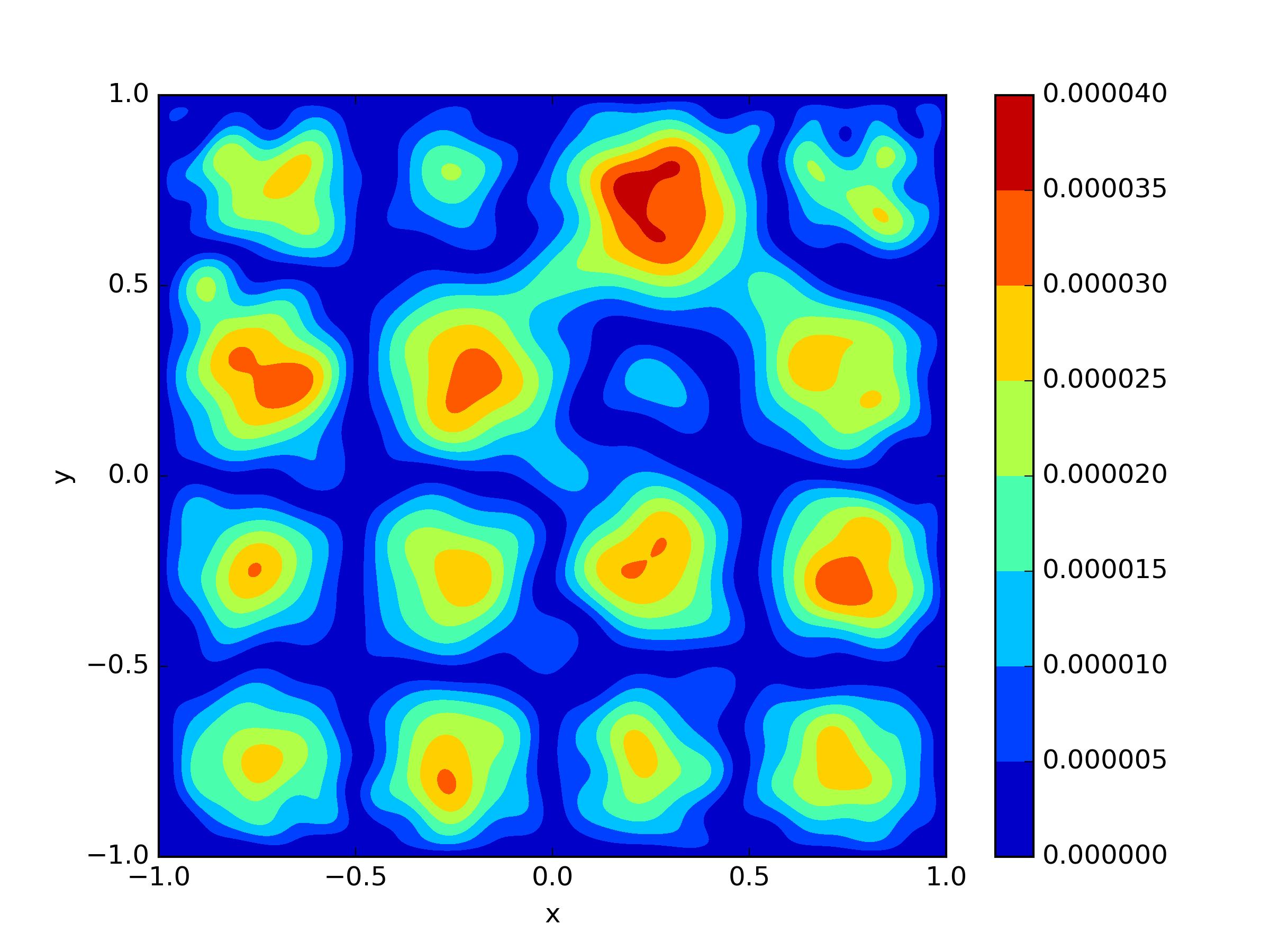}} 
\caption{The results of uniform random sampling  for the inverse problem of two-dimensional Helmholtz equation. (a) the approximate solution $u^{\text{UR}}$; (b) the absolute error $\vert u^* - u^{\text{UR}} \vert$.}
\label{fig:Helmholtz2D_UniformRandom}
\end{figure}

\begin{figure}[htbp]
\centering
\subfloat[approximate solution]{\includegraphics[width = 0.40\textwidth]
{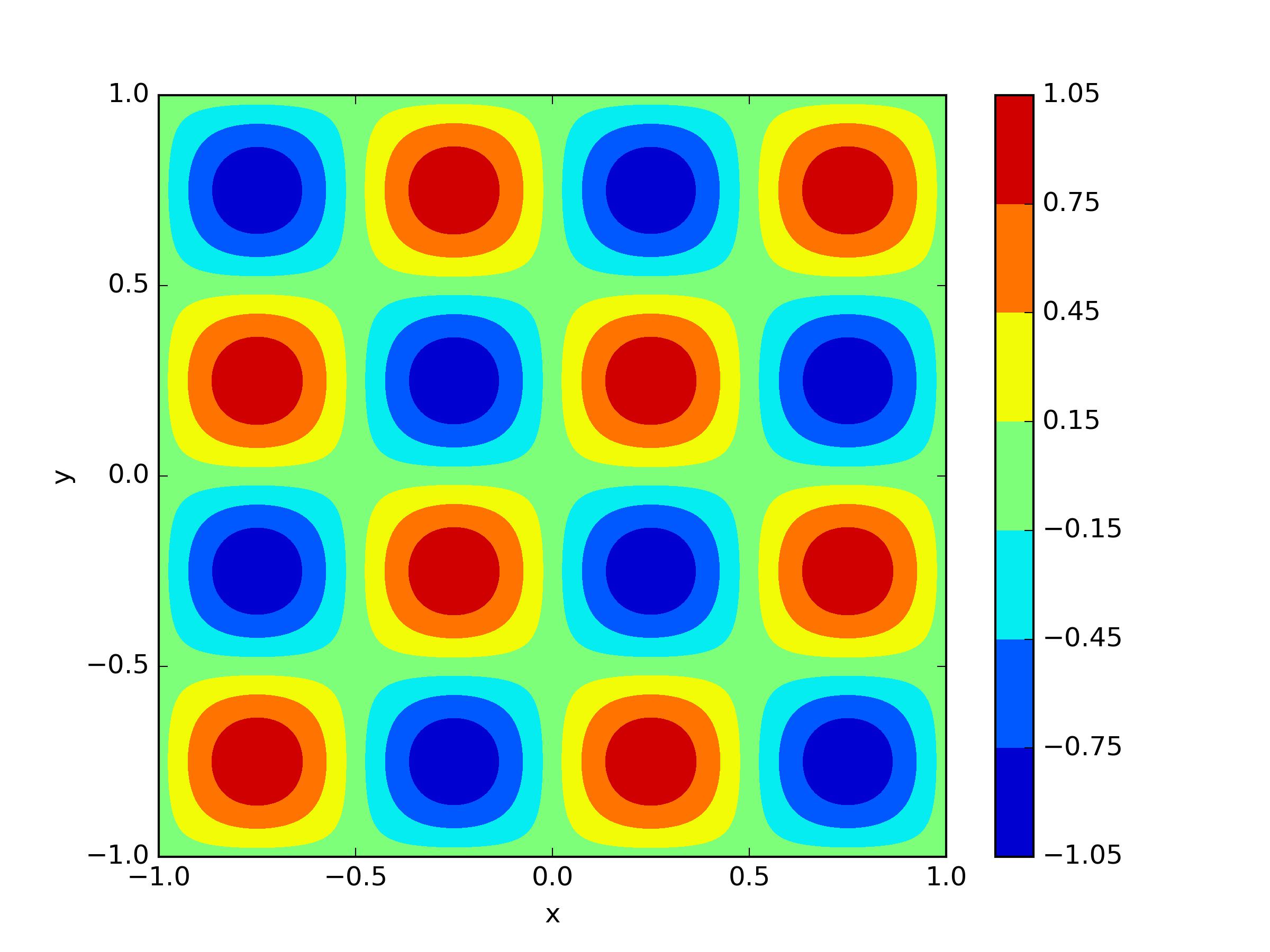}}
\subfloat[absolute error]{\includegraphics[width = 0.40\textwidth]{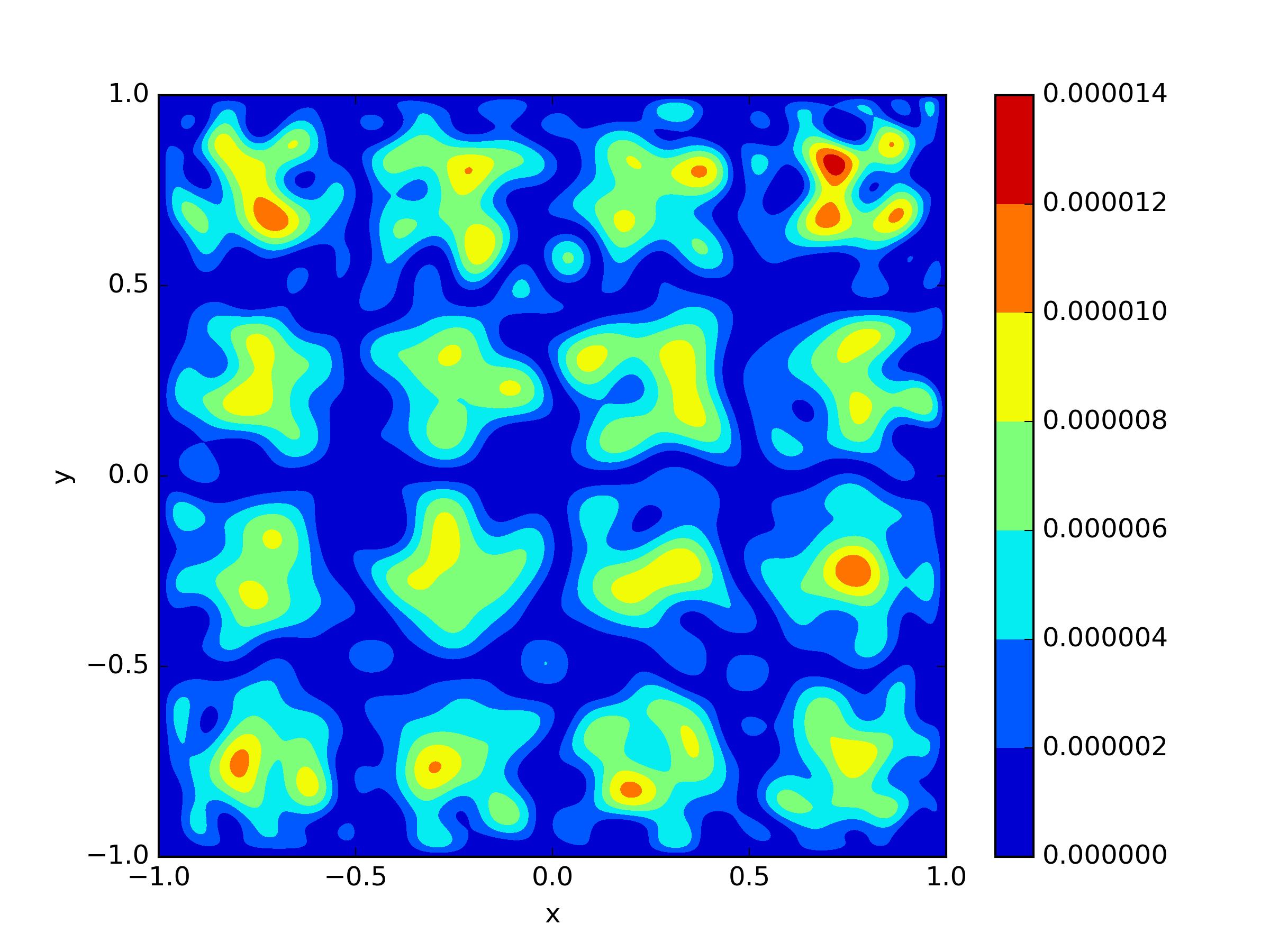}}
\caption{The results of GLP set  for the inverse problem of two-dimensional Helmholtz equation. (a) the approximate solution $u^{\text{GLP}}$; (b) the absolute error $\vert u^* - u^{\text{GLP}} \vert$.}
\label{fig:Helmholtz2D_NTM}
\end{figure}

\begin{figure}[htbp]
\centering
\subfloat[performance of $e_\infty(u)$]{\includegraphics[width = 0.30\textwidth]{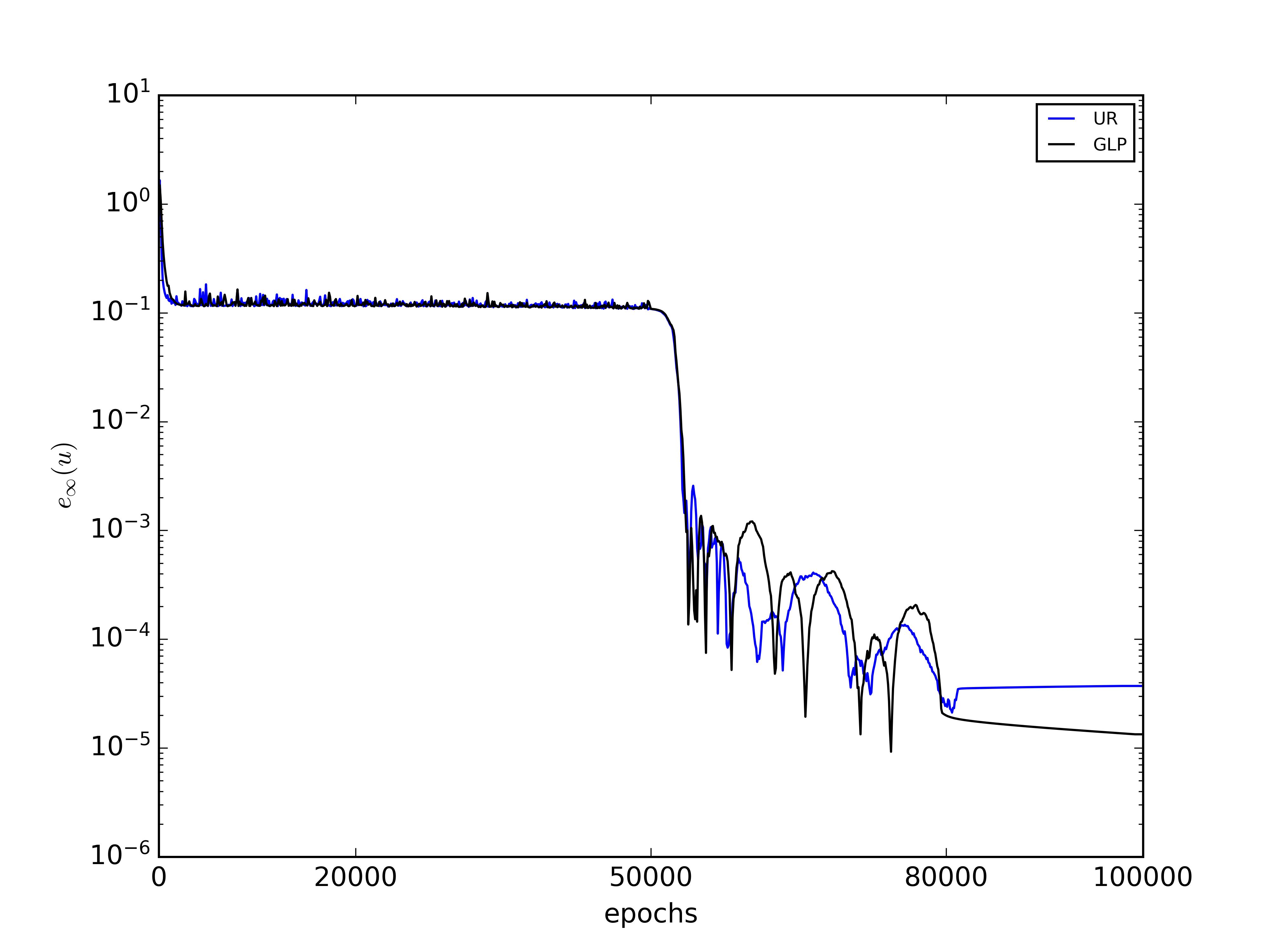}} \quad \quad 
\subfloat[performance of $e_2(u)$]{\includegraphics[width = 0.30\textwidth]
{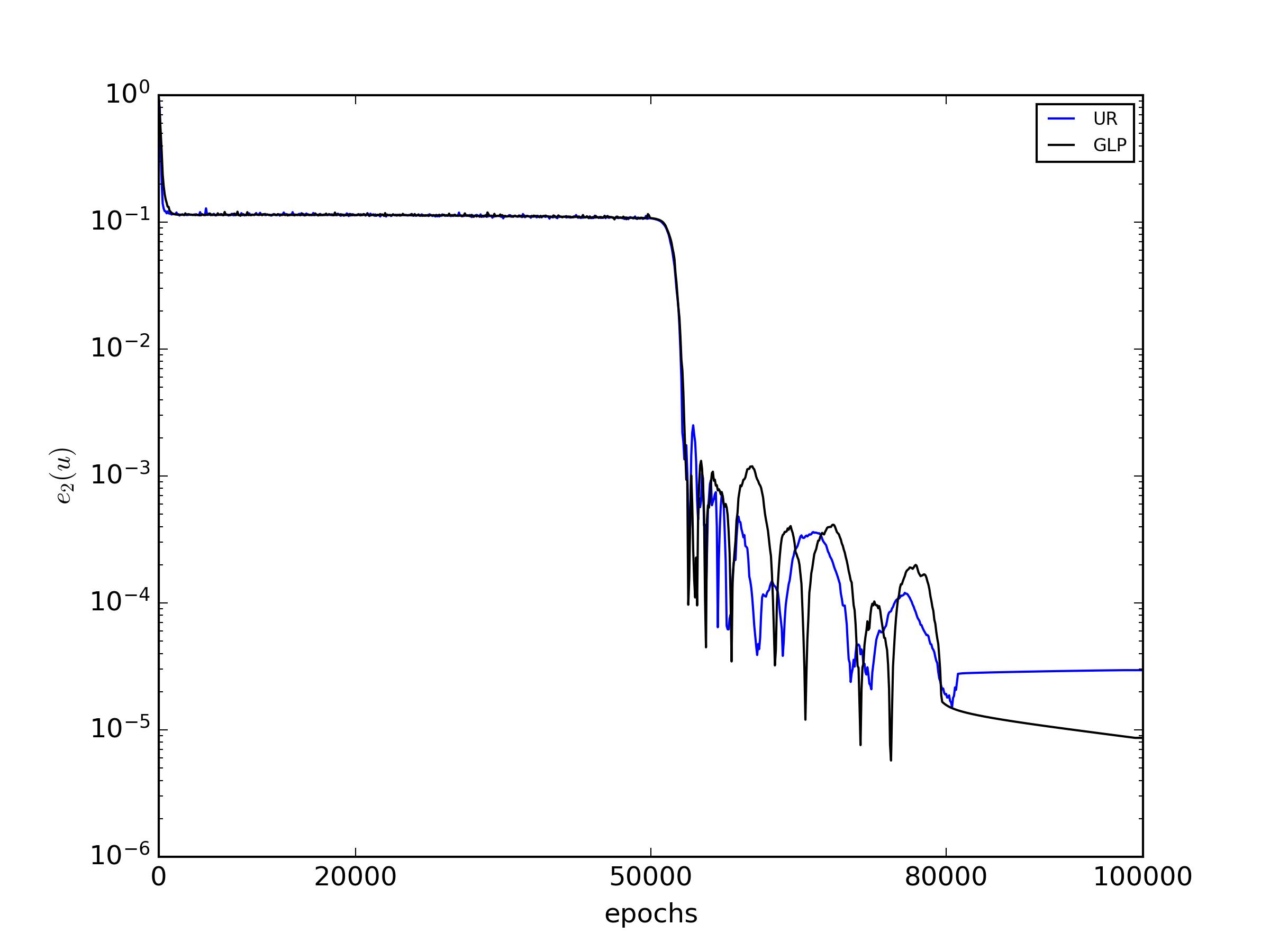}}\quad \quad 
\subfloat[performance of $\vert k - k^*\vert$]{\includegraphics[width = 0.30\textwidth]
{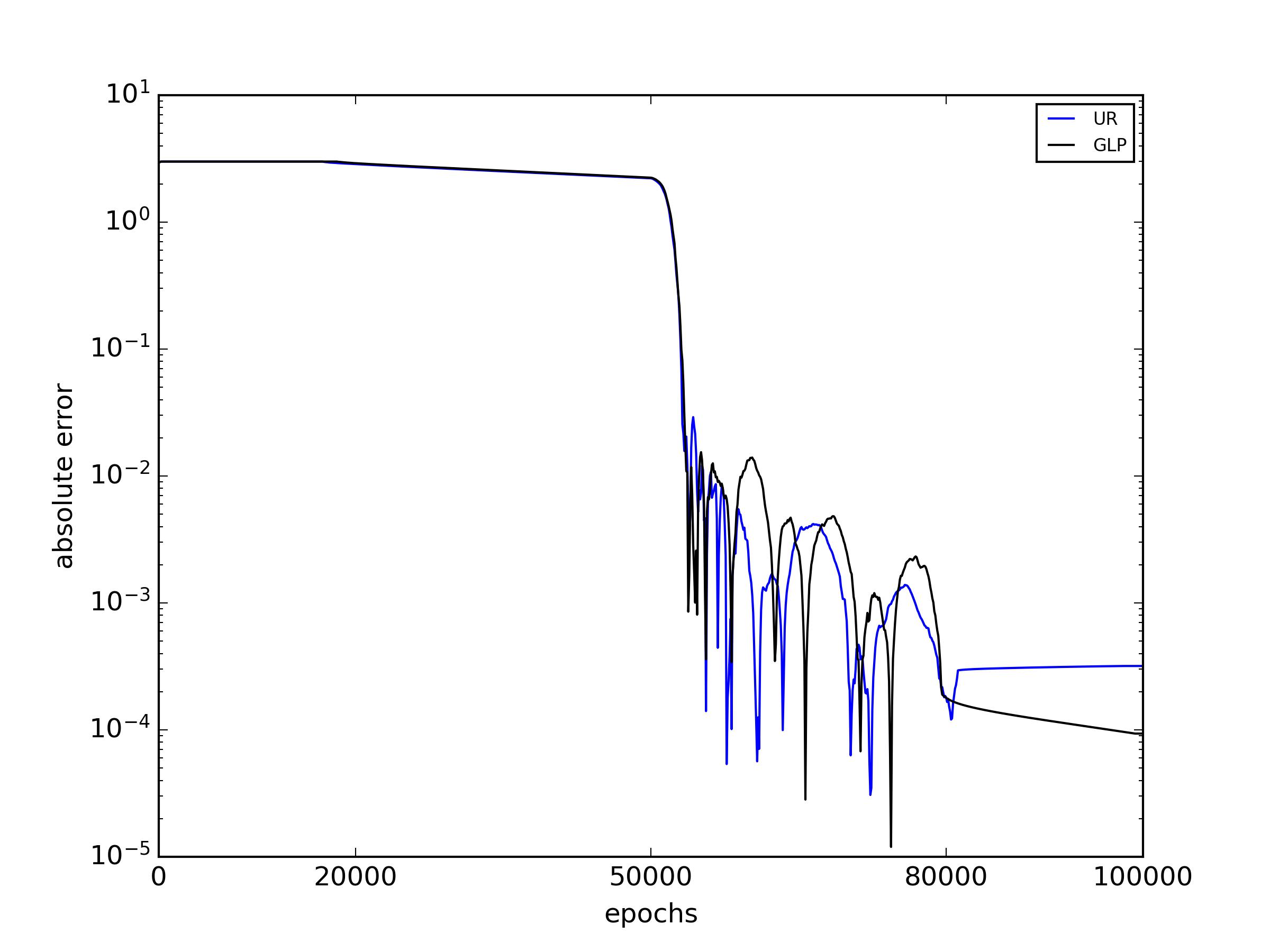}}
\caption{The performance of errors for the inverse problem of two-dimensional Helmholtz equation. (a) the relative error $e_\infty(u)$ with different training epochs; (b)  the relative error $e_2(u)$ with different training epochs; (c)  the absolute error $\vert k - k^*\vert$ with different training epochs.}
\label{fig:Helmholtz2D_ErrorEpochs}
\end{figure}

\begin{figure}[htbp]
\centering
\subfloat[performance of $e_\infty(u)$]{\includegraphics[width = 0.40\textwidth]{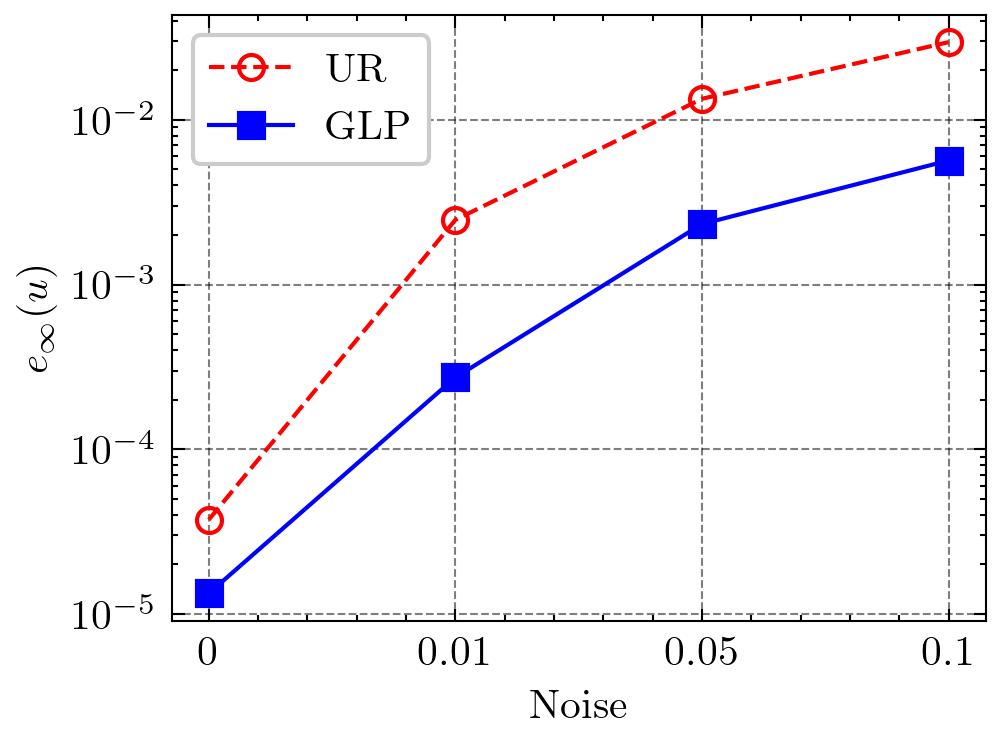}} \quad \quad 
\subfloat[performance of $e_2(u)$]{\includegraphics[width = 0.40\textwidth]
{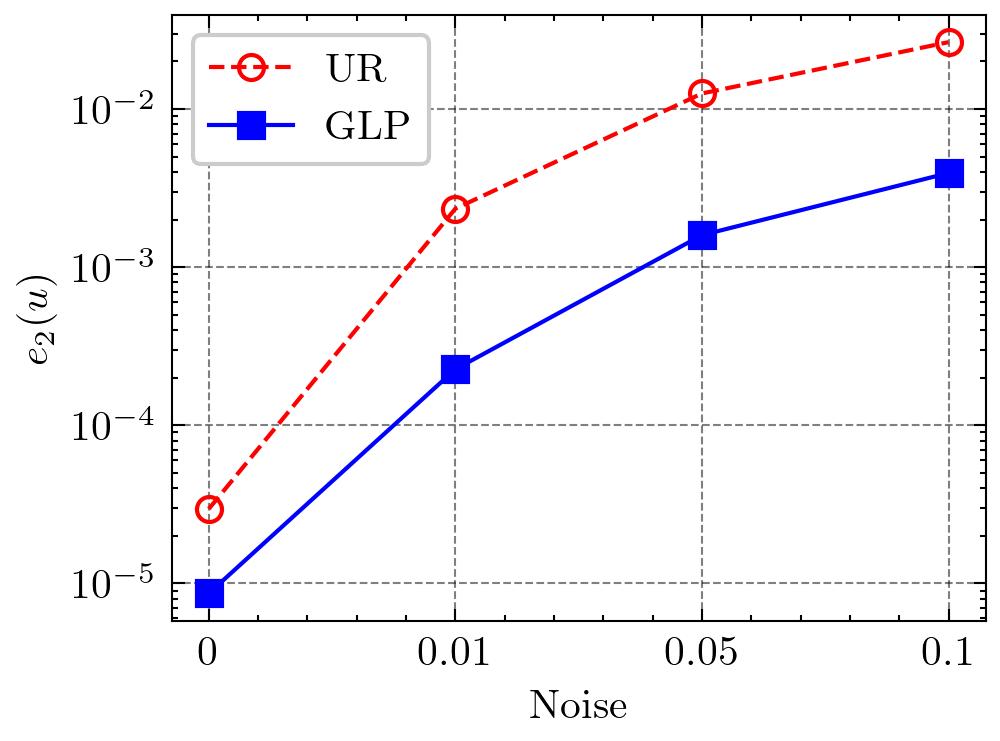}}\quad \quad 
\caption{The performance of errors under different noise scales for the inverse problem of two-dimensional Helmholtz equation. (a) the relative error $e_\infty(u)$ under different noise scales; (b)  the relative error $e_2(u)$ under different noise scales.}
\label{fig:Helmholtz2D_Noise}
\end{figure}

\begin{table}[h]
\scriptsize
\centering
\caption{
Comparison of errors between uniform random sampling and GLP sampling under different noise scales for the inverse problem of two-dimensional Helmholtz equation.
}
\setlength{\tabcolsep}{7.mm}{
\begin{tabular}{|c|c|c|c|}
\hline\noalign{\smallskip}
methods \textbackslash errors     &  $e_\infty(u)$ & $e_2(u)$ & $\vert k -k^* \vert$ \\

\hline
UR  & $3.735 \times 10^{-5}$  & $2.9527 \times 10^{-5}$  &  $3.1859 \times 10^{-4}$ \\
\hline
GLP  & $1.3348 \times 10^{-5}$  & $8.6321 \times 10^{-6}$&
$9.3284\times 10^{-5}$  \\
\hline
UR (1\% Noise)  & $2.4694 \times 10^{-3}$  & $2.3433 \times 10^{-3}$  &  $2.7101 \times 10^{-2}$ \\
\hline
GLP (1\% Noise) & $2.7287 \times 10^{-4}$  & $2.2591 \times 10^{-4}$&
$2.6950\times 10^{-5}$  \\
\hline
UR  (5\% Noise)& $1.3404 \times 10^{-2}$  & $1.2549 \times 10^{-2}$  &  $1.4085 \times 10^{-1}$ \\
\hline
GLP  (5\% Noise)& $2.3285 \times 10^{-3}$  & $1.5940 \times 10^{-3}$&
$1.2436\times 10^{-2}$  \\
\hline
UR  (10\% Noise)& $2.9593 \times 10^{-2}$  & $2.6511 \times 10^{-2}$  &  $2.8508 \times 10^{-1}$ \\
\hline
GLP  (10\% Noise)& $5.6595 \times 10^{-3}$  & $3.9407 \times 10^{-3}$&
$3.4858\times 10^{-2}$  \\
\hline
\end{tabular}
}
\label{tab:Helmholtz_Results} 
\end{table}

\subsection{High-dimensional Linear Problems}
\label{sec:Poisson_HD}

For the following high-dimensional Poisson equation
\begin{equation}
	\label{eq:Hd_Poisson}
	\hspace{-0.3cm}
	\begin{array}{r@{}l}
		\left\{
		\begin{aligned}
			 -\Delta u(x) & = f(x), \quad x \ \mbox{in} \ \Omega, \\
                  u(x) & = g(x),  \quad  x \ \mbox{on} \  \partial \Omega,
		\end{aligned}
		\right.
	\end{array}
\end{equation}
where $\Omega = (0,1)^d$, the  exact solution is defined by 
\begin{equation}
    \label{eq:PoissonHd_solution}
    \hspace{-0.3cm}
    \begin{array}{r@{}l}
        \begin{aligned}
            u = e^{-p\Vert \bx \Vert_2^2}.
        \end{aligned}
    \end{array}
\end{equation}
The Dirichlet boundary condition $g(x)$ on $\partial \Omega$ and function $f(x)$ are given by the exact solution.
We take $d=5$ and $p=10$ in Eq \eqref{eq:Hd_Poisson} and \eqref{eq:PoissonHd_solution}, and sample 10007 points in $\Omega$  and 100 points in each hyperplane on $\partial \Omega$ as the training set and 200000 points as the test set.  In order to validate the statement in Remark \ref{rem:mk_lemma1}, we specifically choose the Sigmoid activation function in this subsection. 
The distribution of the residual training points in the first two dimensions is plotted in Fig \ref{fig:PoissonHd_points}.
\begin{figure}[htbp]
\centering
\subfloat[Uniform random points]{\includegraphics[width = 0.45\textwidth]{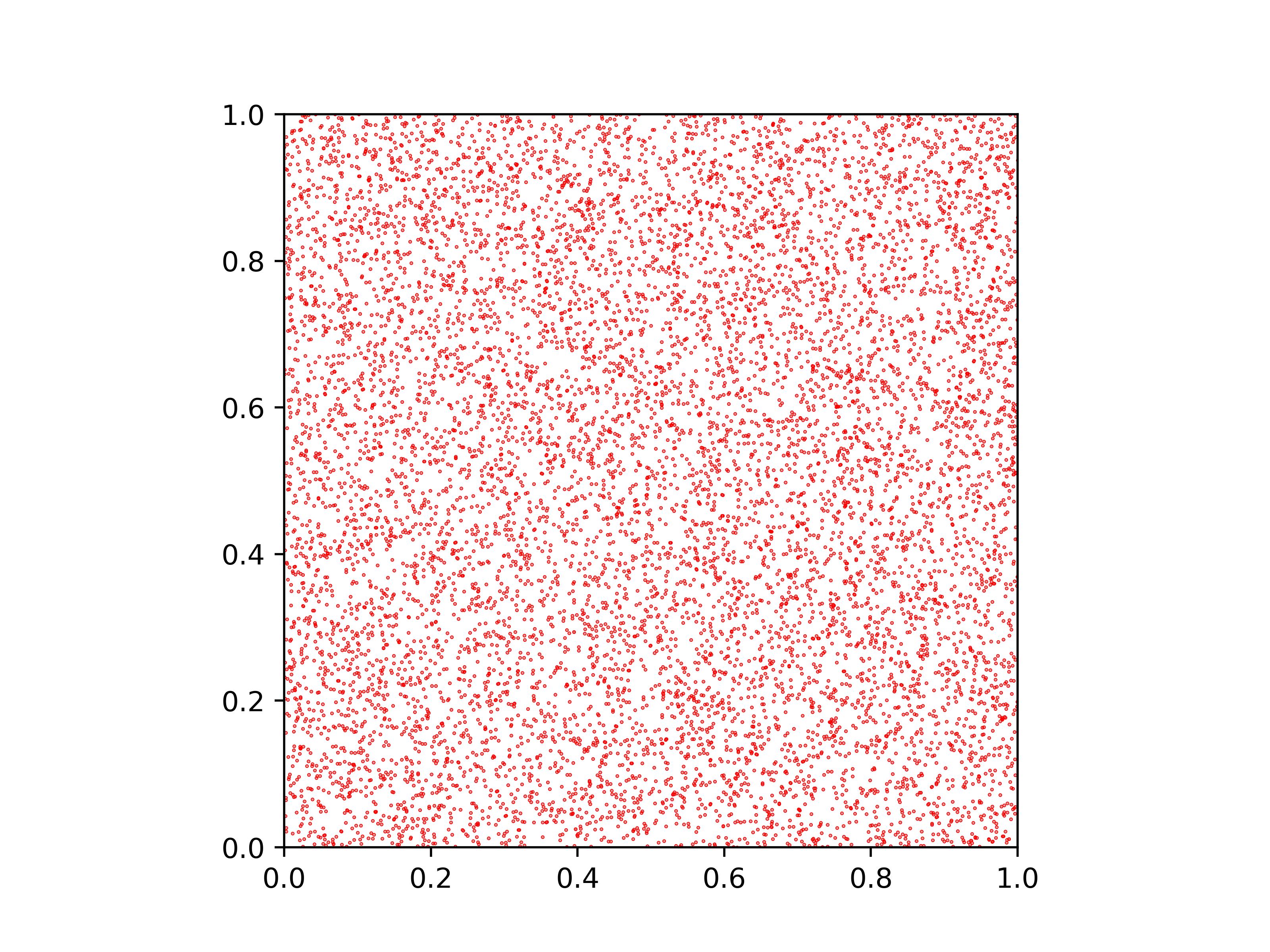}} \quad \quad 
\subfloat[GLP sampling]{\includegraphics[width = 0.45\textwidth]
{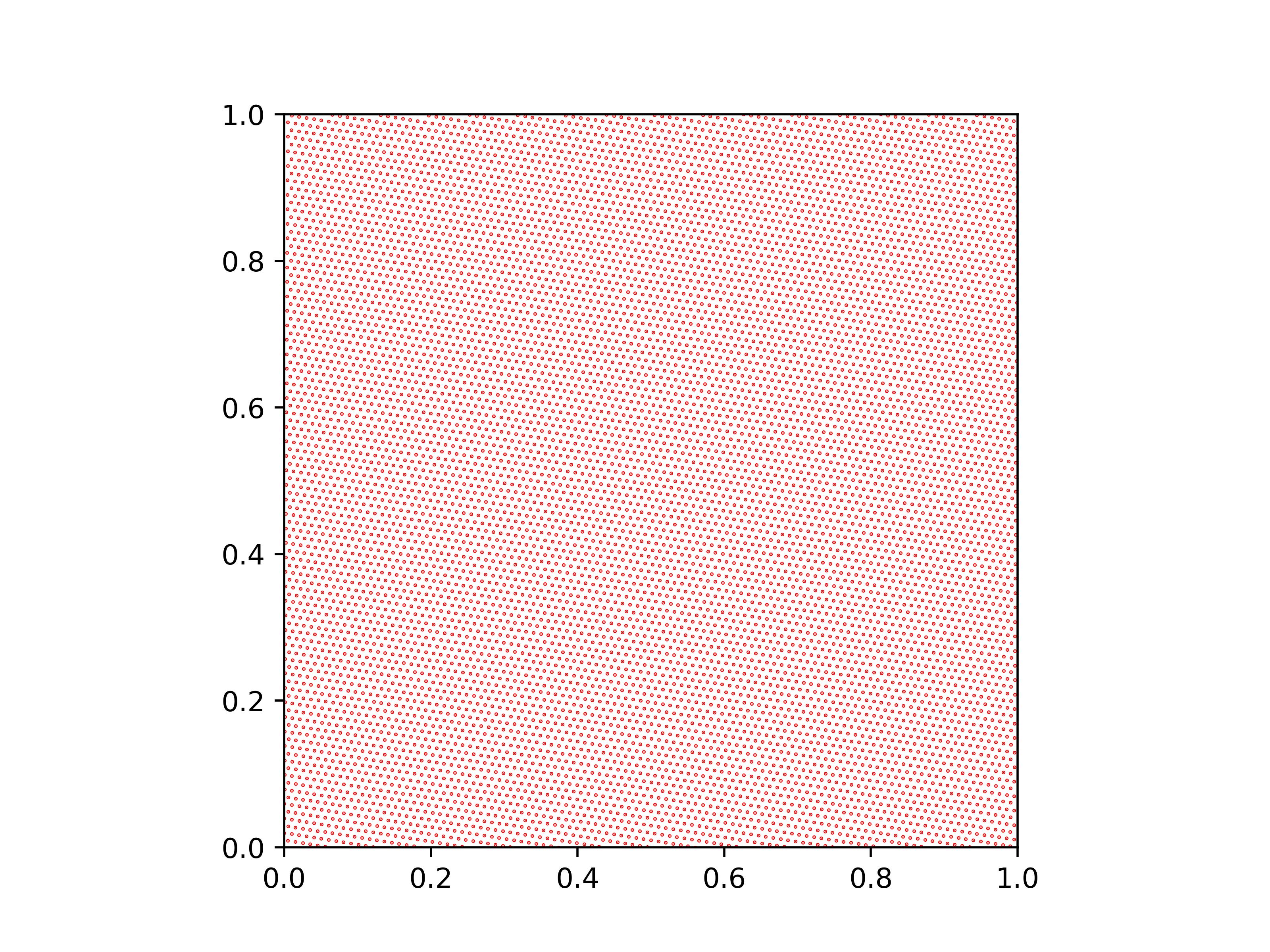}}
\caption{The training points for the five-dimensional linear problem. (a) the residual training points by uniform random sampling; (b)  the residual training points of GLP sampling.}
\label{fig:PoissonHd_points}
\end{figure}

In Fig \ref{fig:PoissonHd_ErrorEpochs}, after 50000 Adam training epochs and 20000 LBFGS training epochs, we present a comparative analysis of the numerical results derived from the GLP sampling versus those obtained through uniform random sampling. This comparison aims to demonstrate the efficacy of our proposed method for the same number of training points. Furthermore, taking into account the number of training points, we devise various sampling strategies grounded on the quantity of uniform random samples and conduct controlled experiments with the GLP sampling.  The experimental results conclusively demonstrate that the GLP sampling retains a distinct advantage over uniform random sampling, even when the latter employs approximately five times the number of points. The details of the different sampling strategies are summarized in Table \ref{tab:PoissonHd_PointsSettings}.

\begin{figure}[htbp]
\centering
\subfloat[performance of $e_\infty(u)$]{\includegraphics[width = 0.45\textwidth]{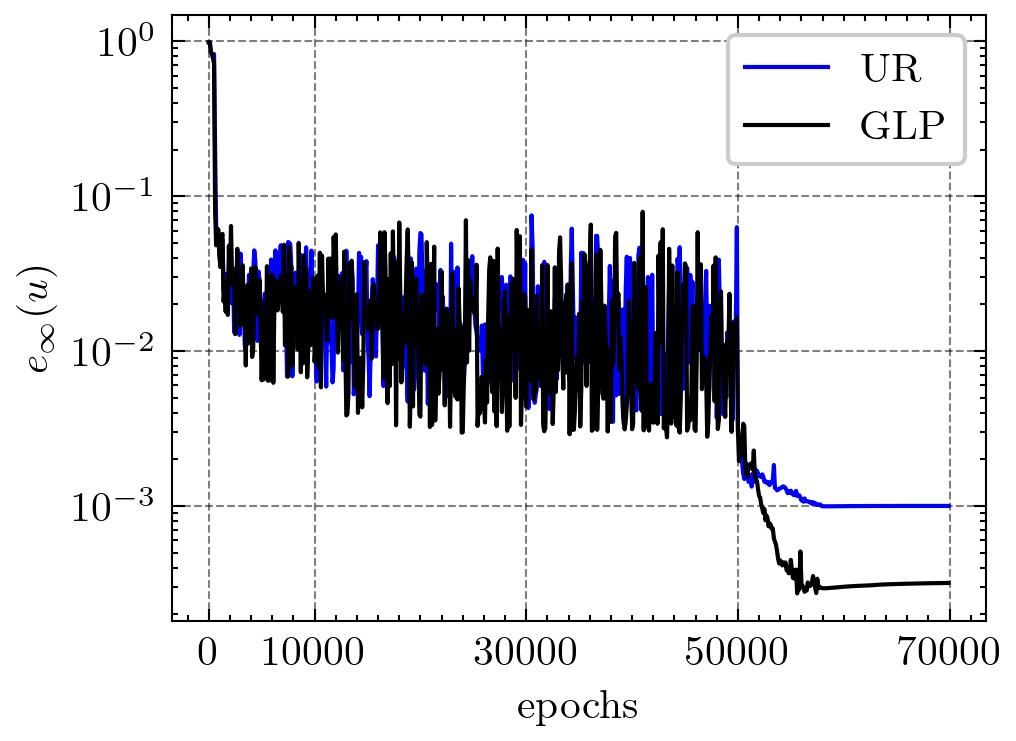}} \quad \quad 
\subfloat[performance of $e_2(u)$]{\includegraphics[width = 0.45\textwidth]
{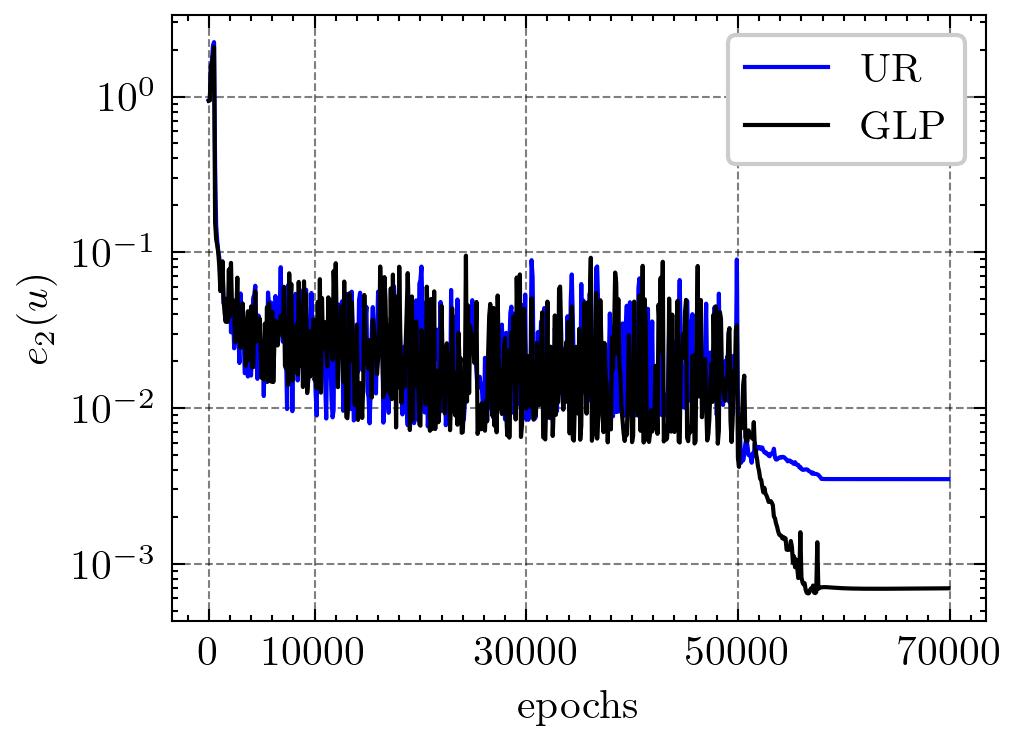}}
\caption{The performance of errors for the five-dimensional linear problem. (a) the relative error $e_\infty(u)$ during training process; (b) the relative error $e_2(u)$ during training process.}
\label{fig:PoissonHd_ErrorEpochs}
\end{figure}

\begin{table}[h]
\scriptsize
\centering
\caption{
The number of residual points in $\Omega$ and the number of boundary points at each hyperplane on $\partial \Omega$ using sampling strategies for the five-dimensional linear problem.
}
\setlength{\tabcolsep}{5.mm}{
\begin{tabular}{|c|c|c|c|c|}
\hline\noalign{\smallskip}
strategy   &  1 &  2 &3 & 4 \\
\hline
method  & UR    & UR & UR & GLP\\
\hline
residual points  & 10007 & 33139 & 51097 & 10007\\
\hline
each hyperplane   & 100   & 100 & 100 &100\\
\hline
$\be_2(\bu)$ &$1.890\times 10^{-3}$ & $1.257\times 10^{-3}$ &  $1.172\times 10^{-3}$ & $6.950\times 10^{-4}$ \\
\hline
$\be_\infty(\bu)$ &  $ 4.755\times 10^{-4}$&  $4.197\times 10^{-4}$& $3.819\times 10^{-4}$ & $3.182\times 10^{-4}$\\
\hline
\end{tabular}
}
\label{tab:PoissonHd_PointsSettings} 
\end{table}


Next, we take $d=8$ and $p=1$ in Eq \eqref{eq:Hd_Poisson} and \eqref{eq:PoissonHd_solution}, and sample 11215 points in $\Omega$  and 200 points in each hyperplane on $\partial \Omega$ as the training set and 500000 points as the test set.  
The distribution of these 
residual training points in the first two dimensions is plotted in Fig \ref{fig:Poisson8d_points}.
\begin{figure}[htbp]
\centering
\subfloat[Uniform random points]{\includegraphics[width = 0.45\textwidth]{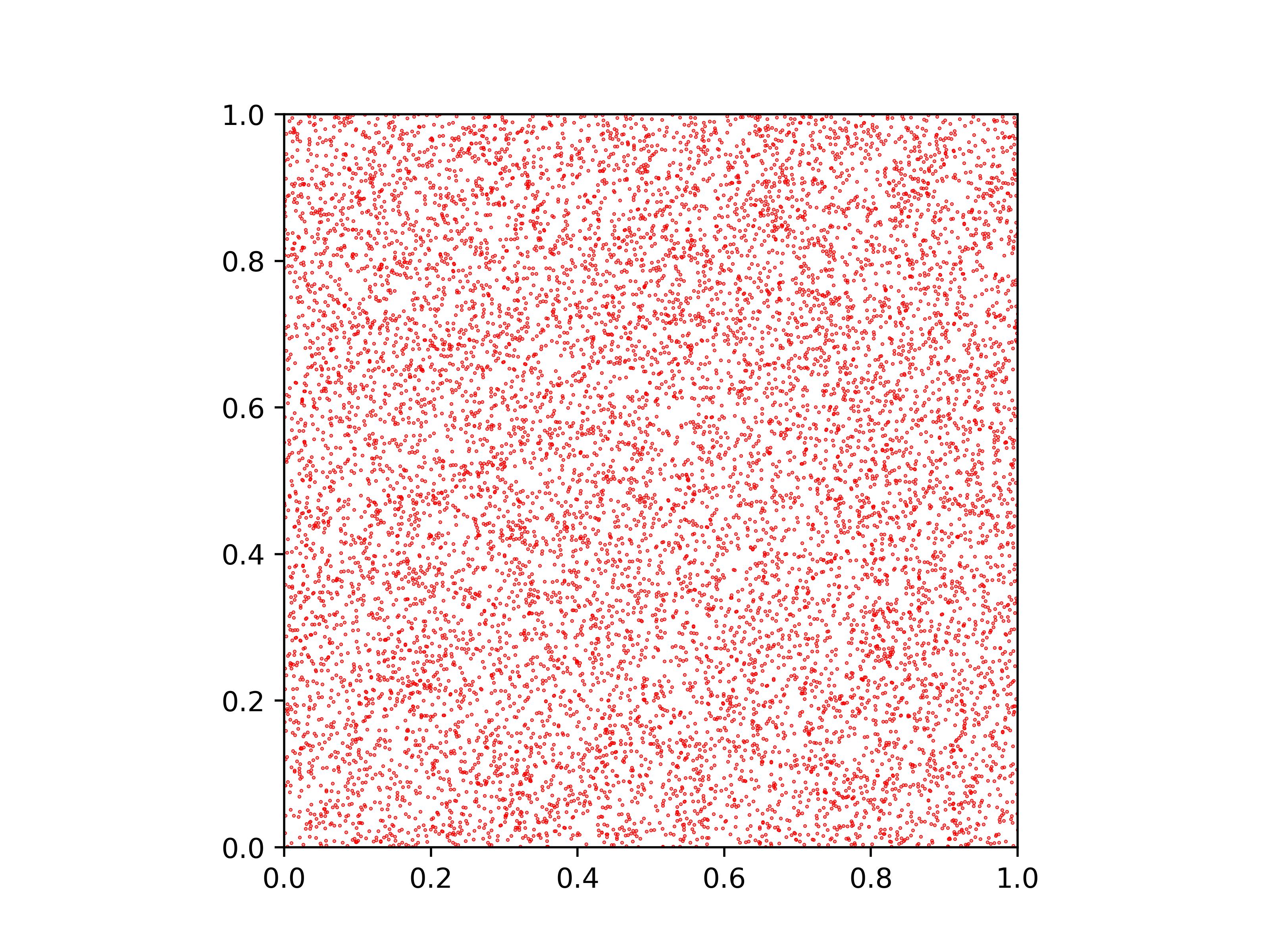}} \quad \quad 
\subfloat[GLP sampling points]{\includegraphics[width = 0.45\textwidth]
{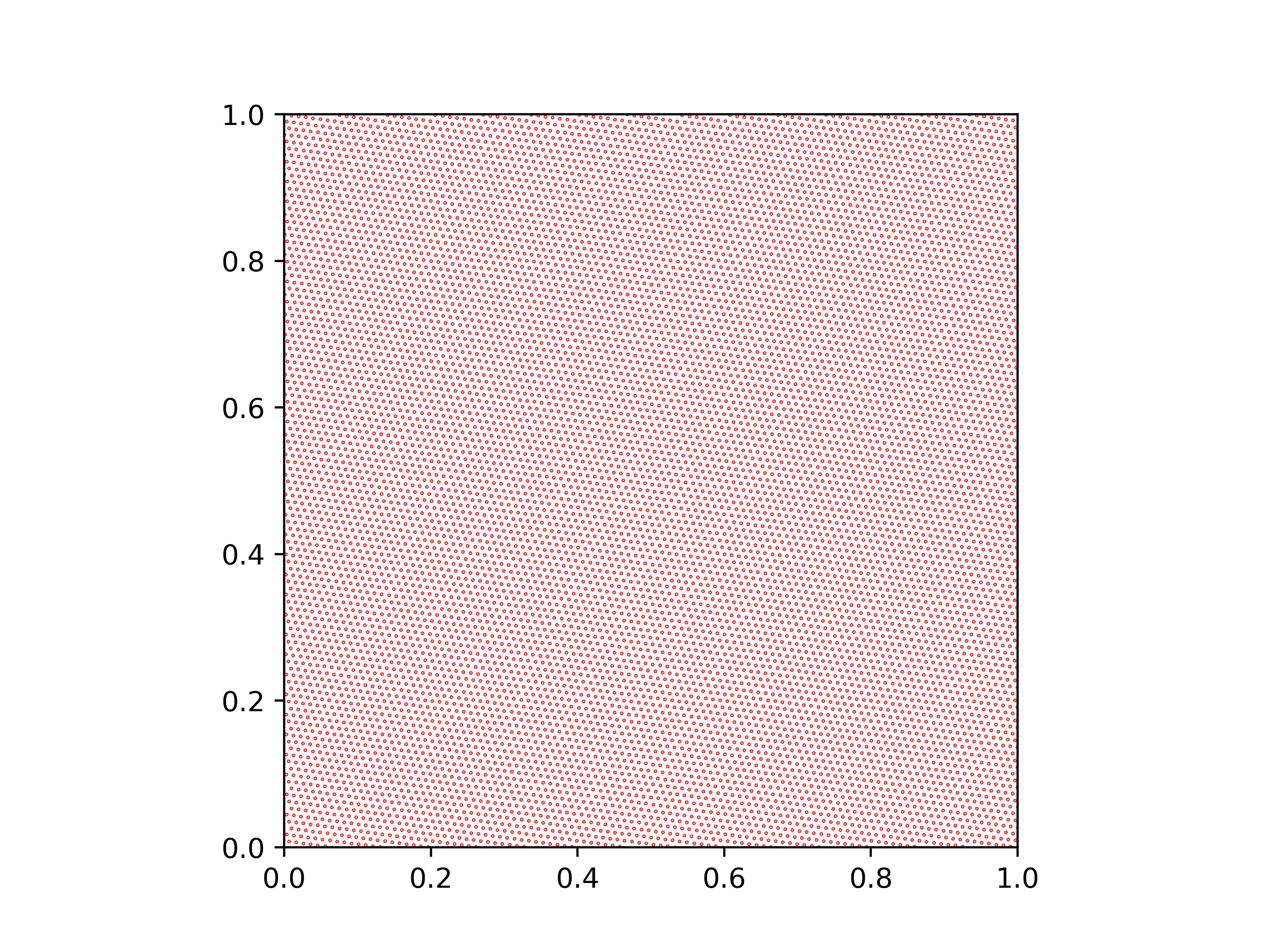}}
\caption{The training points for the eight-dimensional linear problem. (a) the residual training points by uniform random sampling; (b)  the residual training points of GLP sampling.}
\label{fig:Poisson8d_points}
\end{figure}

In Fig \ref{fig:Poisson8d_ErrorEpochs}, after 50000 epochs of Adam training and 20000 epochs of LBFGS training, we undertake a comparative analysis of the numerical results derived from the GLP sampling versus those obtained through uniform random sampling. 
The details the different sampling strategies are summarized in Table \ref{tab:Poisson8d_PointsSettings}, which shows the effectiveness of our method.

\begin{figure}[htbp]
\centering
\subfloat[performance of $e_\infty(u)$]{\includegraphics[width = 0.45\textwidth]{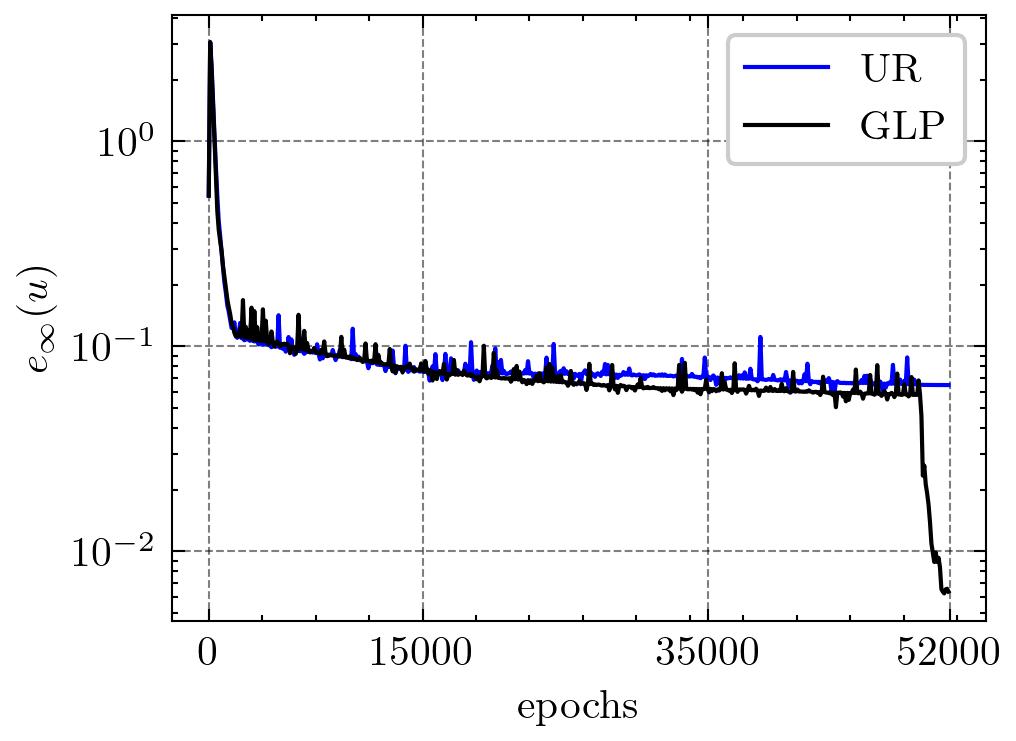}} \quad \quad 
\subfloat[performance of $e_2(u)$]{\includegraphics[width = 0.45\textwidth]
{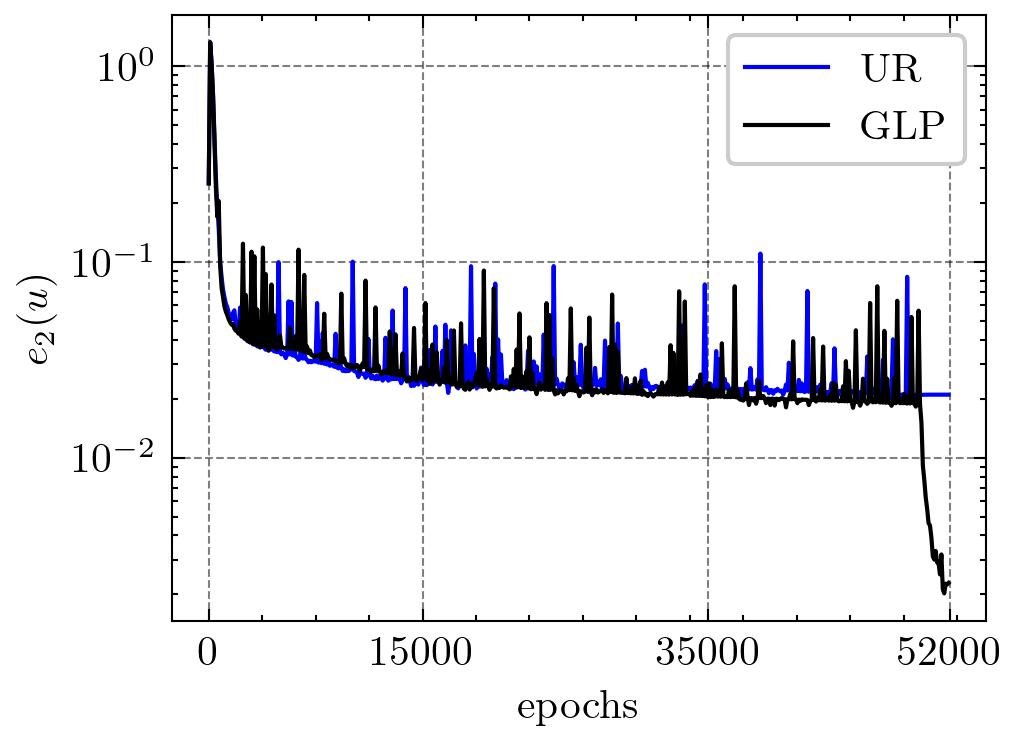}}
\caption{The performance of errors for the eight-dimensional linear problem. (a) the relative error $e_\infty(u)$ during training process; (b) the relative error $e_2(u)$ during training process.}
\label{fig:Poisson8d_ErrorEpochs}
\end{figure}


\begin{table}[h]
\scriptsize
\centering
\caption{
The number of residual points in $\Omega$ and the number of boundary points at each hyperplane on $\partial \Omega$ using sampling strategies for the eight-dimensional linear problem.
}
\setlength{\tabcolsep}{5.mm}{
\begin{tabular}{|c|c|c|c|c|}
\hline\noalign{\smallskip}
strategy   &  1 &  2 &3  & 4\\
\hline
method  & UR  & UR  & UR  & GLP\\
\hline
residual points  & 11215 & 24041  & 33139 & 11215\\
\hline
each hyperplane   & 200  & 200  & 200 & 200\\
\hline
$\be_2(\bu)$& $2.096 \times 10^{-2}$ &$1.883\times 10^{-2}$  &  $1.172\times 10^{-3}$ & $2.287\times 10^{-3}$ \\
\hline  
$\be_\infty(\bu)$& $6.481\times 10^{-2}$ &  $ 5.798\times 10^{-2}$& $5.272\times 10^{-3}$ & $6.369\times 10^{-3}$\\
\hline  
\end{tabular}
}
\label{tab:Poisson8d_PointsSettings} 
\end{table}

\subsection{High-dimensional Nonlinear Problems}
For the following high-dimensional nonlinear equation
\begin{equation}
	\label{eq:Hd_nonlinear}
	\hspace{-0.3cm}
	\begin{array}{r@{}l}
		\left\{
		\begin{aligned}
			 -\Delta u(x) + u^3(x) & = f(x), \quad x \ \mbox{in} \ \Omega, \\
                  u(x) & = g(x),  \quad  x \ \mbox{on} \  \partial \Omega,
		\end{aligned}
		\right.
	\end{array}
\end{equation}
where $\Omega = (0,1)^d$, the exact solution is defined by 
\begin{equation}
    \label{eq:NonlinearHd_solution}
    \hspace{-0.3cm}
    \begin{array}{r@{}l}
        \begin{aligned}
            u = sin(\frac{k\pi}{d}\sum_{i=1}^d x_i).
        \end{aligned}
    \end{array}
\end{equation}
The source function $f(x)$  and the  Dirichlet boundary condition $g(x)$ are given by the exact solution Eq \eqref{eq:NonlinearHd_solution}.
We take $d=5$ and $k=4$ in Eq \eqref{eq:NonlinearHd_solution}, and sample 3001 points in $\Omega$  and 100 points in each hyperplane on $\partial \Omega$ as the training set and 200000 points as the test set.  
The distribution of the residual training points in the first two dimensions is plotted in Fig \ref{fig:NonlinearHd_points}.

\begin{figure}[htbp]
\centering
\subfloat[Uniform random points]{\includegraphics[width = 0.45\textwidth]{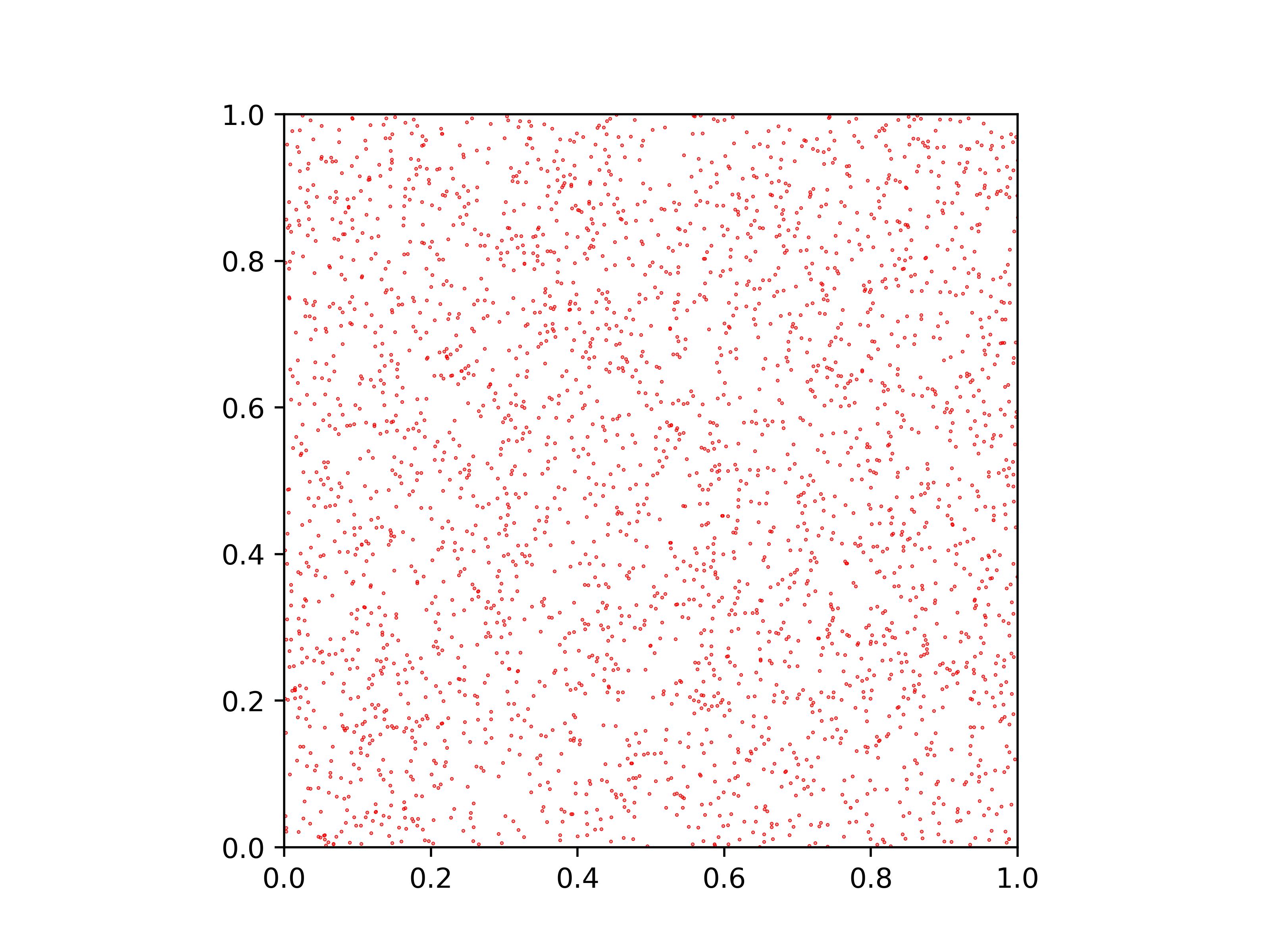}} \quad \quad 
\subfloat[GLP sampling points]{\includegraphics[width = 0.45\textwidth]
{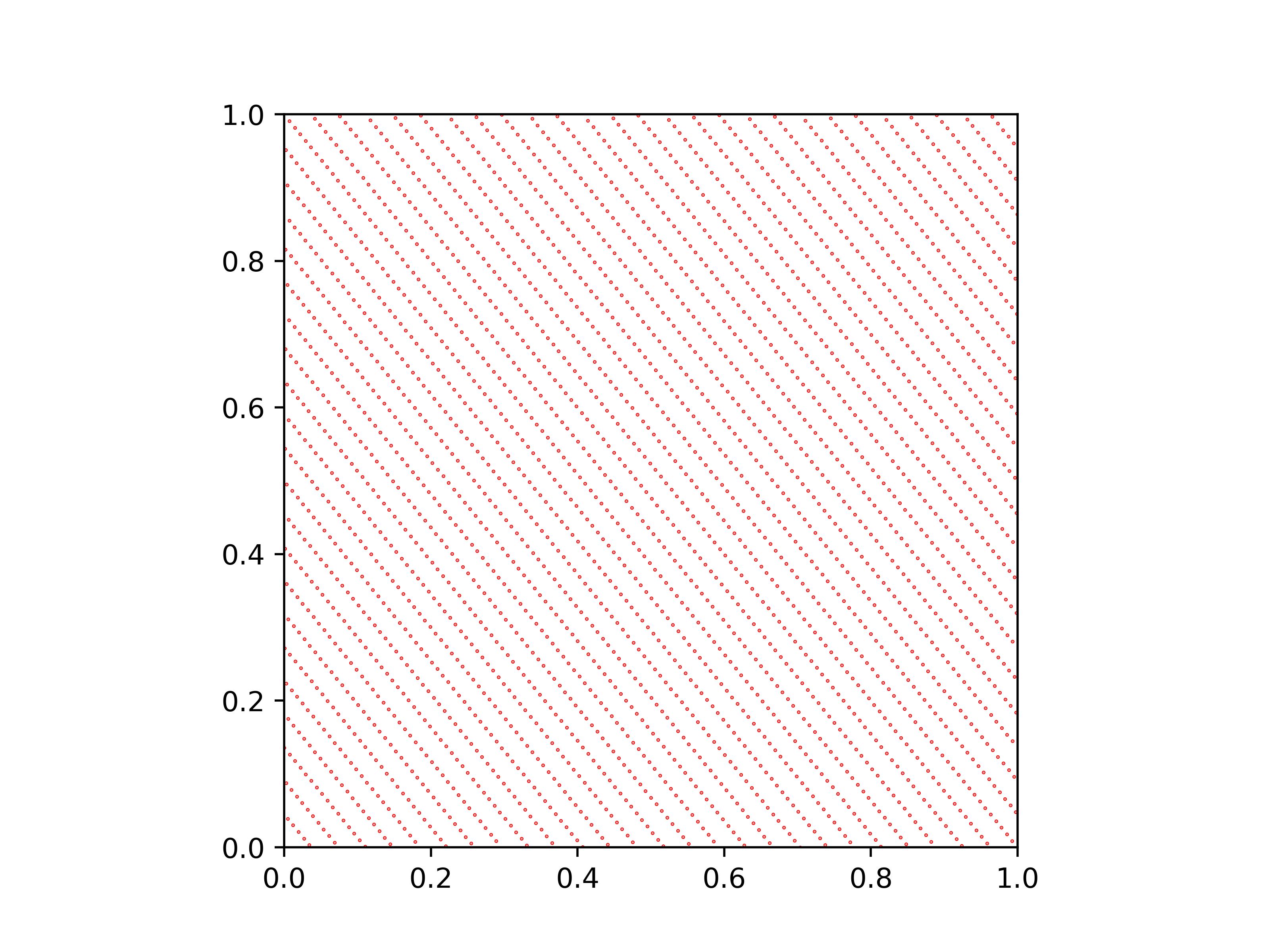}}
\caption{The training points for the five-dimensional nonlinear problem. (a) the residual training points by uniform random sampling; (b)  the residual training points of GLP sampling.}
\label{fig:NonlinearHd_points}
\end{figure}

In Figure \ref{fig:NonlinearHd_ErrorEpochs}, we compare the numerical results derived from GLP sampling with those from uniform random sampling after 7500 epochs of Adam training and 10000 epochs of LBFGS training. This comparison aims to demonstrate the efficiency of our method for the same number of training points.

\begin{figure}[htbp]
\centering
\subfloat[performance of $e_\infty(u)$]{\includegraphics[width = 0.45\textwidth]{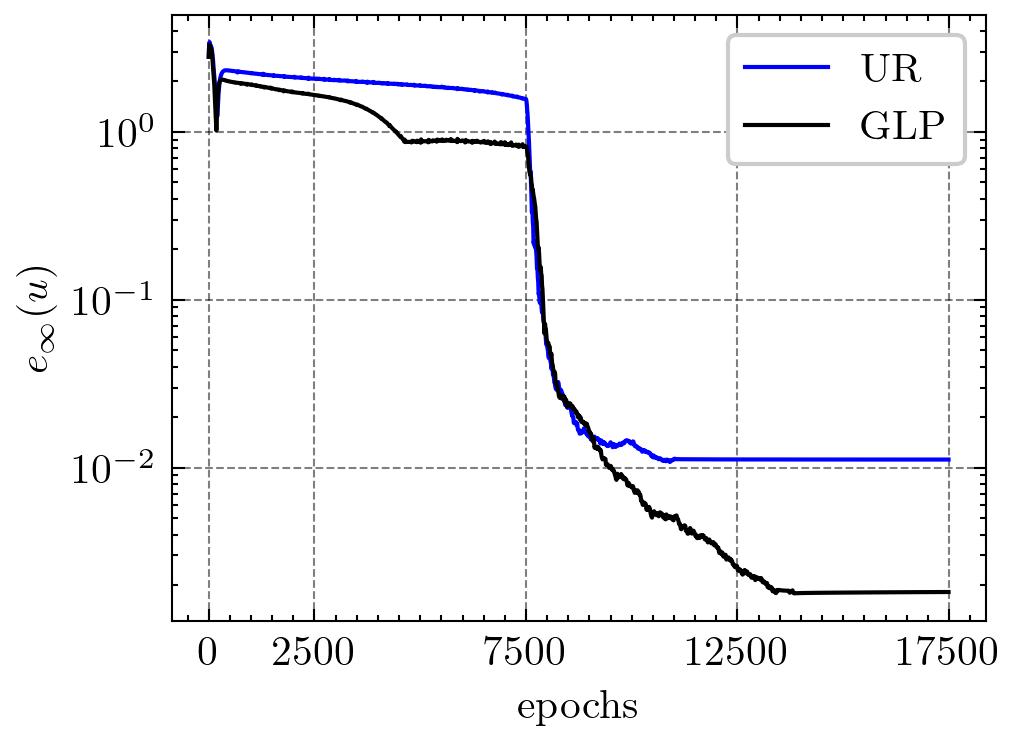}} \quad \quad 
\subfloat[performance of $e_2(u)$]{\includegraphics[width = 0.45\textwidth]
{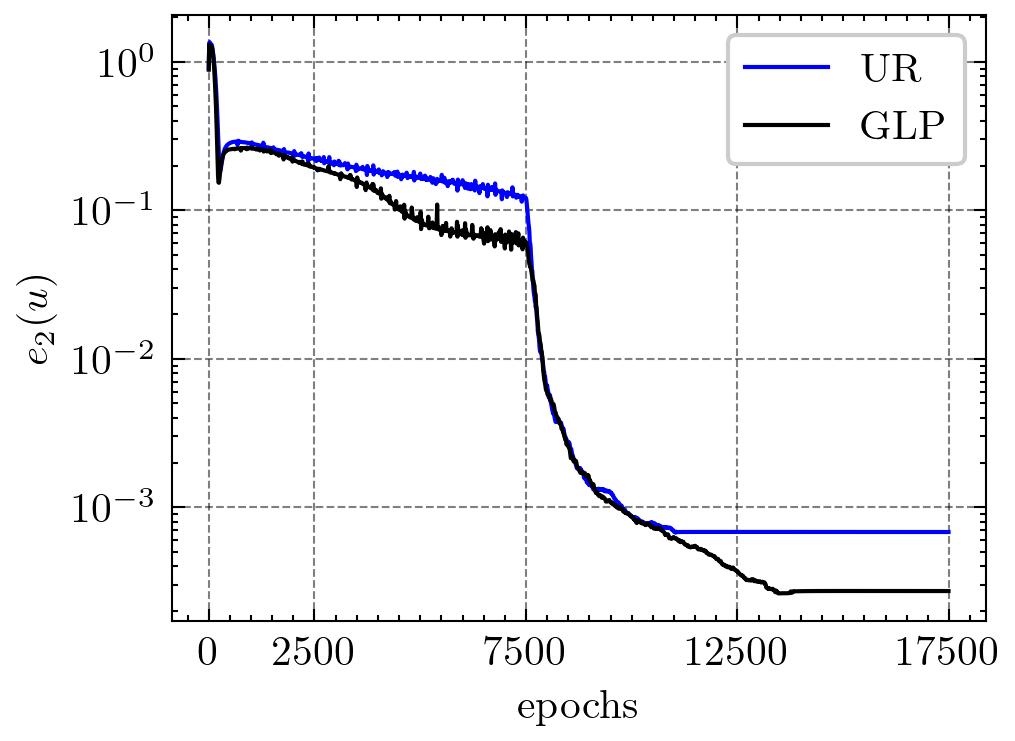}}
\caption{The performance of errors for the five-dimensional nonlinear problem. (a) the relative error $e_\infty(u)$ during training process; (b) the  relative error $e_2(u)$ during training process.}
\label{fig:NonlinearHd_ErrorEpochs}
\end{figure}

We also developed different training strategies by the number of training points and conducted controlled experiments with the GLP sampling.
The results of these different sampling strategies are comprehensively summarized in Table \ref{tab:NonlinearHd_PointsSettings}. The experimental results conclusively demonstrate that the GLP sampling possesses a clear advantage over uniform random sampling, even when the latter utilizes approximately five times the number of training points.

\begin{table}[h]
\scriptsize
\centering
\caption{
The number of residual points in $\Omega$ and the number of boundary points at each hyperplane on  $\partial \Omega$ using sampling strategies for the five-dimensional nonlinear problem.}
\setlength{\tabcolsep}{5.mm}{
\begin{tabular}{|c|c|c|c|c|}
\hline\noalign{\smallskip}
strategy   &  1 &  2 &3 & 4\\
\hline
method  & UR  & UR  & UR  & GLP\\
\hline
residual points  & 3001 & 5003 & 15019  & 3001\\
\hline
each hyperplane   & 100  & 100 & 100  &100\\
\hline
$\be_2(\bu)$& $ 6.813 \times 10^{-4}$ &$5.075\times 10^{-4}$ & $4.027\times 10^{-4}$ &  $2.728\times 10^{-4}$  \\
\hline
$\be_\infty(\bu)$& $1.118\times 10^{-2}$ &  $3.669\times 10^{-3}$&  $3.013\times 10^{-3}$& $1.818\times 10^{-3}$ \\
\hline
\end{tabular}
}
\label{tab:NonlinearHd_PointsSettings} 
\end{table}

Next, we take $d=8$ and $k=3$ in Eq \eqref{eq:NonlinearHd_solution}, and sample 11215 points in $\Omega$  and 200 points in each hyperplane on $\partial \Omega$ as the training set and 500000 points as the test set. Moreover, we adjust the weights in the loss function so that $\alpha_1= 10$  and $\alpha_2= 1$.  
The other settings of main parameters in PINNs are the same as before.
The distribution of the residual training points in the first two dimensions is plotted in Fig \ref{fig:Nonlinear8d_points}.


\begin{figure}[htbp]
\centering
\subfloat[Uniform random points]{\includegraphics[width = 0.45\textwidth]{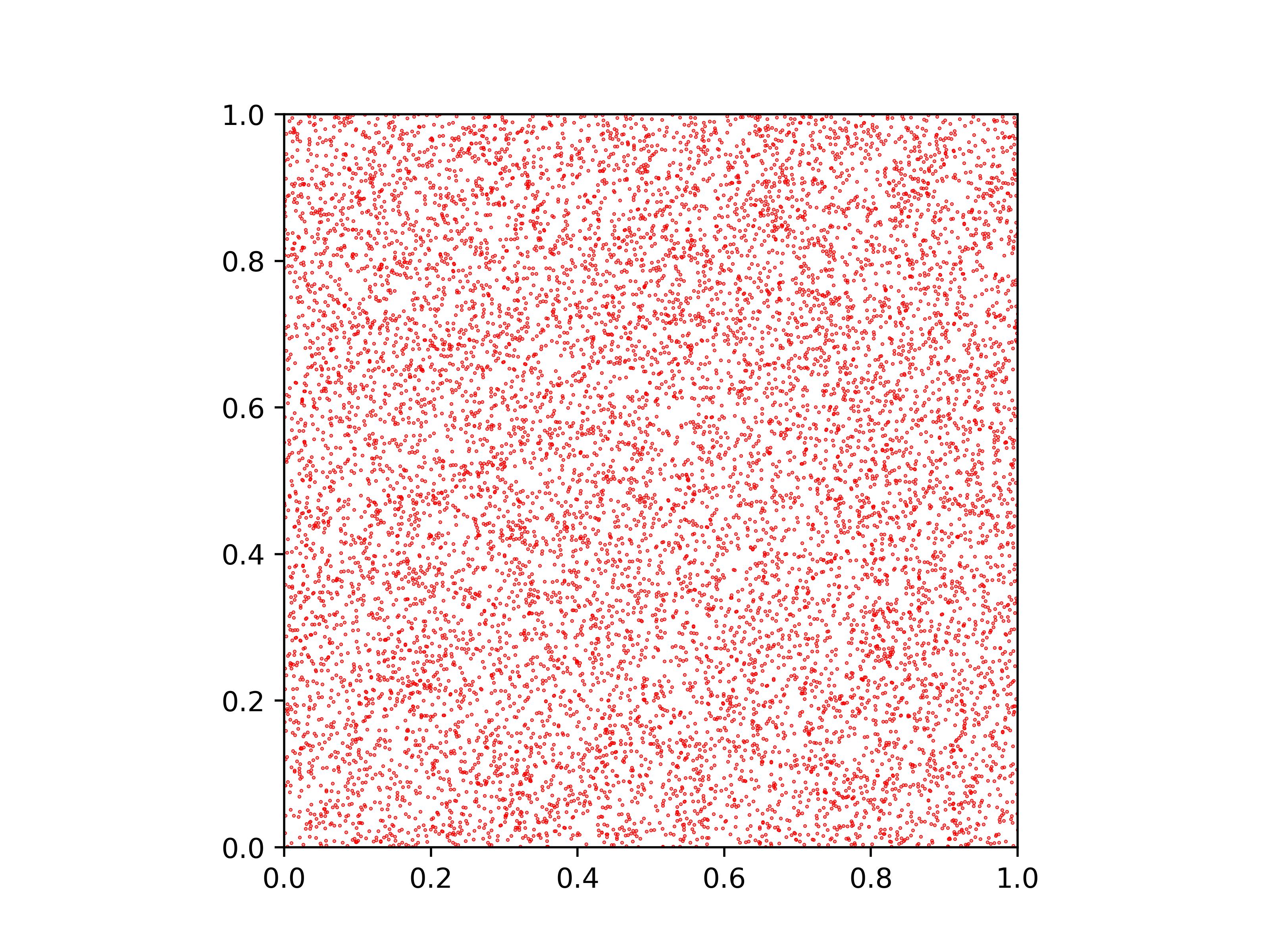}} \quad \quad 
\subfloat[GLP sampling points]{\includegraphics[width = 0.45\textwidth]
{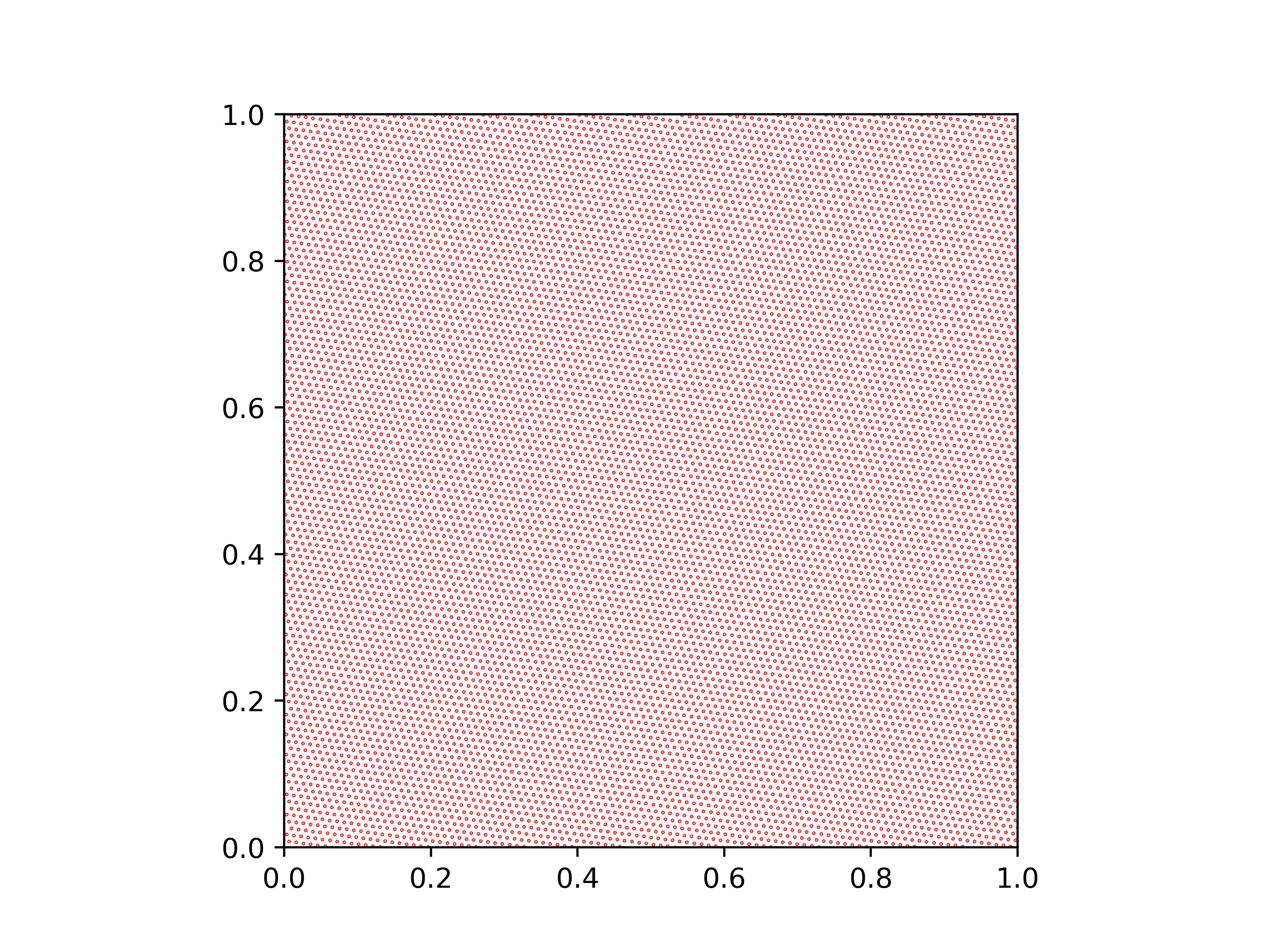}}
\caption{The performance of errors for the eight-dimensional nonlinear problem. (a) the residual training points by uniform random sampling; (b)  the residual training points of GLP sampling.}
\label{fig:Nonlinear8d_points}
\end{figure}

In Figure \ref{fig:Nonlinear8d_ErrorEpochs}, we compare the numerical results derived from GLP sampling with those from uniform random sampling after 20000 epochs of Adam training and 20000 epochs of LBFGS training. This comparison aims to show the effectiveness of our method for the same number of training points. We also developed different training strategies by the number of training points and conducted controlled experiments with the GLP sampling.
The results of these different sampling strategies are comprehensively summarized in Table \ref{tab:Nonlinear8d_PointsSettings}. The experimental results conclusively demonstrate that our method is effective.

\begin{figure}[htbp]
\centering
\subfloat[performance of $e_\infty(u)$]{\includegraphics[width = 0.45\textwidth]{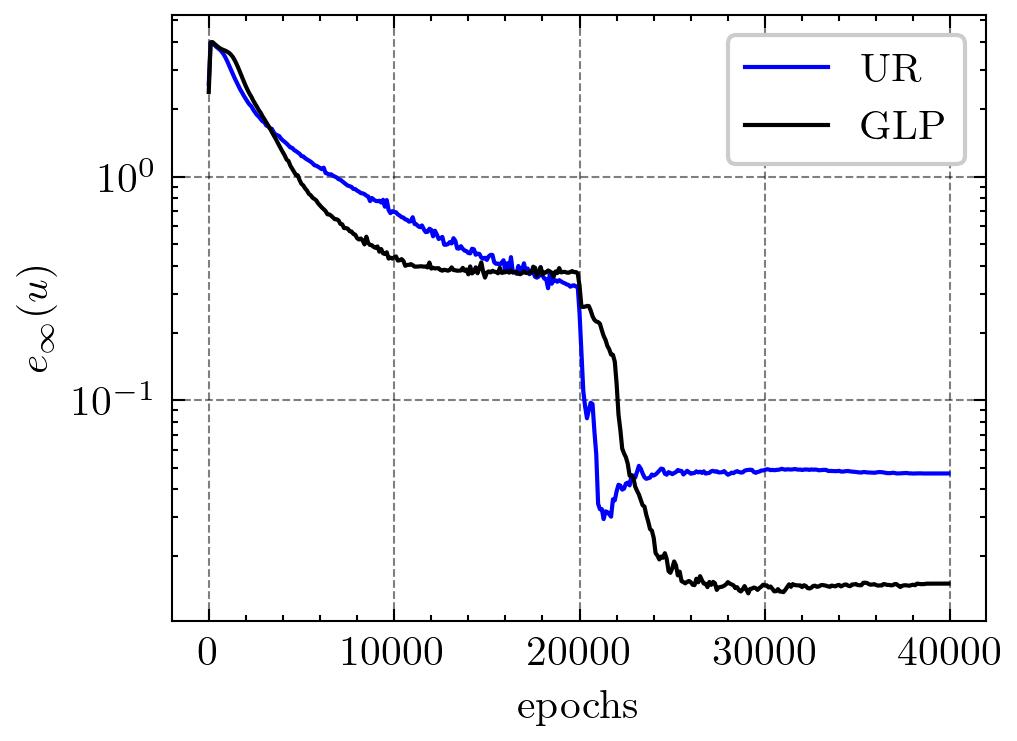}} \quad \quad 
\subfloat[performance of $e_2(u)$]{\includegraphics[width = 0.45\textwidth]
{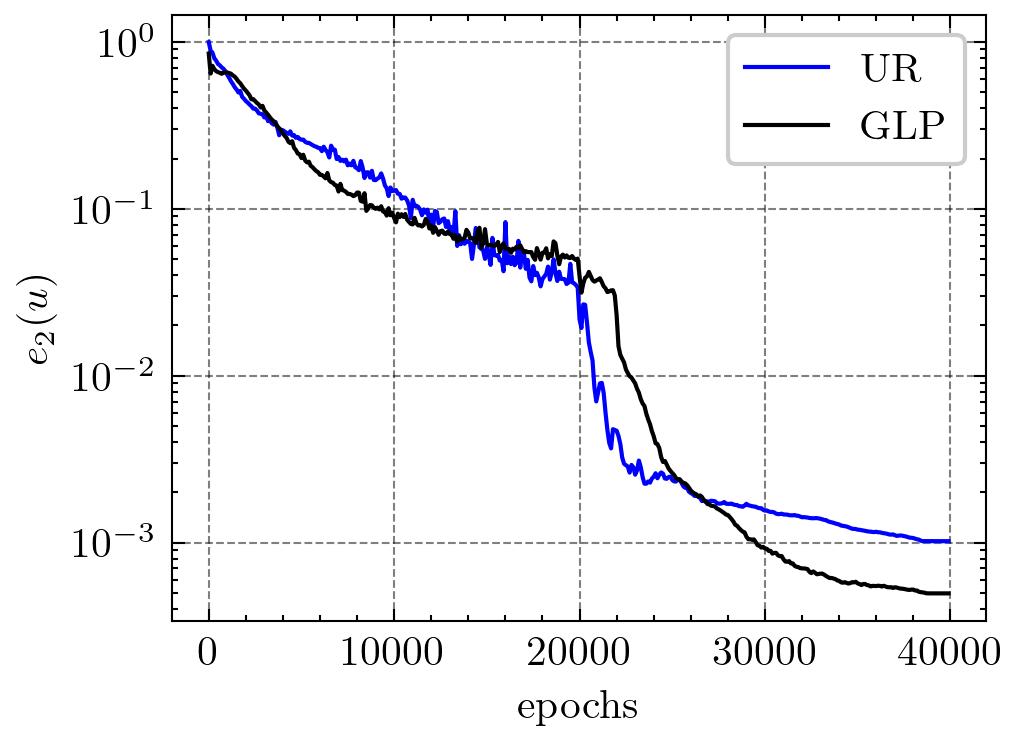}}
\caption{The performance of errors for the eight-dimensional nonlinear problem. (a) the relative error $e_\infty(u)$ during training process; (b) the  relative error $e_2(u)$ during training process.}
\label{fig:Nonlinear8d_ErrorEpochs}
\end{figure}

\begin{table}[h]
\scriptsize
\centering
\caption{
The number of residual points in $\Omega$ and the number of boundary points at each hyperplane on $\partial \Omega$ using sampling strategies for the eight-dimensional nonlinear problem.
}
\setlength{\tabcolsep}{5.mm}{
\begin{tabular}{|c|c|c|c|c|}
\hline\noalign{\smallskip}
strategy   &  1 &  2 &3 & 4\\
\hline
method  & UR  & UR  & UR  & GLP\\
\hline
residual points  & 11215 & 24041 & 46213  & 11215\\
\hline
each hyperplane   & 100  & 100 & 100  &100\\
\hline
$\be_2(\bu)$& $ 1.020 \times 10^{-2}$ &$8.209\times 10^{-3}$ & $5.680\times 10^{-3}$ &  $ 4.965\times 10^{-3}$  \\
\hline
$\be_\infty(\bu)$& $4.702\times 10^{-2}$ &  $4.330\times 10^{-2}$&  $ 4.251\times 10^{-2}$& $1.512\times 10^{-3}$ \\
\hline
\end{tabular}
}
\label{tab:Nonlinear8d_PointsSettings} 
\end{table}

\section{Conclusions}
\label{sec:conclusion}

In this work, we propose a number-theoretic method sampling neural network for solving PDEs, which using a good lattice point (GLP) set as sampling points to approximate the empirical residual loss function of physics-informed neural networks (PINNs), with the aim of reducing estimation errors. From a theoretical perspective, we give the upper bound of the error for PINNs based on the GLP sampling. Additionally, we demonstrate that when using GLP sampling, the upper bound of the expectation of the squared $L_2$ error of PINNs is smaller compared to using uniform random sampling. Numerical results based on our method, when addressing low-regularity and high-dimensional problems, indicate that GLP sampling outperforms traditional uniform random sampling in both accuracy and efficiency.
In the future, we will  work on how to use GLP sampling on other deep learning solvers and how to use number theoretic methods for non-uniform adaptive sampling.



\section*{Acknowledgment}

This research is partially sponsored by the National Key R \& D Program of China (No.2022YFE03040002) and the National Natural Science Foundation of China (No.12371434). 

Additionally, we would also like to express our gratitude to Professor KaiTai Fang and Dr. Ping He from BNU-HKBU United International College for their valuable discussions and  supports in this research.

\section*{Data Availability Statement}
The data that support the findings of this study are available from the corresponding author upon reasonable request.

\section*{Conflict of Interest}

The authors have no conflicts to disclose.




\section*{Appendix A}
\label{sec:AppendixA}

\renewcommand{\proofname}{\textbf{Proof of Theorem \ref{thm:1}}}
\begin{proof}

Using Eq \eqref{eq:descentlemma}, we have
\begin{equation}\label{eq:thm1_1}
         \mathcal{L}_{N}(\bx;\theta_{i+1})  \leq \mathcal{L}_{N}(\bx;\theta_{i}) + \nabla \mathcal{L}_{N}(\bx;\theta_{i})^T(\theta_{i+1}-\theta_{i} ) + \frac{1}{2}  \mathbf{L}  \Vert  \theta_{i+1}-\theta_{i} \Vert_{2}^2.
\end{equation}
According to the stochastic gradient descent method, 
\begin{equation}\label{eq:thm1_2}
        \theta_{i+1} = \theta_{i} - \eta g(x_i,\theta_i).
\end{equation}
Substitute Eq \eqref{eq:thm1_2} into Eq \eqref{eq:thm1_1} and find the expectation for $x_i$, we have
\begin{equation}\label{eq:thm1_3}
       \mathbb{E}_{x_i}\left(\mathcal{L}_{N}(\bx;\theta_{i+1})\right) \leq \mathcal{L}_{N}(\bx;\theta_{i}) - \eta \nabla \mathcal{L}_{N}(\bx;\theta_{i})^T \mathbb{E}_{x_i}\left(g(x_i;\theta_{i})\right) + \frac{\eta^2}{2}  \mathbf{L}   \mathbb{E}_{x_i} \left(\Vert g(x_i,\theta_i) \Vert_{2}^2 \right).
\end{equation}
According to Eq \eqref{eq:asumoment2} and \eqref{eq:asumoment3} in Assumption \ref{asu:stochastic_gradient}, it is obtained that
\begin{equation}\label{eq:thm1_4}
    \begin{aligned}
       \mathbb{E}_{x_i} \left(\Vert g(x_i,\theta_i) \Vert_{2}^2 \right) &=\mathbb{V}_{x_i} \left(g(x_i,\theta_i)\right) + \Vert\mathbb{E}_{x_i} \left( g(x_i,\theta_i)\right) \Vert_{2}^2 \\
       & \leq C_V s(N_r,N_b)+\left(M_V(1+s(N_r,N_b))+\mu^2_G\right) \Vert\nabla \mathcal{L}_{N}(\bx;\theta_{i})\Vert_{2}^2. 
    \end{aligned}
\end{equation}
Using  Eq \eqref{eq:asumoment1} and \eqref{eq:thm1_4}, Eq \eqref{eq:thm1_3} is rewritten as

\begin{equation}\label{eq:thm1_5}
    \begin{aligned}
       &\mathbb{E}_{x_i}\left(\mathcal{L}_{N}(\bx;\theta_{i+1})\right) -  \mathcal{L}_{N}(\bx;\theta_{i}) \leq   \frac{\eta^2}{2}  \mathbf{L}   \mathbb{E}_{x_i} \left(\Vert g(x_i,\theta_i) \Vert_{2}^2 \right) - \eta \nabla \mathcal{L}_{N}(\bx;\theta_{i})^T \mathbb{E}_{x_i}\left(g(x_i;\theta_{i})\right)  \\
       & \leq \frac{\eta^2}{2}  \mathbf{L}  C_V s(N_r,N_b)+\frac{\eta^2}{2}  \mathbf{L}\left(M_V(1+s(N_r,N_b))+\mu^2_G\right) \Vert\nabla \mathcal{L}_{N}(\bx;\theta_{i})\Vert_{2}^2 - \eta \mu \Vert \nabla \mathcal{L}_{N}(\bx;\theta_i) \Vert_2^2\\
       & = \frac{\eta^2}{2}  \mathbf{L}  C_V s(N_r,N_b) -(\eta\mu - \frac{\eta^2}{2}  \mathbf{L}\left(M_V(1+s(N_r,N_b))+\mu^2_G\right))\Vert\nabla \mathcal{L}_{N}(\bx;\theta_{i})\Vert_{2}^2.
    \end{aligned}
\end{equation}
Since $\eta$ satisfies Eq \eqref{eq:assumptioninThm1},
we have
\begin{equation}\label{eq:thm1_7}
    \begin{aligned}
       \mathbb{E}_{x_i}\left(\mathcal{L}_{N}(\bx;\theta_{i+1})\right)  &-  \mathcal{L}_{N}(\bx;\theta_{i}) \\
     &\leq \frac{\eta^2}{2}  \mathbf{L}  C_V s(N_r,N_b) -(\eta\mu - \frac{\eta^2}{2}  \mathbf{L}\left(M_V(1+s(N_r,N_b))+\mu^2_G\right))\Vert\nabla \mathcal{L}_{N}(\bx;\theta_{i})\Vert_{2}^2 \\
     & \leq\frac{\eta^2}{2}  \mathbf{L}  C_V s(N_r,N_b)  - \frac{\mu\eta}{2}\Vert\nabla \mathcal{L}_{N}(\bx;\theta_{i})\Vert_{2}^2.
    \end{aligned}
\end{equation}
\textcolor{black}{
By Assumption \ref{asu:PLcondition} and Lemma \ref{lem:PL*}, 
\begin{equation}\label{eq:thm1_8}
    \begin{aligned}
       \mathbb{E}_{x_i}\left(\mathcal{L}_{N}(\bx;\theta_{i+1})\right)  -  \mathcal{L}_{N}(\bx;\theta_{i}) &\leq\frac{\eta^2}{2}  \mathbf{L}  C_V s(N_r,N_b)  - \frac{\mu\eta}{2}\Vert\nabla \mathcal{L}_{N}(\bx;\theta_{i})\Vert_{2}^2\\
       & \leq \frac{\eta^2}{2}  \mathbf{L}  C_V s(N_r,N_b)  - c\mu\eta\mathcal{L}_{N}(\bx;\theta_{i}).
    \end{aligned}
\end{equation}
Adding $\mathcal{L}_{N}(\bx;\theta_{i})$ to both sides of Eq \eqref{eq:thm1_7} and taking total expectation for $\{x_1,x_2,...,x_i\}$ yields
\begin{equation}\label{eq:thm1_9}
    \begin{aligned}
       \mathbb{E}\left(\mathcal{L}_{N}(\bx;\theta_{i+1})\right)  
       \leq (1 - c\mu\eta)\mathbb{E}\left(\mathcal{L}_{N}(\bx;\theta_{i})\right) +\frac{\eta^2}{2}  \mathbf{L}  C_V s(N_r,N_b). 
    \end{aligned}
\end{equation}
Then we obtain  a recursive inequality,
\begin{equation}\label{eq:thm1_10}
    \begin{aligned}
       \mathbb{E}\left(\mathcal{L}_{N}(\bx;\theta_{i+1})\right) -\frac{\eta\mathbf{L}C_V}{2c\mu} s(N_r,N_b) 
       &\leq (1 - c\mu\eta) \left(\mathbb{E}\left(\mathcal{L}_{N}(\bx;\theta_{i})\right)- \frac{\eta\mathbf{L}C_V}{2c\mu} s(N_r,N_b) \right)\\
       & \leq (1 - c\mu\eta)^2 \left(\mathbb{E}\left(\mathcal{L}_{N}(\bx;\theta_{i-1})\right)- \frac{\eta\mathbf{L}C_V}{2c\mu} s(N_r,N_b) \right) \\
       &... \\
       & \leq (1 - c\mu\eta)^{i} \left(\mathbb{E}\left(\mathcal{L}_{N}(\bx;\theta_{1})\right)- \frac{\eta\mathbf{L}C_V}{2c\mu} s(N_r,N_b) \right).
    \end{aligned}
\end{equation}
Since $0<c\mu \eta \leq \frac{c\mu^2}{ \mathbf{L}\left(M_V+\mu^2_G(1+s(N_r,N_b))\right)} < \frac{c\mu^2}{ \mathbf{L}\mu^2_G} \leq 1,$
 we have the following conclusion
\begin{equation}\label{eq:thm1_11}
    \lim_{i \to \infty} \mathbb{E}\left(\mathcal{L}_{N}(\bx;\theta_i)\right) = \frac{\eta \mathbf{L}C_V}{2c\mu} s(N_r,N_b).
\end{equation}}

\end{proof}

\renewcommand{\proofname}{\textbf{Proof of Theorem \ref{thm:2}}}
\begin{proof}
According to the proof of Theorem \ref{thm:1},  we recall Eq \eqref{eq:thm1_9},
\begin{equation}\label{eq:thm3_1}
\begin{aligned}
   \mathbb{E}\left(\mathcal{L}_{N}(\bx;\theta_{i+1})\right)  
   \leq (1 - c\mu\eta_i)\mathbb{E}\left(\mathcal{L}_{N}(\bx;\theta_{i})\right) +\frac{\eta_i^2}{2}  \mathbf{L}  C_V s(N_r,N_b). 
\end{aligned}
\end{equation}

Now we will prove this theorem by induction. When $i=1$, due to the definition of $\kappa$, it is easy to prove that $   \mathbb{E}\left(\mathcal{L}_{N}(\bx;\theta_1)\right) \leq \frac{\kappa }{\xi+1} s(N_r,N_b)$. Assume that when $i=k >1$, Eq \eqref{eq:optim_gap_diminishingstep} still holds. Next, we consider the case when $i=k+1$.

According to Eq \eqref{eq:thm3_1}, we have
\begin{equation}\label{eq:thm3_2}
    \begin{aligned}
       \mathbb{E}\left(\mathcal{L}_{N}(\bx;\theta_{k+1})\right) 
       &\leq (1 - c\mu\eta_k) \frac{\kappa}{\xi+k}s(N_r,N_b)+
       \frac{1}{2}\eta^2_k \mathbf{L} C_V s(N_r,N_b) \\
       & = (1-\frac{\beta c \mu}{\xi + k})\frac{\kappa}{\xi+k}s(N_r,N_b) + \frac{\beta^2 \mathbf{L}C_Vs(N_r,N_b)}{2(\xi+k)^2} \\
       & =  \frac{\xi+k-\beta c \mu}{(\xi+k)^2} \kappa s(N_r,N_b) + \frac{\beta^2 \mathbf{L}C_Vs(N_r,N_b)}{2(\xi+k)^2} \\
        & = \frac{\xi+k-1}{(\xi+k)^2} \kappa s(N_r,N_b) -
        \frac{\beta c \mu-1}{(\xi+k)^2} \kappa s(N_r,N_b)
        + \frac{\beta^2 \mathbf{L}C_Vs(N_r,N_b)}{2(\xi+k)^2}.
    \end{aligned}
\end{equation}

Due to the definition of $\kappa$,  it is observed that $\kappa \geq \frac{\beta^2 \mathbf{L}C_V}{2(\beta c \mu -1)}$, then we have
\begin{equation}\label{eq:thm3_3}
\frac{\beta^2 \mathbf{L}C_Vs(N_r,N_b)}{2(\xi+k)^2} -\frac{\beta c \mu-1}{(\xi+k)^2} \kappa s(N_r,N_b)
         \leq 0.
\end{equation}

Using Eq \eqref{eq:thm3_2}, there is
\begin{equation}\label{eq:thm3_4}
    \begin{aligned}
       \mathbb{E}\left(\mathcal{L}_{N}(\bx;\theta_{k+1})\right) 
        & \leq \frac{\xi+k-1}{(\xi+k)^2} \kappa s(N_r,N_b) -
        \frac{\beta c \mu-1}{(\xi+k)^2} \kappa s(N_r,N_b)
        + \frac{\beta^2 \mathbf{L}C_Vs(N_r,N_b)}{2(\xi+k)^2}\\
        & \leq \frac{\xi+k-1}{(\xi+k)^2} \kappa s(N_r,N_b)\\
        & \leq \frac{\xi+k-1}{(\xi+k)^2-1} \kappa s(N_r,N_b)\\
        & = \frac{\kappa}{\xi+k+1}  s(N_r,N_b).
    \end{aligned}
\end{equation}

By induction, we have proved that Eq \eqref{eq:optim_gap_diminishingstep} holds.

\end{proof}

\renewcommand{\proofname}{\textbf{Proof of Theorem \ref{thm:3}}}
\begin{proof}

By  Assumption \ref{asu:Normrelations}, we have
\begin{equation}
     \begin{aligned}
     C_1 \Vert \bu^*(\bx)-\bu(\bx;\theta) \Vert_{2}
     &\leq \Vert \mathcal{A}[\bu^*(\bx)] -\mathcal{A}[\bu(\bx;\theta)] \Vert_{2}  + \Vert \mathcal{B}[\bu^*(\bx)]-\mathcal{B}[\bu(\bx;\theta)] \Vert_{2} \\
     &= \Vert \mathcal{A}[\bu(\bx;\theta)] - f(\bx) \Vert_{2} + \Vert \mathcal{B}[\bu(\bx;\theta)] - g(\bx) \Vert_{2} \\
     &= \Vert r(\bx;\theta) \Vert_{2} + \Vert b(\bx;\theta)\Vert_{2}\\
     &\leq  \sqrt{2}  \left(\Vert r(\bx;\theta) \Vert_{2}^2 + 
     \Vert b(\bx;\theta)\Vert_{2}^2\right)^{\frac{1}{2}}.
     \end{aligned}     
\end{equation}

Using Lemma \ref{lem:MC_Discrepancy}, we have

\begin{equation}
    \begin{aligned}
     \mathcal{L}_{r}(\bx;\theta) &= \Vert r(\bx;\theta) \Vert_{2}^2 \leq  \mathcal{L}_{r,N_r}^{GLP}(\bx;\theta) +\bO\left(\frac { (\log N_r)^{d} } {{N_r}}\right), \\
     \mathcal{L}_{b}(\bx;\theta) &= \Vert b(\bx;\theta) \Vert_{2}^2
     \leq \mathcal{L}_{b,N_b}^{UR}(\bx;\theta)  +\bO\left(N_b^{-\frac{1}{2}}\right).
     \end{aligned} 
\end{equation}

Therefore,
\begin{equation}
    \begin{aligned}
     \Vert \bu^*(\bx)-\bu(\bx;\theta) \Vert_{2} &\leq  \frac{\sqrt{2}}{C_1}(\Vert r(\bx;\theta) \Vert_{2}^2 + 
     \Vert b(\bx;\theta)\Vert_{2}^2)^{\frac{1}{2}}\\
     & \leq \frac{\sqrt{2}}{C_1} \left(\mathcal{L}_{r,N_r}^{GLP} + \mathcal{L}_{b,N_b}^{UR}
      + \bO\left(\frac { (\log N_r)^{d} } {{N_r}}\right) +\bO\left(N_b^{-\frac{1}{2}}\right) \right)^{\frac{1}{2}}.
     \end{aligned} 
\end{equation}
 Eq \eqref{eq:errorestimate_UR} can be obtained similarly.
\end{proof}

\renewcommand{\proofname}{\textbf{Proof of Corollary \ref{cor:1}}}
\begin{proof}
By Theorem \ref{thm:1},
\begin{equation}\label{eq:cor1_1}
    \lim_{i \to \infty} \mathbb{E}\left(\mathcal{L}_{N}(\bx;\theta_i)\right) = \frac{\eta \mathbf{L}C_V}{2c\mu} s(N_r,N_b).
\end{equation}
Thus, if $\left\{ \bx_i \right\}_{i=1}^{N_r}$ is a good lattice point set, we have
\begin{equation}\label{eq:cor1_2}
    \begin{aligned}
    \lim_{i \to \infty} \mathbb{E}\left(\mathcal{L}_{r,N_r}^{GLP}(\bx;\theta_i) + \mathcal{L}_{b,N_b}^{UR}(\bx;\theta_i)\right) &= \frac{\eta \mathbf{L}C_V}{2c\mu} \left( \bO\left(\frac { (\log N_r)^{d} } {{N_r}}\right) +\bO\left(N_b^{-\frac{1}{2}}\right)\right)^2. 
    \end{aligned}
\end{equation}
By Theorem \ref{thm:3},
\begin{equation}\label{eq:cor1_3}
    \begin{aligned}
     \Vert \bu^*(\bx)-\bu(\bx;\theta) \Vert_{2}  \leq \frac{\sqrt{2}}{C_1} \left(\mathcal{L}_{r,N_r}^{GLP} + \mathcal{L}_{b,N_b}^{UR}
      + \bO\left(\frac { (\log N_r)^{d} } {{N_r}}\right) +\bO\left(N_b^{-\frac{1}{2}}\right) \right)^{\frac{1}{2}}.
     \end{aligned} 
\end{equation}
Taking squaring of the inequality and calculating the expectation yields
\begin{equation}\label{eq:cor1_4}
    \begin{aligned}
     \mathbb{E} \left(\Vert \bu^*(\bx)-\bu(\bx;\theta_i) \Vert_{2}^2\right)  \leq \frac{2}{{C_1}^2} \left( \mathbb{E} \left(\mathcal{L}_{r,N_r}^{GLP} + \mathcal{L}_{b,N_b}^{UR}\right)
      + \bO\left(\frac { (\log N_r)^{d} } {{N_r}}\right) +\bO\left(N_b^{-\frac{1}{2}}\right) \right).
     \end{aligned} 
\end{equation}
Combining Eq \eqref{eq:cor1_2} and \eqref{eq:cor1_4}, we have
    \begin{equation}\label{eq:cor1_5}
    \begin{aligned}
      \lim_{i \to \infty}\mathbb{E}\left(\Vert \bu^*(\bx)-\bu(\bx;\theta_i) \Vert_{2}^2\right)
     \leq \frac{2}{{C_1}^2} \left(\frac{\eta \mathbf{L}C_V}{2c\mu} \left( \bO\left(\frac { (\log N_r)^{d} } {{N_r}}\right) +\bO\left(N_b^{-\frac{1}{2}}\right)\right)^2 + \bO\left(\frac { (\log N_r)^{d} } {{N_r}}\right) +\bO\left(N_b^{-\frac{1}{2}}\right) \right) .
     \end{aligned} 
\end{equation}
 Eq \eqref{eq:errorestimate_cor_UR} can be proved similarly.
\end{proof}

\vfill

\bibliographystyle{cas-model2-names}
\bibliography{NTM-PINN}

\end{document}